\documentclass[pdftex,11pt,a4paper]{amsart}

\usepackage{amsmath,amsxtra,amssymb,latexsym, amscd,amsthm,young}
\usepackage{graphicx}

\usepackage{color}

\usepackage{colortbl}
\usepackage{graphicx}
\usepackage[all]{xy}

\usepackage{pgf,tikz}
\usetikzlibrary{arrows}

\begin{document}

\title{$\mathrm{Ext}$-groups in the category of strict polynomial functors}

\author{Van Tuan Pham}

\newtheorem{thm}{Theorem}[subsection]
\newtheorem{lem}[thm]{Lemma}
\newtheorem{pro}[thm]{Proposition}
\newtheorem{cor}[thm]{Corollary}
\newtheorem{que}[thm]{Question}
\newtheorem*{kd}{Affirmation}
\newtheorem*{thms}{Theorem}
\newtheorem*{pros}{Proposition}
\newtheorem*{bai}{Problem}
\newtheorem*{deffs}{Definition}
\newtheorem*{cors}{Corollary}
\newtheorem{prob}[thm]{Problem}
\newtheorem{conj}[thm]{Conjecture}

\theoremstyle{definition}
\newtheorem{deff}[thm]{Definition}
\newtheorem{kh}[thm]{Notation}
\newtheorem{exe}[thm]{Example}
\newtheorem{vid}[thm]{}
\newtheorem{cauhoi}[thm]{Question}
\newtheorem{nota}[thm]{Notation}
\newtheorem{conv}[thm]{Convention}
\newtheorem*{remark}{Remark(T)}

\newtheorem{rem}[thm]{Remark}


\newcommand{\Fr}{\operatorname{Fr}}
\newcommand{\Hom}{\operatorname{Hom}}
\newcommand{\Ext}{\operatorname{Ext}}
\newcommand{\GL}{\operatorname{GL}}
\newcommand{\Ooper}{\operatorname{O}}
\newcommand{\Uoper}{\operatorname{U}}

\newcommand{\Ch}{\mathbf{Ch}}

\newcommand{\Ab}{\operatorname{Ab}}
\newcommand{\Ev}{\operatorname{-ev}}

\newcommand{\Boverline}{\overline{B}}

\newcommand{\Model}{\mathsf{M}_\Bbbk}

\newcommand{\Pcoeur}{\mathbf{P}_{p\operatorname{-}\mathfrak{coeur}}}

\newcommand{\ktto}[1]{\mathrm{k}\left(#1\right)}

\newcommand{\Mod}{\operatorname{-Mod}}
\newcommand{\Modr}{\operatorname{Mod-}}
\newcommand{\Gal}{\operatorname{Gal}}
\newcommand{\Set}{\operatorname{Set}}
\newcommand{\Colim}{\operatorname{colim}}
\newcommand{\Hocolim}{\operatorname{hocolim}}
\newcommand{\Holim}{\operatorname{holim}}
\newcommand{\Tor}{\operatorname{Tor}}
\newcommand{\Lie}{\operatorname{Lie}}
\newcommand{\Tot}{\operatorname{Tot}}
\newcommand{\Diag}{\operatorname{diag}}
\newcommand{\Res}{\operatorname{Res}}
\newcommand{\Ind}{\operatorname{Ind}}

\newcommand{\Formel}{\mathfrak{Fc}}

\newcommand{\SP}{\operatorname{SP}}
\newcommand{\Sq}{\operatorname{Sq}}
\newcommand{\Trace}{\operatorname{Tr}}
\newcommand{\Spoper}{\operatorname{Sp}}
\newcommand{\injdim}{\operatorname{inj.dim}}

\newcommand{\Core}{\mathfrak{Co}}
\newcommand{\Corepx}[1]{\mathfrak{Co}_{p}\left(#1\right)}
\newcommand{\Kfoncteur}{\Bbbk\operatorname{-}\mathsf{foncteurs}}


\newcommand{\Comod}{\operatorname{-Comod}}
\newcommand{\Comodr}{\operatorname{Comod-}}

\newcommand{\Aut}{\operatorname{Aut}}
\newcommand{\Algcom}{\operatorname{-alg.com}}
\newcommand{\Alg}{\operatorname{-Alg}}
\newcommand{\Grp}{\operatorname{Grp}}
\newcommand{\Cr}{\operatorname{cr}}
\newcommand{\Ht}{\operatorname{h}}
\newcommand{\Id}{\operatorname{Id}}
\newcommand{\Pol}{\operatorname{Pol}}
\newcommand{\Sch}{\operatorname{Sch}}
\newcommand{\Spec}{\operatorname{Sp}_\Bbbk}
\newcommand{\End}{\operatorname{End}}
\newcommand{\Eval}{\operatorname{ev}}
\newcommand{\Rep}{\operatorname{Rep}}
\newcommand{\Ob}{\operatorname{Ob}}
\newcommand{\Ens}{\operatorname{Ens}}

\newcommand{\Pbf}{\mathbf{P}}
\newcommand{\Triv}{\textrm{triv}}

\newcommand{\wunderline}{\mathbf{w}}

\newcommand{\Op}{\mathrm{op}}


\newcommand{\EndK}[1]{\End_\Bbbk\left(#1\right)}
\newcommand{\EndGV}[2]{\End_{\Gamma^{#1}\mathcal{V}_\Bbbk}\left(#2\right) }

\newcommand{\Evaln}{\operatorname{ev}_{\Bbbk^n}}


\newcommand{\Ocal}{\mathcal{O}}
\newcommand{\Acal}{\mathcal{A}}
\newcommand{\Bcal}{\mathcal{B}}
\newcommand{\Ccal}{\mathcal{C}}
\newcommand{\Dcal}{\mathcal{D}}
\newcommand{\Ecal}{\mathcal{E}}
\newcommand{\FcalK}{\mathcal{F}_{\Bbbk}}
\newcommand{\Ucal}{\mathcal{U}}

\newcommand{\PcalK}{\mathcal{P}_\Bbbk}
\newcommand{\PcalKn}{\mathcal{P}_\Bbbk(n)}

\newcommand{\PcalKoper}[1]{\mathcal{P}_{#1}}
\newcommand{\Pcalx}[1]{\mathcal{P}_{#1}}

\newcommand{\PcalKttilde}{\widetilde{\mathcal{P}_\Bbbk}}

\newcommand{\Kcal}{\mathcal{K}}
\newcommand{\Ncal}{\mathcal{N}}
\newcommand{\Scal}{\mathcal{S}}
\newcommand{\Ical}{\mathcal{I}}
\newcommand{\Jcal}{\mathcal{J}}
\newcommand{\VcalK}{\mathcal{V}_\Bbbk}
\newcommand{\VcalPol}{\mathcal{V}_{\Pol,\Bbbk}}
\newcommand{\PolKoper}[1]{\Pol_\Bbbk\left(#1\right)}

\newcommand{\Vcaln}{\mathcal{V}_\Bbbk^{\times n}}
\newcommand{\VcalKe}{\mathcal{V}_\Bbbk^*}
\newcommand{\VcalKgr}{\mathcal{V}_\Bbbk^\mathrm{gr}}

\newcommand{\Lcal}{\mathcal{L}}
\newcommand{\SPcal}{\mathcal{SP}}

\newcommand{\Mcal}{\mathcal{M}}


\newcommand{\Asf}{\mathsf{A}}
\newcommand{\Bsf}{\mathsf{B}}

\newcommand{\Csf}{\mathsf{C}}
\newcommand{\Lsf}{\mathsf{L}}
\newcommand{\Qsf}{\mathsf{Q}}
\newcommand{\Psf}{\mathbf{P}}
\newcommand{\Isf}{\mathsf{I}}
\newcommand{\Ssf}{\mathsf{S}}
\newcommand{\esf}{\mathsf{e}}
\newcommand{\Ksf}{\mathsf{K}}
\newcommand{\Stsf}{\mathsf{St}}
\newcommand{\Psfpcore}{\mathbf{P}_{p\text{-}\mathfrak{coeur}}}
\newcommand{\Psfhcore}{\mathbf{P}_{2\text{-}\mathfrak{coeur}}}

\newcommand{\IPsf}{\mathbf{I\hskip -0.01cm P}}

\newcommand{\Dbf}{\mathbf{D}}
\newcommand{\Rbf}{\mathbf{R}}
\newcommand{\Lbf}{\mathbf{L}}
\newcommand{\Kbf}{\mathbf{K}}


\newcommand{\Abb}{\mathbb{A}}

\newcommand{\Zbb}{\mathbb{Z}}
\newcommand{\Nbb}{\mathbb{N}}
\newcommand{\Rbb}{\mathbb{R}}
\newcommand{\Qbb}{\mathbb{Q}}
\newcommand{\Sbb}{\mathbb{S}}
\newcommand{\Fbbp}{\mathbb{F}_p}
\newcommand{\Fbbpbar}{\overline{\mathbb{F}}_p}

\newcommand{\Soc}{\operatorname{soc}}


\newcommand{\Sfrak}{\mathfrak{S}}
\newcommand{\Gfrak}{\mathfrak{G}}
\newcommand{\Hfrak}{\mathfrak{H}}
\newcommand{\Bfrak}{\mathfrak{B}}
\newcommand{\gfrak}{\mathfrak{g}}
\newcommand{\hfrak}{\mathfrak{h}}


\newcommand{\dhr}{\partial}

\newcommand{\td}[1]{\left|#1\right|}

\newcommand{\Bloc}[1]{\mathfrak{Bl}\left(#1\right)}

\newcommand{\Ceil}[1]{\left\lceil #1\right\rceil}

\newcommand{\Ngoac}[1]{\left(#1\right)}
\newcommand{\Ngoacv}[1]{\left[#1\right]}
\newcommand{\Ngoacn}[1]{\left\{#1\right\}}

\newcommand{\Pist}{\pi^{\rm st}}

\newcommand{\Lst}{L^{\rm st}}
\newcommand{\SPinfty}{\SP^\infty}

\newcommand{\vh}[1]{\left\langle #1\right\rangle}

\newcommand{\HomVcalPol}[1]{\Hom_{\VcalPol}\left(#1\right)}
\newcommand{\HominP}[1]{\mathbf{\Hom}_{\Pcal_\Bbbk}\left(#1\right)}
\newcommand{\HominPn}[1]{\mathbf{\Hom}_{\Pcal_\Bbbk(n)}\left(#1\right)}

\newcommand{\Chd}[1]{\mathbf{Ch}_{\ge 0}\left(#1\right)}
\newcommand{\Hrat}[2]{H_{\mathrm{rat}}^{#1}\left(#2\right)}
\newcommand{\HomPoper}[2]{\Hom_{\mathcal{P}_{#1,\Bbbk}}\left(#2\right)}

\newcommand{\Endx}[2]{\operatorname{End}_{#1}\left(#2\right)}

\newcommand{\Homx}[2]{\Hom_{#1}\left(#2\right)}
\newcommand{\Homxin}[2]{\mathbf{Hom}_{#1}\left(#2\right)}
\newcommand{\RHomx}[2]{\Rbf\Hom_{#1}\left(#2\right)}
\newcommand{\RHomxin}[2]{\Rbf\mathbf{Hom}_{#1}\left(#2\right)}

\newcommand{\Extx}[2]{\Ext^{*}_{#1}\left(#2\right)}
\newcommand{\Extxin}[2]{\mathbf{Ext}^{*}_{#1}\left(#2\right)}

\newcommand{\pr}{\operatorname{pr}}

\newcommand{\HomPn}[1]{\Hom_{\Pcal_\Bbbk(n)}\left(#1\right)}
\newcommand{\HomPh}[1]{\Hom_{\Pcal_\Bbbk(2)}\left(#1\right)}

\newcommand{\RHomiP}[1]{\mathbf{R}\mathbf{Hom}_{\Pcal_\Bbbk}\left(#1\right)}
\newcommand{\RHomiPn}[1]{\mathbf{R}\mathbf{Hom}_{\Pcal(n)_\Bbbk}\left(#1\right)}

\newcommand{\RHomPn}[1]{\mathbf{R}\Hom_{\Pcal(n)_\Bbbk}\left(#1\right)}
\newcommand{\RHomP}[1]{\mathbf{R}\Hom_{\Pcal_\Bbbk}\left(#1\right)}

\newcommand{\OPcalR}{\odot_\Pcal^\Rbf}
\newcommand{\Enrbar}{\overline{E}_{n,r}}
\newcommand{\Fct}[1]{\mathrm{Fct}\left(#1\right)}
\newcommand{\FctK}[1]{\mathrm{Fct}_\Bbbk\left(#1\right)}

\newcommand{\Rhomx}[2]{\mathbf{R}\Hom_{#1}\left(#2\right)}

\newcommand{\FcalKoper}[1]{\mathcal{F}_{#1,\Bbbk}}

\newcommand{\HomK}[1]{\Hom\left(#1\right)}

\newcommand{\SchK}{\Sch/\Bbbk}

\newcommand{\BbbkEv}{\operatorname{Vect}_\Bbbk}

\newcommand{\PolKoperdegree}[2]{\Pol_{#1,\Bbbk}\left(#2\right)}

\newcommand{\GammadVK}{\Gamma^d\mathcal{V}_\Bbbk}

\newcommand{\SKnd}{S_{\Bbbk}(d,n)}
\newcommand{\SKoper}[1]{S_\Bbbk\left(#1\right)}
\newcommand{\SpecKoper}[1]{\Spec_\Bbbk\left(#1\right)}
\newcommand{\BbbkTriv}{\Bbbk^\mathrm{triv}}
\newcommand{\BbbkSgn}{\Bbbk^\mathrm{sgn}}

\newcommand{\Ftilde}{\widetilde{F}}
\newcommand{\GLnK}{\operatorname{GL}_{n,\Bbbk}}
\newcommand{\SpnK}{\operatorname{Sp}_{n,\Bbbk}}
\newcommand{\OnnK}{\operatorname{O}_{n,n;\Bbbk}}
\newcommand{\lflooroper}[1]{\left\lfloor#1\right\rfloor}
\newcommand{\odotRbf}{\overset{\Rbf}{\odot}}

\newcommand{\Vunderline}{\mathbf{V}}
\newcommand{\Wunderline}{\mathbf{W}}

\newcommand{\otimesx}[1]{\otimes_{#1}}
\newcommand{\otimesxL}[1]{\otimes^\Lbf_{#1}}
\newcommand{\dunder}{\mathbf{d}}
\newcommand{\Sx}[2]{S^{\boxtimes \,#1\,}_{#2}}
\newcommand{\Gx}[2]{\Gamma^{\boxtimes\,#1\,#2}}
\newcommand{\ktt}{\mathrm{k}}
\newcommand{\boxtimesx}[2]{\underset{#1}{\overset{#2}{\text{\Large $\boxtimes$}}}}

\newcommand{\Dbfb}[1]{\mathbf{D}^b\left(#1\right)}
\newcommand{\Extxx}[3]{\operatorname{Ext}^{#1}_{#2}\left(#3\right)}
\newcommand{\Extxxin}[3]{\mathbf{Ext}^{#1}_{#2}\left(#3\right)}
\newcommand{\munderline}{\mathbf{m}}
\newcommand{\Otimesx}[1]{\overset{#1}{\underset{i=1}{\bigotimes}}}

\newcommand{\Modp}{\operatorname{-mod}}
\newcommand{\PcalKh}{\mathcal{P}_\Bbbk(2)}
\newcommand{\Gcal}{\mathcal{G}}
\newcommand{\Itype}{\rm (\mathfrak{I})}
\newcommand{\Zph}{\mathbb{Z}_{>0}^{\times 2}}
\newcommand{\I}{\operatorname{I}}
\newcommand{\II}{\operatorname{I\hskip -0.05cm I}}
\newcommand{\Fctr}[1]{\mathrm{Fct}_*\left(#1\right)}
\newcommand{\psf}{\mathsf{p}}
\newcommand{\dsf}{\mathsf{d}}
\newcommand{\Tw}{\operatorname{t}}
\newcommand{\chuan}[1]{\left\|#1\right\|}

\newcommand{\cunder}{\mathbf{c}}
\newcommand{\Rat}{\mathcal{R}\mathrm{at}}
\newcommand{\Poly}{\mathcal{P}\mathrm{ol}}
\newcommand{\Cbb}{\mathbb{C}}
\newcommand{\Cfrak}{\mathfrak{C}}
\newcommand{\Kbfd}[1]{\mathbf{K}_{\ge 0}\left(#1\right)}
\newcommand{\tunder}{\mathbf{t}}
\newcommand{\Wcal}{\mathcal{W}}
\newcommand{\Urm}{\mathrm{U}}
\newcommand{\Pbftilde}{\widetilde{\Pbf}}
\newcommand{\Jtilde}{\widetilde{J}}
\newcommand{\Dbfd}[1]{\mathbf{D}_{\ge 0}\left(#1\right)}
\newcommand{\eunder}{\mathbf{e}}
\newcommand{\Kbb}{\mathbb{K}}
\newcommand{\KbbEv}{\operatorname{Vect}_\mathbb{K}}
\newcommand{\KbbV}{\mathcal{V}_\mathbb{K}}
\newcommand{\Lbb}{\mathbb{L}}

\newcommand{\ket}{\hfill$\square$}

\begin{abstract}
	The aim of this paper is to study, by using the mathematical tools developed by Cha\l upnik, Touz\'e and Van der Kallen, the effect of the Frobenius twist on $\Ext$-group in the category of strict polynomial functors.
	As an application, we obtain explicit formulas of cohomology of the orthogonal groups and symplectic ones.
\end{abstract}

\footnote{VNU University of science, phamvantuan1987@gmail.com}

\maketitle

\section*{Introduction}

Let $ \Bbbk $ be a field of positive characteristic $ p $.
In this paper, we study rational cohomology of orthogonal group scheme $ \Ooper_{n,n} $ and symplectic group scheme $ \Spoper_{n} $ with coefficients in polynomial algebra $S^*\left(\left(\left(\Bbbk^{2n\vee}\right)^{(r)}\right)^{\oplus\ell}\right)$.
Here, $ \left(\left(\Bbbk^{2n\vee}\right)^{(r)}\right)^{\oplus\ell} $ (for simplicity, we write $ \Bbbk^{2n\vee(r)\oplus \ell} $) is the direct sum of $ \ell $ copies of $ r $-th Frobenius twist of dual linear of $ \Bbbk $-vector space $ \Bbbk^{2n} $ and $\displaystyle S^*\left(\Bbbk^{2n\vee(r)\oplus \ell}\right)=\bigoplus_{d=0}^{\infty}S^d\left(\Bbbk^{2n\vee(r)\oplus \ell}\right)$ is polynomial algebra on $ \Bbbk $-vector space $ \Bbbk^{2n\vee(r)\oplus \ell} $.

Throughout, in order to consider these two group schemes simultaneously, let $ G_{n} $ denote either the orthogonal group scheme $ \Ooper_{n,n} $ or the symplectic group scheme $ \Spoper_{n} $.
Then $ G_{n} $ is by definition a subgroup scheme of $ \GL_{2n} $, and acts naturally on the vector space $ \Bbbk^{2n} $ by matrix multiplication.
By applying the $r$-th Frobenius twist (for $r\ge 0$), and taking $\ell$ copies of the representation obtained, we will obtain an action of $G_n$ on $ \Bbbk^{2n\vee(r)\oplus \ell} $, subsequently an action by algebraic automorphisms on the polynomial algebra $ S^{*}\left(\Bbbk^{2n\vee(r)\oplus \ell}\right) $.

The main theorem of this paper (Theorem \ref{thm3.3.11}) is the following explicit formulas of cohomology of the orthogonal groups and symplectic ones.

\begin{thms}[Partial Cohomological First and Second Fundamental Theorems for $ \Ooper_{n,n} $ and $ \Spoper_{n} $]
	Suppose that $p>2$. If $n\ge p^{r}\ell$ then the cohomology algebra
	\begin{equation*}
	H^{*}_{\mathrm{rat}}\left(G_{n},S^{*}\left(\Bbbk^{2n\vee(r)\oplus \ell}\right)\right)
	\end{equation*}
	is a symmetric algebra on a finite set of generators $(h|i|j)_{G_{n}}\in H^{2h}_{\mathrm{rat}}\left(G_{n},S^{*}\left(\Bbbk^{2n\vee(r)\oplus \ell}\right)\right)$ where $0\le h<p^r$, $0\le i\le j\le \ell$, and $i\ne j$ if $G_n$ is the symplectic group $\Spoper_n$.
	
	Moreover, there are no relations among the $(h|i|j)_{G_{n}}$.
\end{thms}

We don't compute directly the cohomology algebra $ H^{*}_{\mathrm{rat}}\left(G_{n},S^{*}\left(\Bbbk^{2n\vee(r)\oplus \ell}\right)\right) $ in the category of rational $ G_{n} $-modules.
Instead of that, we will develop Touz\'e's idea in \cite{Tou10a}, which allows to convert the calculation this cohomology algebra into the calculation of an appropriate $ \Ext $-group in the category $ \PcalK $ of strict polynomial functors.
This category $ \PcalK $ has been invented by Friedlander and Suslin \cite{FS97}, and the calculation of $ \Ext $-groups here is much simpler than in the category of rational $ G_{n} $-modules.

A similar result with our theorem, but for rational cohomology of general linear group scheme $ \GL_n $ has been carried out by Touz\'e in \cite{Tou12}.
This paper of Touz\'e together with  \cite{Tou10a} are the starting point for our study.

Now, in our main theorem, by taking degree zero part of the cohomology algebra, we get back the partial first and second fundamental theorems for $ \Ooper_{n,n} $ and $ \Spoper_{n} $, an important result in classical invariant theory.
Denote by $e_1^\vee,e_2^\vee,\ldots,e_{2n}^\vee$ the dual basis of the canonical basis of $\Bbbk^{2n}$.
We have
\begin{equation*}
(0|i|j)_{G_{n}}(x_{1},...,x_{\ell})=\omega_{G_{n}}(x_{i},x_{j})
\end{equation*}
where $ \omega_{G_{n}} $ is $ \sum_{i=1}^{n}e_{i}^{\vee}\wedge e_{n+i}^{\vee}\in\Lambda^{2}(\Bbbk^{2n}) $ if $ G_{n}=\Spoper_n $ and is $ \sum_{i=1}^{n}e_{i}^{\vee}e_{n+i}^{\vee}\in S^{2}(\Bbbk^{2n}) $ if $ G_{n}=\Ooper_{n,n} $.

\begin{thms}[Partial First and Second Fundamental Theorems for $ \Ooper_{n,n} $ and $ \Spoper_{n} $ \cite{deCP76}] The set $(0|i|j)_{G_{n}}$
	where $0\le h<p^r$, $0\le i\le j\le \ell$, and $i\ne j$ if $G_n$ is the symplectic group $\Spoper_n$,	
	is a system of generators of the invariant algebra $H^{0}_{\mathrm{rat}}\left(G_{n},S^{*}\left(\Bbbk^{2n\vee\oplus \ell}\right)\right)$.
	Moreover, if $n\ge \ell$, there are no relations among the $(0|i|j)_{G_{n}}.$
\end{thms}

In general, consider a linear algebraic group $ G $, or a linear algebraic group scheme G, defined over the field $ \Bbbk $.
Let $ A $ be a commutative $ \Bbbk $-algebra on which $ G $ acts rationally by $ \Bbbk $-algebra automorphisms.
The classical invariant theory studies the ring of invariants $ A^{G} $.
We say that $ G $ has the finite generation property if the following holds:
\textsf{If $ G $ acts on a finitely generated commutative $ \Bbbk $-algebra $ A $, then the ring of invariants $ A^{G} $ is finitely generated as a $ \Bbbk $-algebra.}
A Noether's result in 1926 said that a finite group $ G $, viewed as a discrete algebraic group over $ \Bbbk $, has the finite generation property.
Due to the works of Nagata in 1964, Haboush in 1975 and Popov in 1979, we know that:
A linear algebraic group scheme $ G $ over $ \Bbbk $ has the the finite generation property if and only if $ G $ is a reductive linear algebraic group over $ \Bbbk $.

Higher invariant theory studies the cohomology algebra $ H^{*}(G,A) $.
Notice that the classical invariant algebra $ A^{G} $ is the degree zero part of $ H^{*}(G,A) $, i.e., $ A^{G}=H^{0}(G,A) $.
The basic problem in higher invariant theory is studying the finite generation property of the cohomology algebra $ H^{*}(G,A) $.
We say that $ G $ has the cohomological finite generation property if the following holds:
\textsf{If $ G $ acts on a finitely generated commutative $ \Bbbk $-algebra $ A $, then the cohomology ring $ H^{*}(G,A) $ is finitely generated as a $ \Bbbk $-algebra.}
It is obvious that the cohomological finite generation property implies the finite generation property.
Evens has generalized the above Noether's result: \emph{A finite group has the cohomological finite generation property over $ \Bbbk $.}

In 1997, Friedlander and Suslin \cite{FS97} extended the above result of Evens: \emph{A finite group scheme has the cohomological finite generation property.}
According to \cite{FS97}, this theorem is implied by the existence of certain universal extension classes for general linear groups over fields of positive characteristic $ p $.
In order to construct the universal extension classes, Friedlander and Suslin have introduced the category $ \PcalK $ of strict polynomial functors.

Finally, in 2010, Touz\'e and Van der Kallen \cite{Tou10b,TvdK10} showed that \emph{all reductive linear algebraic group over $ \Bbbk $ has the cohomological finite generation property}.
In other words, a linear algebraic group $ G $ is a reductive linear algebraic group if and only if $ G $ has the finite generation property, if and only if $ G $ has the cohomological finite generation property.

Thus, according to the theorem of  Touz\'e and Van der Kallen, the algebra $ H^{*}_{\mathrm{rat}}\left(G_{n},S^{*}\left(\Bbbk^{2n\vee(r)\oplus \ell}\right)\right) $ is finite generated.
Therefore, our contribution in this paper is to prove the algebra $ H^{*}_{\mathrm{rat}}\left(G_{n},S^{*}\left(\Bbbk^{2n\vee(r)\oplus \ell}\right)\right) $ is a polynomial algebra if $ n\ge p^{r}\ell $, and give an explicit set of generators in this case, which generalizes partial first and second fundamental theorems for $ \Ooper_{n,n} $ and $ \Spoper_{n} $.

Let us give the main idea of the proof of our main theorem \ref{thm3.3.11}.
We recall that  $\Bbbk$ is a fixed field of positive characteristic $p$.
Let $\VcalK$ denote the category of  finite dimensional $\Bbbk$-vector spaces and  $ \Bbbk $-linear maps.
For simplicity, in this introduction, we assume that $ \Bbbk $ is an infinite field.
A homogeneous strict polynomial functor of degree $ d $ is a functor $ F:\VcalK\to\VcalK $ such that for any $ V,W\in\VcalK $, the structure morphism $ F_{V,W} $ is a homogeneous polynomial map of degree $ d $ from $ \HomK{V,W} $ to $ \HomK{F(V),F(W)} $.
In other words, $ F_{V,W} $ is a $ \Bbbk $-linear map  $ \Gamma^{d}\HomK{V,W}\to \HomK{V,W} $, where $ \Gamma^{d} $ is the $ d $-th divided power.
Therefore, we define the category $ \Gamma^{d}\VcalK $ whose objects are the same as of $\VcalK$, while whose morphisms are $\Homx{\Gamma^d\VcalK}{V,W}=\Gamma^d\HomK{V,W}$.
Then, a homogeneous strict polynomial functor of degree $ d $ is a $\Bbbk$-linear functors from $\Gamma^{d}\VcalK$ to $\VcalK$.
These functors are just objects of the category $ \Pcalx{d}=\Pcalx{d;\Bbbk} $, whose morphisms are natural transformations.
The direct sum $ \PcalK=\bigoplus_{d=0}^{\infty}\Pcalx{d} $ is the category of all strict polynomial functors.
Typical examples of homogeneous strict polynomial functors of degree $ d $ are symmetric power $ S^{d} $, tensor power $ \otimes^{d} $, exterior power $ \Lambda^{d} $ or divided
power $ \Gamma^{d} $.
We have $ I=S^{1}=\Lambda^{1}=\Gamma^{1} $.

In general, we can define strict polynomial functors in several variables.
For each $ n $-tuple of natural numbers $ \dunder=(d_1,d_2,\ldots,d_{n}) $, we define $\Gamma^{\dunder}\VcalK$ to be the tensor product $\bigotimes_{i=1}^n\Gamma^{d_i}\VcalK$
and $ \Pcalx{\dunder}=\Pcalx{\dunder;\Bbbk} $ to be the category of  $\Bbbk$-linear functors from $\Gamma^{\dunder}\VcalK$ to $\VcalK$.
An object of this category is called a homogeneous strict polynomial functor of degree $ \dunder $ in $ n $ variables.

The category $ \PcalKn=\bigoplus_{\dunder}\Pcalx{\dunder} $ where the direct sum is indexed by all $ n $-tuples of natural numbers, is suited to study the representation theory of $\GL_{\munderline,\Bbbk}=\prod_{i=1}^{n}\GL_{m_i,\Bbbk}$, with a tuple of natural numbers $\munderline=(m_1,\ldots,m_n)$.
Evaluation on the standard representation $\Bbbk^{\munderline}=\left(\Bbbk^{m_1},\ldots,\Bbbk^{m_n}\right)$ of $ \GL_{\munderline,\Bbbk} $ induces a map
\begin{equation*}
\Extx{\Pcalx{\dunder}}{F,G}\to \Extx{\GL_{\munderline,\Bbbk}}{F(\Bbbk^{\munderline}),G(\Bbbk^{\munderline})}.
\end{equation*}
This map is an isomorphism in case $ m_i\ge d_i $ for all $ i=1,2,\ldots,n $, see \cite[Lemma 2.3]{Tou10a}.
As a result, we can use the category $ \PcalKn $ to study $\Ext$-groups of the general linear groups.
Moreover, Touz\'e \cite{Tou10a} showed that the strict polynomial functors could be used to study cohomology of the orthogonal and symplectic groups.

By using the relation between the cohomology of the symplectic and orthogonal groups and the computation of extension in $\PcalK$ \cite{Tou10a}, we obtain the following morphism of  $ \Bbbk $- graded algebras
\begin{equation*}
\Phi_{G_n}:\bigoplus_{d\ge 0}\Ext^{*}_{\PcalK} \left(\Gamma^{p^rd}\circ X_{G},S^{2d(r)}_{\Bbbk^\ell}\right)\to H^{*}_{\mathrm{rat}}\left(G_{n},S^{*}\left(\Bbbk^{2n\vee(r)\oplus \ell}\right)\right).
\end{equation*}
where $X_G\in\Pcalx{2}$ denotes  $S^2$ if $G_n=\Ooper_{n,n}$ and $\Lambda^2$ if $G_n=\Spoper_n$  (see \eqref{equ3.5.3}).
Here, for each $ F\in\Pcalx{\dunder} $ and $ r\in\Nbb $, the functor $ F^{(r)}\in\Pcalx{p^{r}\dunder} $ is given by $ F^{(r)}\left(V_1,V_2,\ldots,V_{n}\right)=F\left(V_1^{(r)},V_2^{(r)},\ldots,V_{n}^{(r)}\right) $ where $ V_i^{(r)} $ is  the $ r$-th Frobenius twist of $ \Bbbk $-vector space $ V_i $.

We prove in subsection \ref{subs3.3.5} that the morphism $\Phi_{G_n}$ is an isomorphism if $n\ge p^r\ell$.
For this, we use the notion of the $n$-coresolved functors introduced by Touz\'e \cite{Tou12} to solve an analogous problem for general linear group scheme.

We compute in Lemma \ref{lem3.4.4} the algebra source of the morphism $\Phi_{G_n}$.
We remark that $X_G$ is a direct factor of $\otimes^2$ since $p>2$. It suffices to compute $\Ext$-groups $\Ext^{*}_{\PcalK} \left(\Gamma^{p^rd}\circ \otimes^2,S^{2d(r)}_{\Bbbk^\ell}\right)$, compatible with the products and natural in $\otimes^2$.
It is treated in Section \ref{sect3.3}.
More generally, the $\Ext$-group $\Ext^{*}_{\PcalK} \left(\Gamma^{p^rd}\circ \otimes^n,S^{\mu(r)}\right)$ is computed in Theorem \ref{thm3.4.1}, where $ \mu $ is a tuple of natural numbers.

This research is a continuation of the work of Cha\l upnik, Touz\'e and Van der Kallen.
Among our results, we obtain in Theorem \ref{thm2.6.20} a graded isomorphism, natural in $F\PcalK$
\begin{equation}\label{iso2.1.1}
\Extx{\PcalK}{X^{p^rd}\circ\otimes^n,F^{(r)}}\simeq\Extxx{*-\epsilon(p^rd-d)}{\PcalK}{X^{d}\circ\HomK{E_r^{\otimes n-1},\text{-}}\circ\otimes^n,F}.
\end{equation}
We use the letter $X$ to denote one of the symbols $S$, $\Lambda$, or $\Gamma$; and $\epsilon=0,1,2$ if $X=\Gamma,\Lambda,S$ respectively.
$E_r$ is the graded vector space which equals $\Bbbk$ in degrees $0,2,4,\ldots,2p^r-2$ and which is zero in the other degrees.

\tableofcontents

\section{ Review of strict polynomial functors}

In this section, we introduce the categories $\Pcalx{\dunder}$ of strict polynomial functors in one or several variables.
Our main references are \cite{FS97,SFB97,FFSS99,Tou10a}

\subsection{ Schur categories  $\Gamma^{\dunder}\VcalK$}

Fix a field $\Bbbk$ of positive characteristic $p$.
Let $\VcalK$ denote the category of  finite dimensional $\Bbbk$-vector spaces and  $ \Bbbk $-linear maps.

Let $d$ be a natural number.
In this paper, natural numbers mean non-negative integers.
The tensor product over $ \Bbbk $ induces this functor
\begin{equation*}
\otimes^{d}:\VcalK\to\VcalK,\quad V\mapsto V^{\otimes d}=\underbrace{V\otimes V\otimes\cdots\otimes V}_{d\text{ times}}
\end{equation*}
The symmetric group $ \Sfrak_{d} $ acts on $ \otimes^{d} $ by permuting the factors of the tensor product, i.e., $ \sigma\cdot\left(v_{1}\otimes v_{2}\otimes\cdots\otimes v_{d}\right)=v_{\sigma^{-1}(1)}\otimes v_{\sigma^{-1}(2)}\otimes\cdots\otimes v_{\sigma^{-1}(d)} $.
For each $ V\in  \VcalK$, $ \Gamma^{d}(V):=\left(V^{\otimes d}\right)^{\Sfrak_{d}} $.
We have the $ d $-th divided power
\begin{equation*}
\Gamma^{d}:\VcalK\to\VcalK,\quad V\mapsto \Gamma^{d}(V).
\end{equation*}
In other words, $ \Gamma^{d}=\left(\otimes^{d}\right)^{\Sfrak_{d}} $.

For each pair $(V,W)$ of objects  of $\VcalK$, we define a morphism $\Gamma^d(V)\otimes \Gamma^d(W)\to\Gamma^d(V\otimes W)$ natural in $V,W$ to be the composition
\begin{equation*}
\left(V^{\otimes d}\right)^{\Sfrak_d}\otimes \left(W^{\otimes d}\right)^{\Sfrak_d}\simeq\left(V^{\otimes d}\otimes W^{\otimes d}\right)^{\Sfrak_d\times\Sfrak_d}\hookrightarrow \left((V\otimes W)^{\otimes d}\right)^{\Sfrak_d}.
\end{equation*}

\begin{deff}
	Let $d$ be a natural number.
	Let $\Acal,\Bcal$ be two  $\Bbbk$-linear categories and $F:\Acal\to\Bcal$ be a $\Bbbk$-linear functor.
	Denote by $\Gamma^d\Acal$ the category whose objects are the ones  of $\Acal$, while whose morphisms are $\Homx{\Gamma^d\Acal}{A,B}=\Gamma^d\Homx{\Acal}{A,B}$
	and the composition is given by 
	\begin{align*}
	\Gamma^d\Homx{\Acal}{B,C}\otimes\Gamma^d\Homx{\Acal}{A,B}&\to \Gamma^d\Big(\Homx{\Acal}{B,C}\otimes\Homx{\Acal}{A,B}\Big)\\
	&\to\Gamma^d\Homx{\Acal}{A,C}
	\end{align*}
	where the second morphism is induced by the composition in $\Acal$.
	We note that $\Gamma^1\Acal=\Acal$.
	Define the  $\Bbbk$-linear functor $\Gamma^dF:\Gamma^d\Acal\to\Gamma^d\Bcal$
	to be a functor that associates to each $A\in\Gamma^d\Acal$ a $F(A)$ and the structure morphisms $\left(\Gamma^dF\right)_{A_,B}$ are defined to be $\Gamma^d \left(F_{A,B}\right)$.	
\end{deff}

The categories $ \GammadVK $ are called \emph{Schur categories}.
For each natural number $ n $, we see that $ \Endx{\GammadVK}{\Bbbk^{n}}=\Gamma^{d}\EndK{\Bbbk^{n}}\simeq\Endx{\Bbbk\Sfrak_{d}}{\Bbbk^{n}} $ is the Schur algebra $ S_{\Bbbk}(n,d) $, see \cite{Gre07}.
This is why the categories $ \GammadVK $ are called  Schur categories.
We next present the basic relation of the Schur category with the category of $ \Bbbk $-vector spaces and the category of modules on group algebra of symmetric group.

Let $ \Bbbk\Sfrak_{d} $ be the group algebra of the symmetric group $ \Sfrak_{d} $ and $\Bbbk\Sfrak_d\Modp$ be the category of finite dimensional $\Bbbk\Sfrak_d$-modules.
For each $ \Bbbk $-vector space $ V $, $ V^{\otimes n}=\underbrace{V\otimes\cdots\otimes V}_{n\text{-times}} $ is equipped with a natural action of the  symmetric group $ \Sfrak_{d} $.
We have the functor
\begin{equation*}
\tau_{d}:\VcalK\to\Bbbk\Sfrak_{d}\Modp,\quad V\mapsto V^{\otimes d},\; f\mapsto f^{\otimes d}.
\end{equation*}
Since $ \Homx{\GammadVK}{V,W}\simeq \Homx{\Bbbk\Sfrak_{d}}{V^{\otimes d},W^{\otimes d}} $ natural in $V,W\in\VcalK$, we can consider the Schur category $ \GammadVK $ as the image of the functor $ \tau_{d} $: the objects of  $ \GammadVK $ are the ones of $ \VcalK $, and the morphisms of $ \GammadVK $ are the ones in $ \Bbbk\Sfrak_{d}\Modp $.
Hence, the functor $ \tau_{d} $ can be analyzed through the category $ \GammadVK $

\begin{equation}\label{equ1.1}
\xymatrix{
	\VcalK\ar[rr]^{\tau_{d}}\ar[rd]_{\gamma_d}&&\Bbbk\Sfrak_d\Modp\\
	&\GammadVK\ar[ru]_{\iota_d}	
}
\end{equation}
where the functor $ \gamma_d $ is identical on the objects and the functor $ \iota_d $ is identical on the arrows.
The canonical functor
\begin{equation*}
\Gamma^d\VcalK\to\Bbbk\Sfrak_d\Modp,\quad V\mapsto V^{\otimes n}
\end{equation*}
is an embedding of categories.

Let $d,e$ be two natural numbers.
We will give the definition of the functors $ \Gamma^{d+e}\VcalK\to\Gamma^d\VcalK\otimes\Gamma^{e}\VcalK $ and $ \Gamma^{de}\VcalK\to\Gamma^d\left(\Gamma^e\VcalK\right) $.

We define a morphism $\Delta_{d,e}:\Gamma^{d+e}\to\Gamma^d\otimes\Gamma^e$ to be {the following composition}  where the first morphism is induced by the inclusion $\Sfrak_d\times\Sfrak_e\hookrightarrow\Sfrak_{d+e}$:
\begin{equation*}
\left(\otimes^{d+e}\right)^{\Sfrak_{d+e}}\hookrightarrow\left(\otimes^{d+e}\right)^{\Sfrak_d\times\Sfrak_e}\simeq \left(\otimes^d\right)^{\Sfrak_d}\otimes\left(\otimes^e\right)^{\Sfrak_e}.
\end{equation*}
The morphism $\Delta_{d,e}$ induces naturally a  $\Bbbk$-linear functor $\Delta_{d,e}:\Gamma^{d+e}\VcalK\to\Gamma^d\VcalK\otimes\Gamma^{e}\VcalK$.
This functor associates  to each $V$ a pair $(V,V)$ and the structure morphisms  $\left(\Delta_{d,e}\right)_{V,W}$ are given by
\begin{equation*}
\Gamma^{d+e}\HomK{V,W}\xrightarrow{\Delta_{d,e}}\Gamma^d\HomK{V,W}\otimes\Gamma^e\HomK{V,W}.
\end{equation*}

Define a morphism $c_{d,e}:\Gamma^{de}\to\Gamma^d\circ\Gamma^e$ to be the composition
\begin{equation*}
\left(\otimes^{de}\right)^{\Sfrak_{de}}\hookrightarrow\left(\otimes^{de}\right)^{(\Delta_d\Sfrak_e)\times\Sfrak_d}\xrightarrow{\simeq}\left(\left(\otimes^{de}\right)^{\Delta_d\Sfrak_e}\right)^{\Sfrak_d}\xrightarrow{\simeq}\left(\left((\otimes^e)^{\Sfrak_e}\right)^{\otimes d}\right)^{\Sfrak_d},
\end{equation*}
where the group $\Delta_d\Sfrak_e$ is the image of the composition $\Sfrak_e\xrightarrow{\Delta_d}\left(\Sfrak_e\right)^{\times d}\hookrightarrow\Sfrak_{de}$, {with $\Delta_d$  the diagonal functor.}
The morphism $c_{d,e}$ induces naturally a $\Bbbk$-linear functor $c_{d,e}:\Gamma^{de}\VcalK\to\Gamma^d\left(\Gamma^e\VcalK\right)$. 
This functor sends each  $V$ to $V$ and the structure morphisms $\left(c_{d,e}\right)_{V,W}$ are given by
\begin{equation*}
\Gamma^{de}\HomK{V,W}\xrightarrow{c_{d,e}}\Gamma^d\left(\Gamma^e\HomK{V,W}\right).
\end{equation*}

Let $\dunder=(d_1,\ldots,d_n)$ be a $n$-tuple of natural numbers.
Denote by $\Gamma^{\dunder}\Acal$ the tensor product $\bigotimes_{i=1}^n\Gamma^{d_i}\Acal$.
\begin{itemize}
	\item Let $\dunder=(d_1,\ldots,d_n)$ and $\eunder=(e_1,\ldots,e_n)$ be two $n$-tuples of natural numbers. We denote by $\dunder+\eunder$ (resp. $\dunder\cdot\eunder$) the $n$-tuple $(d_1+e_1,\ldots,d_n+e_n)$ (resp. $(d_1e_1,\ldots,d_ne_n)$).
	\item Let $\dunder=(d_1,\ldots,d_n)$ be a $n$-tuple of natural numbers and $\eunder=(e_1,\ldots,e_m)$ be a $m$-tuple of natural numbers. We denote by $(\dunder,\eunder)$ the $(n+m)$-tuple $(d_1,\ldots,d_n,e_1,\ldots,e_m)$.
	\item {The weight of a $n$-tuple of natural numbers $\dunder=(d_1,\ldots,d_n)$ is $|\dunder|=\sum_{i=1}^{n}d_i$.}
	\item For each $n$-tuple {of natural numbers} $\dunder=(d_1,\ldots,d_n)$, we denote by $\Sfrak_{\dunder}$ the direct product $\prod_{i=1}^{n}\Sfrak_{d_i}$. We have a canonical monomorphism of groups $\Sfrak_{\dunder}\to\Sfrak_{|\dunder|}$.
\end{itemize}
For any two $n$-tuples of natural numbers $\dunder=(d_1,\ldots,d_n)$ and $\eunder=(e_1,\ldots,e_n)$, we have a functor $\Delta_{\dunder,\eunder}:\Gamma^{\dunder+\eunder}\VcalK\to\Gamma^{\dunder}\VcalK\otimes\Gamma^{\eunder}\VcalK$ defined to be {the composition below}
\begin{equation*}
\Gamma^{\dunder+\eunder}\VcalK=\bigotimes_{i=1}^n\Gamma^{d_i+e_i}\VcalK\xrightarrow{\otimes_{i=1}^n\Delta_{d_i,e_i}}\bigotimes_{i=1}^n\Gamma^{d_i}\VcalK\otimes\Gamma^{e_i}\VcalK\simeq \Gamma^{\dunder}\VcalK\otimes\Gamma^{\eunder}\VcalK.
\end{equation*}

\subsection{ Categories $\Pcalx{\dunder}$ of homogeneous strict polynomial functors of degree $\dunder$}

After defining the  Schur categories in the previous subsection, we now introduce the $ \Bbbk $-linear representations of these categories, that are the strict polynomial functors.

\begin{deff}
	Let $\dunder=(d_1,\ldots,d_n)$ be a $n$-tuple of natural numbers.
	The category $\Pcalx{\dunder}=\Pcalx{\dunder;\Bbbk}$ is the category of  $\Bbbk$-linear functors from $\Gamma^{\dunder}\VcalK$ to $\VcalK$.
	The morphisms in $ \Pcalx{\dunder} $ are the natural transformations.
	An object of $\Pcalx{\dunder}$ is called  \emph{homogeneous strict polynomial functor} of degree $\dunder$ in $ n $ variables.

	Let $d$ be a natural number.
	The categories $\Pcalx{d}(n)=\Pcalx{d;\Bbbk}(n)$ and $\PcalKn$ are defined by 
	\begin{equation*}
	\Pcalx{d}(n)=\bigoplus_{|\dunder|=d}\Pcalx{\dunder},\qquad \PcalKn=\bigoplus_{d=0}^\infty\Pcalx{d}(n).
	\end{equation*}
	An object of $\Pcalx{d}(n)$ is called  \emph{strict  polynomial functor} of total degree $d$ in $ n $ variables.
	Moreover, an object of $\PcalKn$ is called  strict polynomial functor in $n$ variables.
	If $n=1$, the categories $\Pcalx{d}(1), \PcalK(1)$ are simply denoted by  $\Pcalx{d},\PcalK$ respectively.
\end{deff}

Since $\VcalK$ is a $\Bbbk$-linear abelian category, the category $\Pcalx{\dunder}$ is also $\Bbbk$-linear abelian: the direct sum, the direct product, the  kernels and cokernels are defined in the category $ \VcalK $, for example $(F\oplus G)(V)=F(V)\oplus G(V)$.

We next give examples of strict polynomial functors and their relation with other objects.

\begin{exe} 	
	Let $d$ be a natural number.
	We recall that, by definition, the functor $\iota_d:\Gamma^d\VcalK\to\Bbbk\Sfrak_d\Modp,V\mapsto V^{\otimes n}$ is an embedding, \eqref{equ1.1}.
	Thus each $\Bbbk$-linear functor $\psi:\Bbbk\Sfrak_d\Modp\to\VcalK$ induces the strict  polynomial functor $\psi\circ\iota_d\in\Pcalx{d}$. 
	Moreover, if we have a natural transformation  $\tau:\psi_1\to\psi_2$ between two  $\Bbbk$-linear functors $\psi_1,\psi_2:\Bbbk\Sfrak_d\Modp\to\VcalK$, then we obtain a morphism $\psi_1\circ\iota_d\to\psi_2\circ\iota_d$ in the category $\Pcalx{d}$.
	
	We now give a list of particular cases.
	\begin{enumerate}
		\item \textbf{Tensor power.} If $\psi$ is the forgetful functor $\Bbbk\Sfrak_d\Modp\to\VcalK$, then the functor $\psi\circ\iota_d$ is denoted by $\otimes^d$ and called \emph{$d$-th tensor power functor}.

		\item \textbf{Divided power.}
		If $\psi$ is the invariant functor $(\text{-})^{\Sfrak_d}$, then the functor $(\text{-})^{\Sfrak_d}\circ\iota_d$ is denoted by $\Gamma^d$ and called  \emph{$d$-th divided power functor}.

		\item Let $\dunder$ be a $n$-tuple of natural numbers with $|\dunder|=d$. If $\psi$ is the functor $(\text{-})^{\Sfrak_{\dunder}}:\Bbbk\Sfrak_d\Modp\to\VcalK$, then the functor $(\text{-})^{\Sfrak_{\dunder}}\circ\iota_d$ is denoted by $\Gamma^{\dunder}$.

		Moreover, the canonical natural  transformations  $(\text{-})^{\Sfrak_d}\to(\text{-})^{\Sfrak_{\dunder}}$ and  $(\text{-})^{\Sfrak_{\dunder}}\to (\text{-})^{\Sfrak_d}$ induce morphisms of strict polynomial functors $\Delta_{\dunder}:\Gamma^d\to \Gamma^{\dunder}$ and $m_{\dunder}:\Gamma^{\dunder}\to\Gamma^{d}$.

		\item \textbf{Symmetric power.}
		If $\psi$ is the covariant functor $(\text{-})_{\Sfrak_d}$, then the functor $(\text{-})_{\Sfrak_d}\circ\iota_d$ is denoted by $S^d$ and called \emph{$d$-th symmetric power functor}.

		\item Let $\dunder$ be a $n$-tuple of natural numbers with $|\dunder|=d$. If $\psi$ is the functor $(\text{-})_{\Sfrak_{\dunder}}:\Bbbk\Sfrak_d\Modp\to\VcalK$, then the functor $(\text{-})_{\Sfrak_{\dunder}}\circ\iota_d$ is denoted by $S^{\dunder}$.

		Moreover, the canonical natural  transformations  $(\text{-})_{\Sfrak_d}\to(\text{-})_{\Sfrak_{\dunder}}$ and $(\text{-})_{\Sfrak_{\dunder}}\to (\text{-})_{\Sfrak_d}$ induce morphisms of strict polynomial functors $\Delta_{\dunder}:S^d\to S^{\dunder}$ and $m_{\dunder}:S^{\dunder}\to S^{d}$.

		\item \textbf{Exterior power.}
		We next define the \emph{$d$-th exterior power functor} as a strict polynomial functor.
		We consider two cases corresponding to the characteristic of the field $ \Bbbk $.
		
		\begin{itemize}
			\item Case $ p $ is odd.
			If $\psi$ is the composition $\Bbbk\Sfrak_d\Modp\xrightarrow{\BbbkSgn\otimes\text{-}}\Bbbk\Sfrak_d\Modp\xrightarrow{(\text{-})_{\Sfrak_d}}\VcalK$ where $\BbbkSgn$ is the  sign representation of $\Sfrak_d$, then the functor $\psi\circ\iota_d$ is denoted by $\Lambda^d$ and called $d$-th  exterior power functor.
			
			\item Case $ p=2 $.
			We define  $d$-th exterior power functor $\Lambda^d$ to be the image of the composition
			$S^d\xrightarrow{\Delta_{(1,\ldots,1)}}\otimes^d\xrightarrow{m_{(1,\ldots,1)}}\Gamma^d$.
		\end{itemize}

		\item We suppose that $p$ is odd. Let $\dunder$ be a $n$-tuple of natural numbers with $|\dunder|=d$. If $\psi$ is the composition $\Bbbk\Sfrak_d\Modp\xrightarrow{\BbbkSgn\otimes\text{-}}\Bbbk\Sfrak_d\Modp\xrightarrow{(\text{-})_{\Sfrak_{\dunder}}}\VcalK$ , then the functor $\psi\circ\iota_d$ is denoted by $\Lambda^{\dunder}$. 
		
		Moreover, the canonical natural transformations $(\text{-})_{\Sfrak_d}\to(\text{-})_{\Sfrak_{\dunder}}$ and $(\text{-})_{\Sfrak_{\dunder}}\to (\text{-})_{\Sfrak_d}$ induce  morphisms of strict polynomial functors $\Delta_{\dunder}:\Lambda^d\to \Lambda^{\dunder}$ and $m_{\dunder}:\Lambda^{\dunder}\to \Lambda^{d}$.
	\end{enumerate}
\end{exe}

\begin{exe}
	\begin{enumerate}
		\item The strict  polynomial functor $\boxtimes^n\in\Pcalx{(1,\ldots,1)}$ is the functor that associates to each  $(V_1,\ldots,V_n)\in \VcalK^{\otimes n}$ the tensor product $\bigotimes_{i=1}^nV_i$. Denote by $I$ the functor $\boxtimes^1\in\Pcalx{1}$.
		
		\item \textbf{Exterior tensor product.}
		Let $\dunder,\eunder$ be two tuples of natural numbers and $F\in\Pcalx{\dunder},G\in\Pcalx{\eunder}$. Define a strict  polynomial functor $F\boxtimes G\in\Pcalx{(\dunder,\eunder)}$ to be the composition
		\begin{equation*}
		\Gamma^{(\dunder,\eunder)}\VcalK=\Gamma^{\dunder}\VcalK\otimes\Gamma^{\eunder}\VcalK\xrightarrow{F\otimes G}\VcalK\otimes\VcalK\xrightarrow{\boxtimes^2}\VcalK.
		\end{equation*}
		In particular, we have $\boxtimes^n\boxtimes\boxtimes^m\simeq\boxtimes^{n+m}$. If $X$ denotes one of the symbols $\Gamma,\Lambda$ or $S$, then we denote by  $X^{\boxtimes\dunder}$ the functor $\boxtimes_{i=1}^nX^{d_i}\in\Pcalx{\dunder}$.
		
		\item \textbf{Tensor product.}
		Let $\dunder,\eunder$ be two  $n$-tuples of natural numbers and $F\in\Pcalx{\dunder},G\in\Pcalx{\eunder}$. Define a strict polynomial functor $F\otimes G\in\Pcalx{\dunder+\eunder}$ to be the composition
		\begin{equation*}
		\Gamma^{\dunder+\eunder}\VcalK\xrightarrow{\Delta_{\dunder,\eunder}}\Gamma^{\dunder}\VcalK\otimes\Gamma^{\eunder}\VcalK\xrightarrow{F\boxtimes G}\VcalK.
		\end{equation*}
		In particular, we have $\otimes^d\otimes\otimes^e\simeq\otimes^{d+e}$.
		Moreover, if $X$ denotes one of the symbols $\Gamma,\Lambda$ or $S$, then we have isomorphisms of strict polynomial functors $X^{\dunder}\simeq\bigotimes_{i=1}^nX^{d_i}$.
		
		\item \textbf{Composition.}
		Let $F\in\Pcalx{d}$ and $G\in\Pcalx{e}$. Define the strict polynomial functor $F\circ G\in\Pcalx{de}$ to be the composition
		\begin{equation*}
		\Gamma^{de}\VcalK\xrightarrow{c_{d,e}}\Gamma^d\left(\Gamma^e\VcalK\right)\xrightarrow{\Gamma^dG}\Gamma^d\VcalK\xrightarrow{F}\VcalK.
		\end{equation*}
		More generally, let $\dunder,\eunder$ be two  $n$-tuples of natural numbers $F\in\Pcalx{\dunder}$ and $G_i\in\Pcalx{e_i}$ for $i=1,\ldots,n$. Define a strict polynomial functor $F\circ(G_1,\ldots,G_n)\in\Pcalx{\dunder\cdot\eunder}$ to be the composition
		\begin{equation*}
		\Gamma^{\dunder\cdot\eunder}\VcalK\simeq\bigotimes_{i=1}^n\Gamma^{d_ie_i}\VcalK\xrightarrow{\otimes_{i=1}^nc_{d_i,e_i}}\bigotimes_{i=1}^n\Gamma^{d_i}\left(\Gamma^{e_i}\VcalK\right)\xrightarrow{\otimes_{i=1}^n\Gamma^{d_i}G_i}\bigotimes_{i=1}^n\Gamma^{d_i}\VcalK\xrightarrow{F}\VcalK.
		\end{equation*}
		
		\item \textbf{Parameterized functor.}
		Let $\dunder$ be a $n$-tuple of natural numbers, $F\in\Pcalx{\dunder}$ and 
		$\Vunderline=(V_1,\ldots,V_n)$ be a $n$-tuple of objects of $\VcalK$. 
		For each $i$, the  $\Bbbk$-linear functor $V_i\otimes\text{-}:\VcalK\to\VcalK$ may be seen to be a strict  polynomial functor $V_i\otimes\text{-}\in\Pcalx{1}$. 
		The functor  {obtained by the composition} $F\circ(V_1\otimes\text{-},\ldots,V_n\otimes\text{-})$ is denoted by $F_{\Vunderline}$ and called \emph{parameterized functor} of $F$ with respect to $\Vunderline$.
		Denote by $F^{\Vunderline}$ the functor $F_{\Vunderline^\vee}$ where $\Vunderline^\vee=\left(V_1^\vee,\ldots,V_n^\vee\right)$ is the $\Bbbk$-linear duality of $\Vunderline$.
		By definition, the parameterization with respect to $\Vunderline$ defines an exact functor $\Pcalx{\dunder}\to\Pcalx{\dunder}, F\mapsto F_{\Vunderline}$.
	\end{enumerate}
\end{exe}

Since the monoidal product $\otimes$ of $\VcalK$ is exact in each variable, 
the functors $ \Pcalx{\dunder}\times\Pcalx{\eunder}\to\Pcalx{\dunder+\eunder},(F,G)\mapsto F\otimes G $ and $ \Pcalx{\dunder}\times\Pcalx{\eunder}\to\Pcalx{(\dunder,\eunder)},(F,G)\mapsto F \boxtimes G $
are again exact in each variable.

Let $ \gamma_{d}:\VcalK\to\GammadVK $, \eqref{equ1.1}.
We notice that $ \gamma_{d} $ is just a linear functor as $ d=1 $, then, $ \gamma_{1}$ is the identity functor of the category $ \VcalK $.
The composition of each functor $ F\in\Pcalx{d} $ with $ \gamma_{d} $ gives the functor $ F\circ\gamma_{d}:\VcalK\to\VcalK $.
Recall that we often denote by $ \FcalK $ the category of functors from $ \VcalK $ to the category $ \mathrm{Vec}_{\Bbbk} $ of the  $ \Bbbk $-vector spaces and the $ \Bbbk $-linear maps.
Thus, the precomposition by $ \gamma_{d} $ yields a functor, is often called  forgetful functor
\begin{equation*}
\text{-}\circ\gamma_{d}:\Pcalx{d}\to\FcalK.
\end{equation*}
Precomposition by $ \gamma_{d} $ of the functor $ \otimes^{d}$ (resp., $S^{d},\Gamma^{d} $ and $ \Lambda^{d} $) gives the ordinary $ d $-tensor power (resp., $ d $-symmetric power and $ d $-exterior power).
The relation between $\Ext$-groups
in the category $ \PcalK $ with $\Ext$-groups in the category $ \FcalK $ has been studied in the paper \cite{FFSS99}, see for example \cite[Thm 3.10]{FFPS03}.
Concerning $\Ext$-groups in the category $ \FcalK $, see \cite{FLS94}.
The category $ \PcalK $ presents several
advantages over the category $ \FcalK $ from the point of view of computing $ \Ext $-groups since these are the accessibility of injectives and projectives, see subsection \ref{subsect1.3}.
We notice that not every functor in the category $ \FcalK $ is a precomposition by $ \gamma_{d} $ of some functor $ F\in\Pcalx{d} $.
Moreover, if this functor $ F $ exists, it may not be unique.

For each $ F\in\Pcalx{d} $ and each natural number $ n $, we have the structure morphism
\begin{equation*}
\SKnd=\Endx{\GammadVK}{\Bbbk^{n}}\xrightarrow{F_{\Bbbk^{n},\Bbbk^{n}}}\EndK{F(\Bbbk^{n})}.
\end{equation*}
This provides $ F(\Bbbk^{n}) $ with the structure of a $ \SKnd $-module.
Furthermore, associating to $ F\in\Pcalx{d} $ the corresponding module $ F(\Bbbk^{n})\in\SKnd\Modp $, we get an exact functor
\begin{equation*}
\mathrm{eval}_{\Bbbk^{n}}:\Pcalx{d}\to S_{\Bbbk}(n,d)\text{-mod}.
\end{equation*}

In the following, we recall an important duality in $ \PcalK $, called \emph{Kuhn duality}.
The $ \Bbbk $-linear duality of a vector space $ V $ is denoted by $ V^{\vee} $.
We have the functor $ (\text{-})^{\vee}:\VcalK^\Op\to\VcalK $.
Moreover, for each $ V\in\VcalK $, we have $ V\cong \left(V^{\vee}\right)^{\vee} $ canonically.

\begin{deff}
	Let $F\in\Pcalx{\dunder}$.
	Denote by $F^\sharp$ the following composition, where the first and third morphisms are induced by the $\Bbbk$-linear duality  $(\text{-})^\vee:\VcalK^\Op\to\VcalK$:
	\begin{equation*}
	\Gamma^{\dunder}\VcalK\to\Gamma^{\dunder}\VcalK^\Op\xrightarrow{F}\VcalK^\Op\to \VcalK.
	\end{equation*}
	In particular, we have $F^\sharp(\Vunderline)=F\left(\Vunderline^\vee\right)^\vee$.
	The functor $F^\sharp$ is called \emph{Kuhn duality} of $F$.	
	We also obtain the functor  $\Pcalx{\dunder}^\Op\to\Pcalx{\dunder},F\mapsto F^\sharp$.
\end{deff}

By definition, we have isomorphisms
\begin{equation*}
S^{d\sharp}\simeq \Gamma^d,\quad \Lambda^{d\sharp}\simeq\Lambda^{d},\quad \Gamma^{d\sharp}\simeq S^{d}.
\end{equation*}
Since the functor $V\mapsto V^\vee$ is exact, the functor $F\mapsto F^\sharp$ is also exact.
Moreover, the isomorphisms $\left(V^\vee\right)^\vee\simeq V$ natural in $V\in\VcalK$ induces isomorphisms natural in $F,G\in\Pcalx{\dunder}$, with $i\in\Nbb$:
\begin{equation*}
\left(F^\sharp\right)^\sharp\simeq F,\qquad \Extxx{i}{\Pcalx{\dunder}}{F,G}\simeq\Extxx{i}{\Pcalx{\dunder}}{G^\sharp,F^\sharp}.
\end{equation*}

\begin{exe} Let $F,G\in\PcalK(n),H\in\PcalK(m)$ and $\Vunderline$ be a  $n$-tuple of objects of $\VcalK$, and $G_1,\ldots,G_n\in\PcalK$. We have isomorphisms, natural 
	\begin{align*}
	&(F\oplus G)^\sharp\simeq F^\sharp\oplus G^\sharp,\\
	&(F\otimes G)^\sharp\simeq F^\sharp\otimes G^\sharp,\\
	&(F\boxtimes H)^\sharp\simeq F^\sharp\boxtimes H^\sharp,\\
	&\left(F_{\Vunderline}\right)^\sharp\simeq \left(F^\sharp\right)^{\Vunderline},\\
	&\Big(F\circ (G_1,\ldots,G_n) \Big)^\sharp\simeq F^{\sharp}\circ\left(G_1^\sharp,\ldots,G_n^\sharp\right).
	\end{align*}
\end{exe}

\subsection{ Projective and injective functors in $\Pcalx{\dunder}$}
\label{subsect1.3}

One of the reasons why the calculation of $\Ext$-groups in the category $ \PcalK $ is simpler than in other categories, such as $ \FcalK $,  $ \Ucal $ or $ \GLnK\text{-mod} $ is that,  $ \PcalK $ (and $ \PcalKn $) has projective objects and injective objects which are relatively simple.
Moreover, the class of the projective (respectively, injective) objects is closed with respect to the tensor product.

In this section, we recall the definition and some simple properties of the projective objects and injective objects in $ \PcalKn $.

Let $\dunder$ be a $n$-tuple of natural numbers.
By definition,  $ \Pcalx{\dunder} $ is the category of $ \Bbbk $-linear functors from $ \Gamma^{\dunder}\VcalK $ to $ \VcalK $.
Moreover, for each $ \Vunderline\in \Gamma^{\dunder}\VcalK$, we have $\Gamma^{\boxtimes\dunder,\Vunderline}(\Wunderline)=\left(\Gamma^{\boxtimes\dunder}\right)^{\Vunderline}(\Wunderline)\simeq\Homx{\Gamma^{\dunder}\VcalK}{\Vunderline,\Wunderline}$ natural in $ \Wunderline\in \Gamma^{\dunder}\VcalK$.
Therefore, due to $ \Bbbk $-linear version of the Yoneda lemma, there exists an isomorphism natural in $F\in\Pcalx{\dunder}$ and $\Vunderline\in\Gamma^{\dunder}\VcalK$:
\begin{equation}\label{equ1.3.1a}
\Homx{\Pcalx{\dunder}}{\Gamma^{\boxtimes\dunder,\Vunderline},F}\simeq F(\Vunderline).
\end{equation}

\begin{pro}\label{pro1.5.7}
	Let $\dunder$ be a $n$-tuple of natural numbers.
	\begin{enumerate}
		\item The functor $\Gamma^{\dunder}\VcalK^\Op\to\Pcalx{\dunder},\Vunderline\mapsto \Gx{\dunder,}{\Vunderline}$ is an embedding (called the \emph{Yoneda embedding}).
		\item The functor  $\Gx{\dunder,}{\Vunderline}$ is projective for all $\Vunderline\in\Gamma^{\dunder}\VcalK$. Moreover, these functors form a projective generating system of $\Pcalx{\dunder}$.
	\end{enumerate}
\end{pro}


A direct consequence of Proposition \ref{pro1.5.7} is that the product $F\boxtimes G$ is projective in $ \Pcalx{(\dunder,\eunder)} $ if the functors
$F\in\Pcalx{\dunder},G\in\Pcalx{\eunder}$ are  projective. 
Moreover, 
for each natural number $\ell$, we have an  isomorphism of strict polynomial functors
\begin{equation}\label{equ1.3.0}
\Gamma^{d,\Bbbk^\ell}\simeq\bigoplus_{\mu}\Gamma^{\mu}
\end{equation}
where the direct sum is indexed by the $\ell$-tuples of natural numbers $\mu=(\mu_1,\ldots,\mu_\ell)$ such that $|\mu|=d$.
The functors $\Gamma^\mu$ are then projective. As a result, if $F\in\Pcalx{\dunder},G\in\Pcalx{\eunder}$ are two projective functors where $\dunder,\eunder$ are two  $n$-tuples of natural numbers, thus tensor product  $F\otimes G\in\Pcalx{\dunder+\eunder}$ is also a projective functor. 

By using the exact functor $F\mapsto F^\sharp$ of  $\Pcalx{\dunder}$, we can establish some properties of injective functors from those of projective functors.
Note that a functor $ F\in\Pcalx{\dunder} $ is injective if and only if its Kuhn duality $ F^{\sharp} $ is projective.

\begin{pro}Let $\dunder$ be a $n$-tuple of natural numbers.
	\begin{enumerate}
		\item There exists an  isomorphism natural in  $F\in\PcalKoper{\dunder}$ and in $\Vunderline\in\Gamma^{\dunder}\VcalK$:
		\begin{equation*}
		\Homx{\Pcalx{\dunder}}{F,\Sx{\dunder}{\Vunderline}}\simeq F^\sharp\left(\Vunderline\right).
		\end{equation*}
		\item The functor $\Gamma^{\dunder}\VcalK\to\Pcalx{\dunder},\Vunderline\mapsto \Sx{\dunder}{\Vunderline}$ is an embedding (called the \emph{Yoneda embedding}).
		\item The functor $\Sx{\dunder}{\Vunderline}$ is injective for all $\Vunderline\in\Gamma^{\dunder}\VcalK$. Moreover, these functors form a system of injective cogenerators of $\PcalKoper{d}$.
	\end{enumerate}
\end{pro}

Exactness of tensor products gives us the following K\"{u}nneth-type formula.

\begin{pro}
	Let $\dunder$ be a $n$-tuple of natural numbers and $F_i,G_i\in\Pcalx{d_i}$ for $i=1,\ldots,n$. We have a graded isomorphism natural in $F_i,G_i,i=1,\ldots,n$:
	\begin{equation}\label{equ1.3.2a}
	\Extx{\Pcalx{\dunder}}{\boxtimesx{i=1}{n}F_i,\boxtimesx{i=1}{n}G_i}\simeq\bigotimes_{i=1}^n\Extx{\Pcalx{d_i}}{F_i,G_i}.
	\end{equation}
\end{pro}
\begin{proof}
	We first prove the isomorphism \eqref{equ1.3.2a} for all $\Hom$. Define the morphism:
	\begin{align*}
	\phi: \bigotimes_{i=1}^n\Homx{\Pcalx{d_i}}{F_i,G_i}&\to \Homx{\Pcalx{\dunder}}{\boxtimesx{i=1}{n}F_i,\boxtimesx{i=1}{n}G_i}\\
	f_1\otimes\cdots\otimes f_n&\mapsto f_1\boxtimes \cdots\boxtimes f_n.
	\end{align*}
	Consider $ \bigotimes_{i=1}^n\Homx{\Pcalx{d_i}}{F_i,G_i}$ and $ \Homx{\Pcalx{\dunder}}{\boxtimesx{i=1}{n}F_i,\boxtimesx{i=1}{n}G_i} $ of the morphism $\phi$ as two functors in  variables $(F_1,\ldots,F_n)$. Since these two functors are left exact, in order to prove that $\phi$ is an  isomorphism, we can assume for each  $i$, the functor  $F_i$ is projective, of the form $\Gamma^{d_i,V_i}$. In this case, $\phi$ is an  isomorphism since the following diagram is commutative, where the vertical  maps are induced by the Yoneda lemma ( the isomorphism \eqref{equ1.3.1a}):
	\begin{equation*}
	\xymatrix{
		\bigotimes\limits_{i=1}^n\Homx{\Pcalx{d_i}}{\Gamma^{d_i,V_i},G_i}\ar[rr]\ar[d]^\simeq&&\Homx{\Pcalx{\dunder}}{\boxtimesx{i=1}{n}\Gamma^{d_i,V_i},\boxtimesx{i=1}{n}G_i}\ar[d]^\simeq\\
		\bigotimes\limits_{i=1}^nG_i(V_i)\ar@{=}[rr]&&\left(\boxtimesx{i=1}{n}G_i\right)(V_1,\ldots,V_n).
	}
	\end{equation*}
	We obtain the isomorphism \eqref{equ1.3.2a} for all $\Hom$. Let $P^i$ be a  projective resolution of $F_i$ in $\Pcalx{d_i}$, $i=1,\ldots,n$. Since the functor $\boxtimes$ is exact in each variable and preserves the projective ones, the complex  $\boxtimesx{i=1}{n}P^i$ is a  projective resolution of $\boxtimesx{i=1}{n}F_i$ in $\Pcalx{\dunder}$.
	Moreover, the isomorphism $\phi$ induces an isomorphism of complexes:
	\begin{equation*}
	\Homx{\Pcalx{\dunder}}{\boxtimesx{i=1}{n}P^i,\boxtimesx{i=1}{n}G_i}\simeq\bigotimes_{i=1}^n\Homx{\Pcalx{d_i}}{P^i,G_i}.
	\end{equation*}
	By taking homology, we obtain the result.
\end{proof}

\subsection{ Adjoint functors and $\Ext$-groups}~

Let $n_1,n_2$ be two positive integer and $d$ be a natural number. Let $\alpha_1:\VcalK^{\times n_1}\rightleftarrows\VcalK^{\times n_2}:\alpha_2$ be an adjoint pair, where $\alpha_1,\alpha_2$ are $\Bbbk$-linear functors. 
Then we deduce an adjoint pair  $\Gamma^d\alpha_1:\Gamma^d\left(\VcalK^{\times n_1}\right)\rightleftarrows\Gamma^d\left(\VcalK^{\times n_2}\right):\Gamma^d\alpha_2$.
According to \cite[Lemma 1.3]{Pir03}, we obtain an adjoint pair
\begin{equation*}
\left(\Gamma^d\alpha_2\right)^*:\Pcalx{d}(n_1)\rightleftarrows\Pcalx{d}(n_2):\left(\Gamma^d\alpha_1\right)^*.
\end{equation*}
Let $F_1\in\Pcalx{d}(n_1)$ and $F_2\in\Pcalx{d}(n_2)$. The functors $F_1\circ \Gamma^d\alpha_2$ and $F_2\circ \Gamma^d\alpha_1$ are simply denoted by $F_1\circ\alpha_2$ and $F_2\circ\alpha_1$. We have an isomorphism natural in $F_1,F_2$:
\begin{equation*}
\Homx{\Pcalx{d}(n_2)}{F_1\circ\alpha_2,F_2}\simeq \Homx{\Pcalx{d}(n_1)}{F_1,F_2\circ\alpha_1}.
\end{equation*}
Moreover, by \cite[Lemma 1.4]{Pir03}, we have a graded isomorphism natural 
\begin{equation*}
\Extx{\Pcalx{d}(n_2)}{F_1\circ\alpha_2,F_2}\simeq \Extx{\Pcalx{d}(n_1)}{F_1,F_2\circ\alpha_1}.
\end{equation*}
Now we give the two particular cases.
\begin{enumerate}
	\item Let $n$ be a positive integer. Denote by $\Delta_n$ the functor $\VcalK\to\VcalK^{\times n},V\mapsto (V,\ldots,V)$ and by  $\boxplus^n$ the functor $\VcalK^{\times n}\to \VcalK,(V_1,\ldots,V_n)\mapsto \bigoplus_{i=1}^nV_i$. {They are}  $\Bbbk$-linear functors. We have two adjoint pairs $(\Delta_n,\boxplus^n)$ and $(\boxplus^n,\Delta_n)$. Therefore, we obtain graded isomorphisms  natural in $F\in\PcalK$ and $G\in\PcalKn$:
	\begin{align*}
	\Extx{\PcalKn}{F\circ\boxplus^n,G}&\simeq\Extx{\PcalK}{F,G\circ\Delta_n},\\ \Extx{\PcalK}{G\circ\Delta_n,F}&\simeq\Extx{\PcalKn}{G,F\circ\boxplus^n}.
	\end{align*}
	\item Let $n$ be a positive integer,  $\dunder$ be a $n$-tuple of natural numbers and  $\Vunderline=(V_1,\ldots,V_n)$ be a $n$-tuple of objects of $\VcalK$. Denote by $\Vunderline\otimes\text{-}$ the functor $\VcalK^{\times n}\to\VcalK^{\times n},\Wunderline\mapsto (V_1\otimes W_1,\ldots,V_n\otimes W_n)$. We get an adjoint pair $\Vunderline^\vee\otimes\text{-}:\VcalK^{\times n}\rightleftarrows\VcalK^{\times n}:\Vunderline\otimes\text{-}$. Hence, we obtain a graded isomorphism  natural in $F,G\in\Pcalx{\dunder}$:
	\begin{equation*}
	\Extx{\Pcalx{\dunder}}{F^{\Vunderline},G}\simeq\Extx{\Pcalx{\dunder}}{F,G_{\Vunderline}}.
	\end{equation*}
	
\end{enumerate}

\subsection{ Frobenius twist}~
Since $\Bbbk$ is a field of prime characteristic $p$, the Frobenius morphism  $\phi:\Bbbk\to\Bbbk,a\mapsto a^p$ is a ring homomorphism. Denote by $\Bbbk_\phi$ 
the $\Bbbk$-bimodule that coincides with $\Bbbk$
as left canonical $ \Bbbk $-module and the right scalar multiplication by $ \lambda $ is given by $a\cdot \lambda=a\phi(\lambda)=a\lambda^p$.
If $V$ is a $\Bbbk$-vector space, we define the $\Bbbk$-vector space $V^{(1)}$ to be the tensor product $\Bbbk_\phi\otimes V$. We obtain a $\Bbbk$-linear functor
$$(\text{-})^{(1)}=\Bbbk_\phi\otimes\text{-} :\VcalK\to\VcalK.$$
Since $\Bbbk$ is a field, this functor is exact.
If $v\in V$, we denote by $v^{(1)}$ the element  $1\otimes v\in V^{(1)}$. Then $V^{(1)}$ is the  $\Bbbk$-vector space generated by the set  $\Ngoacn{v^{(1)}:v\in V}$ quotienting by the relations, where $\lambda\in\Bbbk $ and $v,w\in V$:
\begin{equation*}
v^{(1)}+w^{(1)}=(v+w)^{(1)},\qquad (\lambda v)^{(1)}=\lambda^pv^{(1)}.
\end{equation*}
We give some basic properties of Frobenius twists \cite[page 212]{FS97}.
\begin{pro} Let $V,W$ be two finite dimensional $\Bbbk$-vector spaces.
	\begin{enumerate}
		\item If $V$ is a $\Bbbk$-vector space with basis $\Ngoacn{v_1,\ldots,v_n}$, then $V^{(1)}$ is a $\Bbbk$-vector space with basis $\Ngoacn{v_1^{(1)},\ldots,v_n^{(1)}}$.
		\item There exists a unique $\Bbbk$-linear map $V^{(1)}\otimes W^{(1)}\to (V\otimes W)^{(1)}$ 
		that sends  $v^{(1)}\otimes w^{(1)}$ to $(v\otimes w)^{(1)}$ for $v\in V$ and $w\in W$ and this map is {an isomorphism.}
		\item There exists a unique $\Bbbk$-linear map  $$\Big(\Hom(V,W)\Big)^{(1)}\to\Hom\left(V^{(1)},W^{(1)}\right)$$ that sends  $a\otimes f$ to  $ \left(a\Id_{\Bbbk_\phi}\right)\otimes f$ and this map is an isomorphism. In particular, we have an isomorphism
		$\left(V^\vee\right)^{(1)}\xrightarrow{\simeq}\left(V^{(1)}\right)^\vee$.
	\end{enumerate}
\end{pro}

The Frobenius twist is related to the symmetric algebra.
\begin{pro}
	Let $V$ be a finite dimensional $\Bbbk$-vector space. There is a $\Bbbk$-linear map $V^{(1)}\to S^p(V)$ natural in $V$, that associates to each $v^{(1)}\in V^{(1)}$ the $v^p\in  S^p(V)$ and this  map is a monomorphism.
\end{pro}
\begin{proof}
	Since $p=0$ in $\Bbbk$, we have relations in $S^p(V)$
	$$(v+w)^p=v^p+w^p,\qquad (\lambda v)^p=\lambda^pv^p.$$
	Hence, there is a  unique $\Bbbk$-linear map  $f_V:V^{(1)}\to S^p(V)$ that sends  $v^{(1)}$ to $v^p$ and this map is natural in $V$.
	It is clear that $f_\Bbbk$ is an isomorphism.
	Moreover, if $V=V_1\oplus V_2$, then the map $f_V$ is the composition
	$$V^{(1)}=V_1^{(1)}\oplus V_2^{(1)}\xrightarrow{f_{V_1}\oplus f_{V_2}}S^p(V_1)\oplus S^p(V_2)\hookrightarrow S^p(V).$$
	We then deduce, by induction on the dimension of $V$, that $f_V$ is a monomorphism.
\end{proof}

Let $V$ be a finite dimensional $\Bbbk$-vector space. The inclusion $V^{(1)}\hookrightarrow S^p(V),  v^{(1)}\mapsto v^p$ induces an algebra homomorphism 
$S^*\left(V^{(1)}\right)\to S^*(V)$.
We then obtain a monomorphism $S^n\left(V^{(1)}\right)\hookrightarrow S^{pn}(V),
v_1^{(1)}\cdots v_n^{(1)}\mapsto v_1^p\cdots v_n^p$.
Dually, we have an epimorphism $\Gamma^{pn}(V)\twoheadrightarrow \Gamma^n\left(V^{(1)}\right)$ natural in $ V $.

\begin{deff}
	The \emph{Frobenius twist} $ I^{(1)} $ is the strict polynomial functor in $\PcalKoper{p}$ that associates to each $V\in\Gamma^p\VcalK$ the  $\Bbbk$-vector space  $V^{(1)}$ and the structure morphisms  $\left(I^{(1)}\right)_{V,W}$ are defined to be the compositions:
	$$\Gamma^d\HomK{V,W}\twoheadrightarrow \left(\HomK{V,W}\right)^{(1)}\simeq\HomK{V^{(1)},W^{(1)}}.$$
\end{deff}

\begin{deff} Let $\dunder$ be a $n$-tuple of natural numbers and $F\in\PcalKoper{\dunder}$. We define a strict  polynomial functor $F^{(1)}$ in  $\Pcalx{p\dunder}$ to be the composition 
	$$F^{(1)}=F\circ \left(\underbrace{I^{(1)},\ldots,I^{(1)}}_{n}\right).$$
	By repeating this process $r$ times, we obtain the $r$-th  Frobenius twist of the functor $F$ defined by induction
	$F^{(r)}=F^{(r-1)}\circ \left(I^{(1)},\ldots,I^{(1)}\right)$.
\end{deff}

\subsection{ Symmetric monoidal structure of $\Pcalx{\dunder}$}
\label{Subs1.6}

In this section, we define a canonical closed $\Bbbk$-linear symmetric monoidal structure on the category $\Pcalx{\dunder}$, which is a generalization of one variable strict polynomial functors in Krause \cite{Krause13}.
The key in this construction is the Yoneda embedding
\begin{equation*}
\begin{array}{ccc}
\Gamma^{\dunder}\VcalK^\Op&\to&\Pcalx{\dunder}\\
\Vunderline&\mapsto &\Gx{\dunder,}{\Vunderline}=\Gamma^{d_1,V_1}\boxtimes\cdots\boxtimes\Gamma^{d_n,V_n}.
\end{array}
\end{equation*}
Thanks to this embedding, we can see $\Gamma^{\dunder}\VcalK^\Op$ as a full subcategory of $\Pcalx{\dunder}$, moreover, it is dense in $\Pcalx{\dunder}$, i.e., every object in $\Pcalx{\dunder}$ is canonically a colimit of objects in $\Gamma^{\dunder}\VcalK^\Op$.
Therefore, based on a result of Day \cite{Day70}, the canonical closed $\Bbbk$-linear symmetric monoidal structure on $\Pcalx{\dunder}$ is obtained by defining it on $\Gamma^{\dunder}\VcalK^\Op$ then extending it on $\Pcalx{\dunder}$.
For details, we proceed as follows.

The canonical closed $\Bbbk$-linear symmetric monoidal structure on $\VcalK$ induces a closed $\Bbbk$-linear  symmetric  monoidal structure on $\Gamma^d\VcalK$ whose monoidal product $\otimesx{\Gamma^d\VcalK}$ and the $\Hom$-internal $\mathbf{Hom}_{\Gamma^d\VcalK}$ are given on the objects by the formulas

\begin{equation*}
V\otimesx{\Gamma^d\VcalK}W:=V\otimes W,\qquad \Homxin{\Gamma^d\VcalK}{V,W}:=\HomK{V,W}.
\end{equation*}
We have an isomorphism
\begin{equation*}
\Homx{\Gamma^d\VcalK}{V_1\otimesx{\Gamma^d\VcalK}V_2,V_3}\simeq\Homx{\Gamma^d\VcalK}{V_1,\Homxin{\Gamma^d\VcalK}{V_2,V_3}}
\end{equation*}
natural in $V_1,V_2$ and $V_3$ of $\Gamma^d\VcalK$.
The functor $\gamma_d:\VcalK\to\Gamma^d\VcalK$ is a strict monoidal functor, i.e., there are isomorphisms 
$\gamma_d(\Bbbk)\simeq\Bbbk$ and $\gamma_d(V\otimes W)\simeq \gamma_d(V)\otimesx{\Gamma^d\VcalK}\gamma_d(W)$
natural in $V,W\in\VcalK$.

Let $\dunder=(d_1,\ldots,d_n)$ be a $n$-tuple of natural numbers.
The  $\Bbbk$-linear monoidal structures on the categories $\Gamma^{d_i}\VcalK$ induce a $\Bbbk$-linear monoidal structure on $\Gamma^{\dunder}\VcalK=\bigotimes_{i=1}^n\Gamma^{d_i}\VcalK$.
Let $\Vunderline$ and $\Wunderline$ be two objects of $\Gamma^{\dunder}\VcalK$, we define
\begin{align*}
\Vunderline\otimesx{\Gamma^{\dunder}\VcalK}\Wunderline&=\left(V_1\otimesx{\Gamma^{d_1}\VcalK}W_1,\ldots,V_n\otimesx{\Gamma^{d_n}\VcalK}W_n\right),\\
\Homxin{\Gamma^{\dunder}\VcalK}{\Vunderline,\Wunderline}&=\left(\Homxin{\Gamma^{d_1}\VcalK}{V_1,W_1},\ldots,\Homxin{\Gamma^{d_n}\VcalK}{V_n,W_n}\right).
\end{align*}
In the following, for simplicity, we write  $\otimes,\Hom$ respectively for $\otimesx{\Gamma^{\dunder}\VcalK}$ and $\mathbf{Hom}_{\Gamma^{\dunder}\VcalK}$.

Using the Yoneda embedding $\Gamma^{\dunder}\VcalK^\Op\hookrightarrow\Pcalx{\dunder}$,
the symmetric monoidal structure of $\Gamma^{\dunder}\VcalK^\Op$ can be extended over the category $\Pcalx{\dunder}$. 
This result is due to Day \cite{Day70}, and was stated in the case of strict polynomial functors in one variable by Krause \cite{Krause13}.
The symmetric monoidal product obtained on $\Pcalx{\dunder}$ is called convolution product of Day, denoted by $\otimesx{\Pcalx{\dunder}}$. The following statement recapitulates the principle properties of the closed symmetric monoidal structure obtained.

\begin{thm}[\text{\cite{Day70,Krause13}}]\label{pro2.3.2}
	There exists a closed symmetric monoidal category 
	\begin{equation*}
	\left(\Pcalx{\dunder},\otimesx{\Pcalx{\dunder}},\Gx{\dunder}{},\mathbf{Hom}_{\Pcalx{\dunder}}\right)
	\end{equation*}
	such that the Yoneda embedding
	$\Gamma^{\dunder}\VcalK^\Op\to\Pcalx{\dunder} $
	is a symmetric monoidal functor. 
	The monoidal product $\otimesx{\Pcalx{\dunder}}$ is characterized by the following properties.
	\begin{enumerate}
		\item[\rm (1)] The bifunctor $\otimesx{\Pcalx{\dunder}}$ is right exact  in each variable.
		\item[\rm (2)] There is an isomorphism $F\otimesx{\Pcalx{\dunder}}\Gx{\dunder,}{\Vunderline}\simeq F^{\Vunderline}$, natural in $F$ and $\Vunderline$.
	\end{enumerate}
	Similarly, the \emph{$\Hom$-internal} $\mathbf{Hom}_{\Pcalx{\dunder}}$ is characterized by the two following properties.
	\begin{enumerate}
		\item[\rm (1)] The bifunctor $\mathbf{Hom}_{\Pcalx{\dunder}}$ is left exact in each variable.
		\item[\rm (2)] There is an isomorphism $\mathbf{Hom}_{\Pcalx{\dunder}}(\Gx{\dunder,}{\Vunderline},F)\simeq F_{\Vunderline}$, natural in $F$ and $\Vunderline$.
	\end{enumerate}
\end{thm}

These tensor internal and $\Hom$-internal functors play an important role throughout this paper.

We denote by $\otimesxL{\Pcalx{\dunder}}$ (resp. $\Rbf\mathbf{Hom}_{\Pcalx{\dunder}}$) the total left (resp. right) derived functor of the functor $\otimesx{\Pcalx{\dunder}}$ (resp. $\mathbf{Hom}_{\Pcalx{\dunder}}$).
We denote by $\mathbf{Ext}^i_{\Pcalx{\dunder}}$ the $i$-th cohomology of $\Rbf\mathbf{Hom}_{\Pcalx{\dunder}}$.
The following proposition gives some explicit formulas of the $\Hom$-internal and  the convolution product.

\begin{pro}\label{pro2.3.8}
	Let $F,G\in\Pcalx{\dunder}$ and $\Vunderline\in\Gamma^{\dunder}\VcalK$.
	Then, we have isomorphisms, natural in $F,G$ and $\Vunderline$
	\begin{eqnarray}
	\Homxin{\Pcalx{\dunder}}{F,G}(\Vunderline)&\simeq&\Homx{\Pcalx{\dunder}}{F^{\Vunderline},G},\label{iso2.3.1}\\
	\Extxxin{i}{\Pcalx{\dunder}}{F,G}(\Vunderline)&\simeq& \Ext_{\Pcalx{\dunder}}^i\left(F^{\Vunderline},G\right)\label{iso2.3.6},\\
	\Homxin{\Pcalx{\dunder}}{F,G}&\simeq&\left(F\otimesx{\Pcalx{\dunder}}G^\sharp\right)^\sharp,\label{iso2.3.2a}\\
	\Rbf\Homxin{\Pcalx{\dunder}}{F,G}&\simeq &\left(F\otimesxL{\Pcalx{\dunder}}G^\sharp\right)^\sharp\label{iso2.3.4}.
	\end{eqnarray}
\end{pro}
\begin{proof}
	We consider each side of the isomorphism \eqref{iso2.3.1} as a functor of the variable $F$.
	By Yoneda's lemma and the characterization of the $\Hom$-internal given in Theorem \ref{pro2.3.2}, these two functors are isomorphic (natural in $G$) on the full subcategory of projective ones of $\Pcalx{\dunder}$.
	Since the functors are left exact, the isomorphism extends to an isomorphism, natural in $F$, $G$ and $\Vunderline$.
	We obtain the isomorphism \eqref{iso2.3.1}.
	This isomorphism implies the isomorphism \eqref{iso2.3.6}.

	Since the functor $\mathbf{Hom}_{\Pcalx{\dunder}}$ is left exact and the functor $\otimesx{\Pcalx{\dunder}}$ is right exact, in order to prove the isomorphism \eqref{iso2.3.2a}, without loss of generality, we may assume that $F=\Gx{\dunder,}{\Vunderline}$.
	By Theorem \ref{pro2.3.2} and the isomorphism \eqref{iso2.3.1}, we have isomorphisms 
	\begin{equation*}
	\Homxin{\Pcalx{\dunder}}{\Gx{\dunder,}{\Vunderline},G}\simeq G_{\Vunderline}\simeq \left(\Gx{\dunder,}{\Vunderline}\otimesx{\Pcalx{\dunder}}G^\sharp\right)^\sharp,
	\end{equation*}
	which proves \eqref{iso2.3.2a}.
	This isomorphism implies the isomorphism \eqref{iso2.3.4}.
\end{proof}

\section{ $\Ext$-groups of the form $\Extx{\PcalKn}{F,G^{(r)}}$}

In this section, we will study the effect of the characteristic of the coefficient field on $ \Ext $-groups in the category $ \Pcalx{\dunder} $.
Since the information concerning the characteristic is encoded in the Frobenius twist, we study the $ \Ext $-groups of the form $ \Extx{\PcalKn}{F^{(r)},G^{(r)}}$ where $ F,G\in\PcalKn$.
More generally, we consider the $ \Ext $-groups of the form 
\begin{equation}\label{equ2.1}
\Extx{\PcalKn}{F,G^{(r)}}\quad\text{and}\quad \Extx{\PcalKn}{F^{(r)},G}
\end{equation}
where $ F,G\in\PcalKn $.
However, by using Kuhn duality, we know that
\begin{equation*}
\Extx{\PcalK}{F^{(r)},G}
\simeq \Extx{\PcalK}{G^\sharp,\left(F^{(r)}\right)^\sharp}
\simeq
\Extx{\PcalK}{G^\sharp,\left(F^{\sharp}\right)^{(r)}}
\end{equation*}
natural in $F$ and $G$.
Therefore, we will focus on the $ \Ext $-groups of the form $ \Extx{\PcalKn}{F,G^{(r)}} $.
The key idea to study these $\Ext$-groups is to use the adjoint to the precomposition with the $r$-th Frobenius twist.
This is Cha\l upnik's idea \cite{Chalupnik15}.

\subsection{ Left adjoint to precomposition with Frobenius twist}

Let $r$ be a fixed natural number.
The precomposition with the $r$-th Frobenius twist $I^{(r)}$ yields an exact functor
\begin{equation*}
\begin{array}{cccc}
\Fr_r:&\Pcalx{\dunder}&\to&\Pcalx{p^r\dunder}\\
&F&\mapsto&F^{(r)}=F\circ I^{(r)}.
\end{array}
\end{equation*}
In the following proposition, we show that $\Fr_r$ is an embedding functor.
This proposition is a version of strict polynomial functors in several variables of \cite[Lemma 2.2]{Tou12}.

\begin{pro}\label{pro2.3.5a}
	Let $F,G\in\Pcalx{\dunder}$ and $H\in\Pcalx{p^r\dunder}$. 
	There are isomorphisms
	\begin{eqnarray}
	\Homx{\Pcalx{\dunder}}{F,G}&\xrightarrow{\simeq}&\Homx{\Pcalx{p^r\dunder}}{F^{(r)},G^{(r)}}\label{equ2.3.1},\\
	\Homx{\Pcalx{p^r\dunder}}{H\otimesx{\Pcalx{p^r\dunder}} \Gx{\dunder(r)}{},G^{(r)}}&\xrightarrow{\simeq}& \Homx{\Pcalx{p^r\dunder}}{H,G^{(r)}}\label{iso2.3.2},
	\end{eqnarray}
	natural in $F,G$ and $H$, where the former is induced by Frobenius twist and the latter is induced by the composition  $H\xrightarrow{\simeq}H\otimesx{\Pcalx{p^r\dunder}}\Gx{p^r\dunder}{}\to H\otimesx{\Pcalx{p^r\dunder}} \Gx{\dunder(r)}{}$.
\end{pro}

\begin{proof}
	By left exactness of functors $\Hom_{\Pcalx{\dunder}}(\text{-},\text{-})$ and $\Hom_{\Pcalx{p^r\dunder}}(\text{-},\text{-})$, it suffices to prove that the morphisms \eqref{equ2.3.1} and \eqref{iso2.3.2} are isomorphisms as $F$ and $H$ are respectively standard projective generators of $\Pcalx{\dunder}$ and $\Pcalx{p^r\dunder}$, 
	this means $F$ and $H$ are of the form $F=\Gx{\dunder,}{\Vunderline^{(r)}}$ and $H=\Gx{p^r\dunder,}{\Vunderline}$. 
	We have $\Gx{p^r\dunder,}{\Vunderline}\otimesx{\Pcalx{p^r\dunder}} \Gx{\dunder(r)}{}\simeq \left(\Gx{\dunder(r)}{}\right)^{\Vunderline}$ and $\left(\Gx{\dunder,}{\Vunderline^{(r)}}\right)^{(r)}\simeq \left(\Gx{\dunder(r)}{}\right)^{\Vunderline}$. 
	Thus we just need to show that the maps
	\begin{align*}
	\Homx{\Pcalx{\dunder}}{\Gx{\dunder,}{\Vunderline^{(r)}},G}&\to\Homx{\Pcalx{p^r\dunder}}{\left(\Gx{\dunder(r)}{}\right)^{\Vunderline},G^{(r)}},\\
	\Homx{\Pcalx{p^r\dunder}}{\left(\Gx{\dunder(r)}{}\right)^{\Vunderline},G^{(r)}}&\to \Homx{\Pcalx{p^r\dunder}}{\Gx{p^r\dunder,}{\Vunderline},G^{(r)}}.
	\end{align*}
	are isomorphisms.
	By left exactness of $\Hom_{\Pcalx{p^r\dunder}}$, the second map is a monomorphism. 
	Hence, it is left to show that the composition 
	\begin{equation*}
	\Homx{\Pcalx{\dunder}}{\Gx{\dunder,}{\Vunderline^{(r)}},G}\to \Homx{\Pcalx{p^r\dunder}}{\Gx{p^r\dunder,}{\Vunderline},G^{(r)}}
	\end{equation*}
	of these morphisms is an isomorphism. 
	It follows from the commutativity of the diagram
	\begin{equation*}
	\xymatrix{\Homx{\Pcalx{\dunder}}{\Gx{\dunder,}{\Vunderline^{(r)}},G}
		\ar[r]\ar[d]^\simeq & \Homx{\Pcalx{p^r\dunder}}{\Gx{p^r\dunder,}{\Vunderline},G^{(r)}}\ar[d]^\simeq\\ 
		G\left(V^{(r)}\right)\ar@{=}[r]& G^{(r)}(V)
	}
	\end{equation*}
	where the two vertical morphisms are defined by Yoneda's lemma.
	This completes the proof.
\end{proof}

\begin{cor}\label{cor2.3.8}
	Let $F$ and $G$ be two objects of the category $\Pcalx{\dunder}$ or two chain complexes in $\Pcalx{\dunder}$.
	Then $F\simeq G$ if and only if $F^{(r)}\simeq G^{(r)}$.
\end{cor}

\subsection*{Functors $ \ell^{r} $}

In the following, we consider the adjoint to the functor $ \Fr_r $.
Since we only focus on the $ \Ext $-groups of the form $ \Extx{\PcalKn}{F,G^{(r)}} $, only the left adjoint to $ \Fr_r $ is studied.

Following Cha\l upnik \cite[Section 2]{Chalupnik15} (see also \cite[page 548]{Tou13a}), we introduce the functor $\ell^r:\PcalKoper{p^r\dunder}\to\Pcalx{\dunder}$.
In Proposition \ref{pro2.2.11}, we will point out that $ \ell^{r} $ is the left adjoint to  the functor $ \Fr_r $.

\begin{deff}\label{def2.5.5}
	Let $\dunder=(d_1,\ldots,d_n)$ be a $n$-tuple of natural numbers and $r$ be a natural number. We define a functor $\ell^r:\PcalKoper{p^r\dunder}\to\Pcalx{\dunder}$
	by the formula
	\begin{equation*}
	\ell^r(F)(\Vunderline)=\left( \Homx{\Pcalx{p^r\dunder}}{F,\left(\Sx{\dunder}{\Vunderline^{\vee}}\right)^{(r)}}\right)^{\vee},
	\end{equation*}
	where $F\in\Pcalx{p^r\dunder}$, $\Vunderline$ is a $n$-tuple of objects of $\VcalK$ and $\Sx{\dunder}{\Vunderline}$ denotes the product $\boxtimesx{i=1}{n}S^{d_i}_{V_i}\in\Pcalx{\dunder}$.
\end{deff}

According to Kuhn duality, we obtain the equality
\begin{equation*}
\ell^r(F)^\sharp(\Vunderline)= \Homx{\Pcalx{p^r\dunder}}{F,\left(\Sx{\dunder}{\Vunderline}\right)^{(r)}}
\end{equation*}
natural in $ F $ and $ \Vunderline $.

\begin{pro}\label{lem2.6.12}
	Let $F\in \Pcalx{p^r\dunder}$. 
	There is an isomorphism, natural in $F$:
	\begin{equation}\label{iso2.4.3a}
	\ell^r(F)^{\sharp(r)}\simeq \Homxin{\Pcalx{p^r\dunder}}{F,S^{\boxtimes\dunder(r)}}.
	\end{equation}
\end{pro}

\begin{proof}
	Let $\Vunderline$ be a $n$-tuple of objects of $\VcalK$. 
	By definition and the isomorphism $\left(\Sx{\dunder}{\Vunderline^{(r)}}\right)^{(r)}\simeq \left(S^{\boxtimes\dunder(r)}\right)_{\Vunderline}$, there are isomorphisms
	\begin{align*}
	\ell^r(F)^{\sharp(r)}(\Vunderline)=\ell^r(F)^{\sharp}\left(\Vunderline^{(r)}\right)& = \Homx{\Pcalx{p^r\dunder}}{F,\left(\Sx{\dunder}{\Vunderline^{(r)}}\right)^{(r)}}\\
	&\simeq  \Homx{\Pcalx{p^r\dunder}}{F,\left(S^{\boxtimes\dunder(r)}\right)_{\Vunderline}}=\Homxin{\Pcalx{p^r\dunder}}{F,S^{\boxtimes\dunder(r)}}(\Vunderline)
	\end{align*}
	natural in $F,\Vunderline$,	which establish the formula.
\end{proof}

\begin{pro}\label{pro2.2.11}
	The functor $\ell^r:\Pcalx{p^{r}\dunder}\to\Pcalx{\dunder}$ is the left adjoint to $\Fr_r:\Pcalx{\dunder}\to\Pcalx{p^{r}\dunder}$, i.e., there is an isomorphism
	\begin{eqnarray}
	\Homx{\Pcalx{\dunder}}{\ell^r(F),G}&\simeq&\Homx{\Pcalx{p^r\dunder}}{F,G^{(r)}},\label{equ2.4.4}
	\end{eqnarray}
	natural in $F\in\Pcalx{p^r\dunder}$ and $G\in\Pcalx{\dunder}$.
	Moreover, we have
	\begin{eqnarray}
	\Homxin{\Pcalx{\dunder}}{\ell^r(F),G}^{(r)}&\simeq&\Homxin{\Pcalx{p^r\dunder}}{F,G^{(r)}}.\label{iso2.4.5a}
	\end{eqnarray}
	natural in $F\in\Pcalx{p^r\dunder}$ and $G\in\Pcalx{\dunder}$.
\end{pro}

\begin{proof}
	Define $F\to\ell^r(F)^{(r)}$ to be the composition
	\begin{equation}\label{equ2.7}
	F\xrightarrow{\simeq} F\otimesx{\Pcalx{p^r\dunder}}\Gx{p^r\dunder}{}\to F\otimesx{\Pcalx{p^r\dunder}}\Gx{\dunder(r)}{}\simeq \ell^r(F)^{(r)}
	\end{equation}
	where the second map is induced by the canonical projection $\Gx{p^r\dunder}{}\twoheadrightarrow\Gx{\dunder}{(r)}$ and the third one is the isomorphism \eqref{iso2.4.3a} defined in Proposition \ref{lem2.6.12}.
	The map \eqref{equ2.7} induces a map
	$\Homx{\Pcalx{p^r\dunder}}{\ell^r(F)^{(r)},G^{(r)}}\to\Homx{\Pcalx{p^r\dunder}}{F,G^{(r)}}$. 
	By Proposition \ref{pro2.3.5a}, this map and the one induced by the Frobenius twist $\Homx{\Pcalx{\dunder}}{\ell^r(F),G}\to$$ \Homx{\Pcalx{p^r\dunder}}{\ell^r(F)^{(r)},G^{(r)}}$ are isomorphisms.
	The isomorphism \eqref{equ2.4.4} is obtained by composing these isomorphisms. 
	Thus the functor $\ell^r$ is the left adjoint to $\Fr_r$. 
	
	Now, let $\Vunderline$ be a $n$-tuple of objects of $\VcalK$. The isomorphism \eqref{equ2.4.4} yields isomorphisms
	\begin{align*}
	\Homxin{\Pcalx{\dunder}}{\ell^r(F),G}^{(r)}(\Vunderline)
	&=\Homxin{\Pcalx{\dunder}}{\ell^r(F),G}\left(\Vunderline^{(r)}\right)\\
	&=\Homx{\Pcalx{\dunder}}{\ell^r(F),G_{\Vunderline^{(r)}}}\\&\simeq \Homx{\Pcalx{p^r\dunder}}{F,\left(G_{\Vunderline^{(r)}}\right)^{(r)}}\\
	&\simeq \Homx{\Pcalx{p^r\dunder}}{F, \left(G^{(r)}\right)_{\Vunderline}}=\Homxin{\Pcalx{p^r\dunder}}{F,G^{(r)}}(\Vunderline)
	\end{align*}
	natural in $F,G$ and $\Vunderline$, and \eqref{iso2.4.5a} is proved.
\end{proof}

\subsection*{Derived functors $ \Lbf\ell^{r} $}

Denote by $\Lbf\ell^r:\Dbfb{\Pcalx{p^r\dunder}}\to\Dbfb{\Pcalx{\dunder}}$ the total left derived functor of the functor $\ell^r$.
The isomorphism \eqref{iso2.4.3} in the following  proposition  is the derived version of the isomorphism \eqref{iso2.4.5a} in Proposition \ref{pro2.2.11}.

\begin{pro}\label{pro2.4.7}
	Let $r,s$ be two natural numbers.
	Let $F\in\Pcalx{p^r\dunder}$ and $G\in\Pcalx{\dunder}$. There are isomorphisms
	\begin{eqnarray}
	\RHomxin{\Pcalx{p^r\dunder}}{F,G^{(r)}}&\simeq&\RHomxin{\Pcalx{\dunder}}{\Lbf\ell^r(F),G}^{(r)},\label{iso2.4.3}\\
	\Lbf\ell^r(F)^{\sharp(r)}&\simeq& \RHomxin{\Pcalx{p^r\dunder}}{F,S^{\boxtimes\dunder(r)}}\label{equ2.4.5},\\
	\Lbf\ell^r\circ\Lbf\ell^s&\simeq&\Lbf\ell^{r+s}.\label{equ2.4.6}
	\end{eqnarray}
	natural in $F$ and $G$.
\end{pro}

\begin{proof}
	By Proposition \ref{pro2.2.11}, the functor $\ell^r$ is the left adjoint to the exact functor $\Fr_r$, so it preserves the projectives. 
	Moreover, the isomorphism
	\eqref{iso2.4.5a} in Proposition \ref{pro2.2.11} induces an isomorphism
	\begin{equation*}
	\Homxin{\Pcalx{p^r\dunder}}{-,G^{(r)}}\simeq\Fr_r\circ\Homxin{\Pcalx{\dunder}}{-,G}\circ\ell^r
	\end{equation*}
	natural in $G$.
	Since the functor $\ell^r$ preserves the projectives and the functor $\Fr_r$ is exact, by \cite[Theorem 18.8.2]{Weibel94}, there are isomorphisms natural in $G$:
	\begin{align*}
	\RHomxin{\Pcalx{p^r\dunder}}{-,G^{(r)}}&\simeq \Rbf\left(\Fr_r\circ\Homx{\Pcalx{\dunder}}{-,G}\circ\ell^r\right)\\
	&\simeq\Fr_r\circ \RHomx{\Pcalx{\dunder}}{-,G}\circ\Lbf\ell^r
	\end{align*}
	which prove \eqref{iso2.4.3}. 
	Replacing $G$ by $S^{\boxtimes\dunder}$ in the isomorphism  \eqref{iso2.4.3} and using Yoneda's lemma, we obtain the isomorphism \eqref{equ2.4.5}.
	
	By Proposition \ref{pro2.2.11} and the equality  $\Fr_r\circ\Fr_s=\Fr_{r+s}$, we have an isomorphism $\ell^r\circ\ell^s\simeq\ell^{r+s}$. 
	Moreover, since the functor $\ell^s$
	preserves the projectives, \cite[Theorem 10.8.2]{Weibel94} yields isomorphisms
	\begin{equation*}
	\Lbf\ell^r\circ\Lbf\ell^s\simeq \Lbf\left(\ell^r\circ\ell^s\right)\simeq \Lbf\ell^{r+s}
	\end{equation*}
	which establishes the desired conclusion.
\end{proof}

\begin{cor}\label{cor2.4.7} 
	Let $F\in\Pcalx{p^r\dunder}$ and $G\in\Pcalx{\dunder}$. 
	If the  complex  $\Lbf\ell^r(F)$ is \emph{formal}, i.e., there is an isomorphism $\Lbf\ell^r(F)\simeq H_*\Lbf\ell^r(F)$ in the derived category $\Dbfb{\Pcalx{\dunder}}$, then for all $k$ we have an isomorphism, natural in $G$:
	\begin{equation}\label{iso2.4.4a}
	\Extxxin{k}{\Pcalx{p^r\dunder}}{F,G^{(r)}}\simeq\bigoplus_{i+j=k}\Extxxin{i}{\Pcalx{\dunder}}{H_j\Lbf\ell^r(F),G}^{(r)}.
	\end{equation}
\end{cor}
\begin{proof}
	Since the complex  $\Lbf\ell^r(F)$ is formal, it is isomorphic to its homology $H_*\Lbf\ell^r(F)$ in the derived category $\Dbfb{\Pcalx{\dunder}}$.
	Then, \eqref{iso2.4.3} induces an isomorphism $\RHomxin{\Pcalx{p^r\dunder}}{F,G^{(r)}}\simeq\RHomxin{\Pcalx{\dunder}}{H_*\Lbf\ell^r(F),G}^{(r)}$, natural in $G$. 
	Taking their homology, we obtain the isomorphism \eqref{iso2.4.4a}.
\end{proof}

\subsection{ Formal class $\Formel(r,n)$}

The corollary \ref{cor2.4.7} leads us to introduce the following definition.

\begin{deff}\label{def2.4.8}
	Let $r,n$ be two natural numbers, $n>0$. 
	Denote by $\Formel(r,n)$ the class of the strict polynomial functors $F\in\PcalKn$ such that the complex  $\Lbf\ell^r(F)$ is formal.
\end{deff}

Using this definition, we have the following reformulation of Corollary \ref{cor2.4.7}.

\begin{cor}\label{cor2.4.9}
	If $F\in \Formel(r,n)$, then, for all $k$, we have an isomorphism
	\begin{equation*}
	\Extxxin{k}{\Pcalx{p^r\dunder}}{F,G^{(r)}}\simeq\bigoplus_{i+j=k}\Extxxin{i}{\Pcalx{\dunder}}{H_j\Lbf\ell^r(F),G}^{(r)}
	\end{equation*}
	natural in $G\in\PcalKn$.
\end{cor}

We will give some examples and properties of the class $\Formel(r,n)$ in the remainder of this section.

\begin{exe}\label{exe2.5.3}
	The formal class $\Formel(0,n)$ is  the class of all objects of $\PcalKn$.
	Each projective functor belongs to $\Formel(r,n)$ for all $r$.
	Furthermore, if $F\in\PcalKn$ is a homogeneous functor of degree $\dunder$ that is not a multiple of $p$, then $F\in\Formel(r,n)$ for all $r$.
\end{exe}

The following proposition gives another example of functors that belong to $\Formel(r,n)$ for all $r\ge 1$.
In this proposition, we use the notion of block, see \cite{Tou13b}.
This result is just to illustrate formal classes, is not used in the rest of this paper.

\begin{pro}
	Let $F$ be a strict polynomial functor.
	If the blocks of composition factors of the functor $F$ do not contain the trivial block, then $F$ belongs to $\Formel(r,n)$ for all $r\ge 1$.
\end{pro}
\begin{proof}
	Since the set of blocks of the functor $F$ does not contain the trivial block, there is a projective resolution $P$ of $F$ such that the set of blocks of each $P_i$ does not contain the trivial block.
	By definition $\Lbf\ell^r(F)\simeq \ell^r(P)$.
	It follows from	Proposition \ref{lem2.6.12} that
	\begin{equation*}
	\ell^r(P_i)^{\sharp(r)}=\Homxin{\PcalKn}{P_i,\Sx{\dunder(r)}{}}=0.
	\end{equation*}
	Hence $\Lbf\ell^r(F)=0$. 
	We conclude that $F\in \Formel(r,n)$ for all $r\ge 1$.
\end{proof}

The following proposition establishes the stability of $\Formel(r,n)$ with
\emph{direct sum}, \emph{tensor product} and \emph{parametrization}.

\begin{pro}\label{pro2.5.5}
	Let $F$, $G$ be elements of the formal class $\Formel(r,n)$ and $\Vunderline$ be a $n$-tuple of objects of $\VcalK$. 
	Then the functors $F_{\Vunderline},F\oplus G$ and $F\otimes G$ also belong to $\Formel(r,n)$. 
	Moreover, there are graded isomorphisms
	\begin{align*}
	H_*\Lbf\ell^r\left(F_{\Vunderline}\right)&\simeq \left(H_*\Lbf\ell^r(F)\right)_{\Vunderline^{(r)}}\\
	H_*\Lbf\ell^r(F\oplus G)&\simeq H_*\Lbf\ell^r(F)\oplus H_*\Lbf\ell^r(G)\\
	H_*\Lbf\ell^r(F\otimes G)&\simeq H_*\Lbf\ell^r(F)\otimes H_*\Lbf\ell^r(G)
	\end{align*}
	natural in $F,G,\Vunderline$.
\end{pro}

\begin{proof} Since the functors $(\text{-})_{\Vunderline},\oplus,\otimes$ are exact, they preserve the formal complexes.
	To prove this proposition, it suffices to show there are isomorphisms natural in $F,G$:
	\begin{gather}
	\Lbf\ell^r\left(F_{\Vunderline}\right)\simeq \left(\Lbf\ell^r(F)\right)_{\Vunderline^{(r)}},\quad \Lbf\ell^r(F\oplus G)\simeq\Lbf\ell^r(F)\oplus\Lbf\ell^r(G),\label{iso2.4.10}\\
	\Lbf\ell^r(F\otimes G)\simeq\Lbf\ell^r(F)\otimes\Lbf\ell^r(G).\label{iso2.4.11}
	\end{gather}
	By Corollary \ref{cor2.3.8}, Proposition \ref{lem2.6.12} and exponential formula of symmetric powers $ S^{d} $, there are isomorphisms natural in $F,G$:
	\begin{equation*}
	\ell^r\left(F_{\Vunderline}\right)\simeq\ell^r(F)_{\Vunderline^{(r)}},\;
	\ell^r(F\oplus G)\simeq \ell^r(F)\oplus\ell^r(G),\;\ell^r(F\otimes G)\simeq \ell^r(F)\otimes\ell^r(G).
	\end{equation*}
	Moreover, the functors $(\text{-})_{\Vunderline},\oplus,\otimes$ are exact and preserve the projectives. Then, we obtain the isomorphisms \eqref{iso2.4.10} and \eqref{iso2.4.11}.
\end{proof}

In order to establish the relation between formal classes and the \emph{precomposition with the diagonal functor}, we need the following lemmas.

\begin{lem}\label{lem2.5.6} Let $F\in\PcalKn$ and $G\in\PcalK$. 
	There are isomorphisms, natural in $F,G$:
	\begin{eqnarray}
	\Homxin{\PcalK}{F\circ\Delta_n,G}&\simeq &\Homxin{\PcalKn}{F,G\circ\boxplus^n}\circ\Delta_n,\label{equ2.6.3}\\
	\RHomxin{\PcalK}{F\circ\Delta_n,G}&\simeq& \RHomxin{\PcalKn}{F,G\circ\boxplus^n}\circ\Delta_n.\label{equ2.6.4}
	\end{eqnarray}
\end{lem}
\begin{proof} 
	The sum-diagonal adjunction $\text{-}\circ\Delta_n:\PcalKn\rightleftarrows\PcalK:\text{-}\circ\boxplus^n$ yields isomorphisms
	\begin{align*}
	\Homx{\PcalK}{F\circ\Delta_n,G_V}&\simeq \Homx{\PcalKn}{F,G_V\circ\boxplus^n}\\
	&\simeq  \Homx{\PcalKn}{F,\left(G\circ\boxplus^n\right)_{\Delta_n(V)}}
	\end{align*}
	natural in $F,G,V$, where $V$ is an object of $\VcalK$.
	The isomorphism \eqref{equ2.6.3} is then proved.
	This isomorphism induces an isomorphism, natural in $G$:
	\begin{equation*}
	\Homxin{\PcalK}{\text{-},G}\circ(\text{-}\circ\Delta_n)\simeq (\text{-}\circ\Delta_n)\circ\Homxin{\PcalKn}{\text{-},G\circ\boxplus^n}.
	\end{equation*}
	Since the functor $\text{-}\circ\Delta_n$ is exact and preserves projectives, by \cite[Theorem 10.8.2]{Weibel94}, there is an isomorphism, natural in $G$:
	\begin{equation*}
	\RHomxin{\PcalK}{\text{-},G}\circ(\text{-}\circ\Delta_n)\simeq (\text{-}\circ\Delta_n)\circ\RHomxin{\PcalKn}{\text{-},G\circ\boxplus^n}
	\end{equation*}
	which gives the isomorphism \eqref{equ2.6.4}.
\end{proof}

\begin{lem}\label{lem2.6.10} Let $F,G\in\PcalKn$ and $\Vunderline$ be a $n$-tuple of objects of $\VcalK$. 
	We have isomorphisms
	\begin{eqnarray}
	\Homxin{\PcalKn}{F,G}(\Vunderline)&\simeq& \Homxin{\PcalKn}{F^{\Vunderline},G}\circ\Delta_n(\Bbbk),\label{equ2.6.7}\\
	\RHomxin{\PcalKn}{F,G}(\Vunderline)&\simeq& \RHomxin{\PcalKn}{F^{\Vunderline},G}\circ\Delta_n(\Bbbk)\label{equ2.6.8}
	\end{eqnarray}
	natural in $F,G$ and $\Vunderline$.
\end{lem}
\begin{proof}
	By using the definition of the functor $\Hom$-internal and the canonical isomorphism $\Vunderline\otimes\Delta_n(\Bbbk)\simeq \Vunderline$, we obtain \eqref{equ2.6.7}. 
	Let $J$ be an injective coresolution of the functor $G$. 
	The isomorphism \eqref{equ2.6.7} induces an isomorphism of the complexes:
	\begin{equation*}
	\Homxin{\PcalKn}{F,J}(\Vunderline)\simeq \Homxin{\PcalKn}{F^{\Vunderline},J}\circ\Delta_n(\Bbbk)
	\end{equation*}
	which yields the isomorphism \eqref{equ2.6.8}.
\end{proof}

Using the two above lemmas, we can now establish a relation of the formal classes with  precomposition by the diagonal functor $\Delta_n$.

\begin{pro}\label{pro2.5.8}	
	If the functor $F\in\Pcalx{\dunder}$ belongs to the class $\Formel(r,n)$ then $F\circ\Delta_n\in\Formel(r,1)$ and $H_*\Lbf\ell^r(F\circ\Delta_n)\simeq H_*\Lbf\ell^r(F)\circ\Delta_n$.
\end{pro}
\begin{proof}
	According to Proposition \ref{lem2.6.12}, the isomorphism \eqref{equ2.6.3} of Lemma \ref{lem2.5.6} and the fact that $S^{\boxtimes\dunder}$ is the homogeneous part of degree $\dunder$ of the functor $S^d\circ\boxplus^n$ where $d=d_1+\cdots +d_n$, there are isomorphisms:
	\begin{align*}
	\ell^r(F\circ\Delta_n)^{\sharp(r)}&\simeq \Homxin{\PcalK}{F\circ\Delta_n,S^{d(r)}}\\
	&\simeq \Homxin{\PcalKn}{F,S^{d(r)}\circ\boxplus^n}\circ\Delta_n\\
	&\simeq \Homxin{\PcalKn}{F,S^{\dunder(r)}}\circ\Delta_n\\ 
	&\simeq\left(\ell^r(F)\right)^{\sharp(r)}\circ\Delta_n\\
	&\simeq\left(\ell^r(F)\circ\Delta_n\right)^{\sharp(r)}.
	\end{align*}
	Combining with Corollary \ref{cor2.3.8}, there is an isomorphism $\ell^r(F\circ\Delta_n)\simeq \ell^r(F)\circ \Delta_n$ natural in $F$. 
	Since the functor $\text{-}\circ\Delta_n:\PcalKn\to\PcalK$ is exact and preserves projectives, there is an isomorphism	$\Lbf\ell^r(F\circ\Delta_n)\simeq \Lbf\ell^r(F)\circ \Delta_n$ natural in $F$. This completes the proposition.
\end{proof}

By definition, a strict polynomial functor $F\in\PcalKn$ belongs to $\Formel(r,n)$ if and only if the complex $\Lbf\ell^r(F)$ is formal.
However, it lacks an appropriate formula to work with the complex $\Lbf\ell^r(F)$.
Instead of that, due to \eqref{equ2.4.5}, the $r$-th Frobenius twist of $\Lbf\ell^r(F)$ is isomorphic to 
$\RHomxin{\Pcalx{p^r\dunder}}{F,S^{\boxtimes\dunder(r)}}$.
This raises a question: \emph{Is a  complex $C$ in $\PcalKn$ formal if the complex $C^{(r)}$ is formal?}

\begin{thm}\label{thm:vdK} 
	Let $C\in\Ch\left(\Pcalx{\dunder}\right)$ be a \emph{bounded} complex  and $r$ be a natural number.
	Then, the complex  $C$ is formal if the complex  $C^{(r)}$ is formal.
\end{thm}

This theorem is proved in author's thesis \cite{PhamVT15}, using the classical theory of the $\mathrm{k}$-invariants given by Dold since 1960s, \cite{Dol60}.
The one variable case of this theorem was proved by W. van der Kallen \cite{vdK13}, by using spectral sequences.

\begin{pro}[\text{\cite[Theorem 2.6.12]{PhamVT15}}]\label{pro2.4.12} 
	Let $F\in\PcalKn$. 
	If there exists a graded object $C$ such that $\Lbf\ell^r(F)^{(r)}\simeq C^{(r)}$ then $F$ belongs to the class $\Formel(r,n)$ and $H_*\Lbf\ell^r(F)\simeq C$.
\end{pro}
\begin{proof} 
	Let $C$ be a graded object in $\PcalKn$. 
	Suppose that there exists an isomorphism $\Lbf\ell^r(F)^{(r)}\simeq C^{(r)}$ in the derived category. 
	Since the complex  $C^{(r)}$ is formal, by Theorem \ref{thm:vdK}, the complex  $\Lbf\ell^r(F)$ is formal. 
	Moreover, we have a graded isomorphism  $\left(H_*\Lbf\ell^r(F)\right)^{(r)}\simeq C^{(r)}$. 
	This isomorphism yields a graded isomorphism $H_*\Lbf\ell^r(F)\simeq C$ since the precomposition with Frobenius twist is an embedding of categories.
\end{proof}

The following result concerning the relation between the formality and parametrization will be used later.

\begin{pro}[\text{\cite[Proposition 2.6.14]{PhamVT15}}]\label{pro2.6.13}
	Let $C\in\Ch\left(\Pcalx{\dunder}\right)$ be a bounded complex and $r$ be a natural number. 
	The three following properties are equivalent.
	\begin{enumerate}
		\item[\rm (1)] The complex  $C$ is formal.
		\item[\rm (2)] The complex  $C_{\Vunderline}$ is formal for all $n$-tuples  $\Vunderline$ of objects of $\VcalK$.
		\item[\rm (3)] There is a $n$-tuple $\Vunderline$ of objects of $\VcalK$ such that $V_i\ne 0$ for all $i$ and the complex  $C_{\Vunderline}$ is formal.
	\end{enumerate}
\end{pro}

\section{ Study of the formal class $\Formel(r,n)$}\label{sect2.6}

In this section, we study deeper properties of the formal classes $ \Formel(r,n) $.
The main results expressing the effect of the Frobenius twist on the $\Ext$-groups are isomorphisms \eqref{equ3.1a}, \eqref{equ3.7} and \eqref{equ3.8} in Theorems \ref{thm2.7.3} and \ref{thm2.6.8} as follows.

\subsection{ Compatibility of formal class $\Formel(r,n)$ with the precomposition by Frobenius twist}~

We denote by $E_r$ the finite dimensional graded vector space which equals $\Bbbk$ in degrees $0,2,\ldots,2p^r-2$ and which is zero in the other degrees.
By \cite[Theorem 4.5]{FS97}, we have an isomorphism between graded vector spaces
\begin{equation*}\label{equstar}
\Extx{\Pcalx{p^{r}}}{I^{(r)},I^{(r)}}\simeq E_r.\tag{$ \star $}
\end{equation*}
This important result of Friedlander and Suslin is the starting point for all subsequent studies on Frobenius twist and $\Ext$-groups.

The main result in this subsection is the following theorem.

\begin{thm}\label{thm2.7.3}
	Let $F,G\in\PcalKn$. 
	\begin{enumerate}
		\item[\rm (1)] The functor $F^{(r)}$ belongs to $\Formel(r,n)$ and $H_*\Lbf\ell^r\left(F^{(r)}\right)\simeq F^{\Delta_n(E_r)}$.
		\item[\rm (2)] There is an isomorphism natural in $G$
		\begin{eqnarray}\label{equ3.1a}
		\Extxin{\PcalKn}{F^{(r)},G^{(r)}}&\simeq& \Extxin{\PcalKn}{F^{\Delta_n(E_r)},G}^{(r)}.
		\end{eqnarray}
		\item[\rm (3)] If $F$ belongs to $\Formel(r,n)$ then $F^{(1)}$ belongs to $\Formel(r+1,n)$. Moreover, there is an isomorphism $H_*\Lbf\ell^{r+1}\left(F^{(1)}\right)\simeq H_*\Lbf\ell^r(F)^{\Delta_n\left(E_1^{(r)}\right)}$.
		\item[\rm (4)] Conversely, if the functor $F^{(1)}$ belongs to $\Formel(r+1,n)$, then $F$ belongs to $\Formel(r,n)$.
	\end{enumerate}
\end{thm}

The isomorphism \eqref{equ3.1a} is an isomorphism between strict polynomial functors. 
Computing its value at $ \Bbbk $ gives an isomorphism between graded vector spaces below
\begin{eqnarray}\label{equ3.1b}
\Extx{\PcalKn}{F^{(r)},G^{(r)}}&\simeq& \Extx{\PcalKn}{F^{\Delta_n(E_r)},G}.
\end{eqnarray}
In the isomorphism \eqref{equ3.1b}, if we replace $ F=G=I\in\Pcalx{1} $, then we obtain the isomorphism \eqref{equstar} of Friedlander-Suslin.

In order to prove Theorem \ref{thm2.7.3}, we recall the following important theorem of Touz\'e, see \cite[Lemma 4.4 and Theorem 4.6]{Tou12} (see also \cite[Proposition 13]{Tou13a}). 

\begin{thm}[Touz\'e]\label{thm2.6.1} Let $F\in\Pcalx{d}$. There is an isomorphism natural in $F$:
	\begin{equation*}
	\RHomxin{\PcalK}{F^{(r)},S^{d(r)}}\simeq \left(F^{E_r}\right)^{\sharp(r)}.
	\end{equation*}
\end{thm}

\begin{proof}	
	For each vector space $ W $, the functor $ \left(S^{d}_{W}\right)^{(r)} $ has injective resolution $ T(W,r)^{*} $, called Troesch coresolutions, such that the complex $ \Homx{\PcalK}{F^{(r)},T(W,r)^{*}} $ is concentrated in even degrees.
	Therefore, this complex is formal, and isomorphic to its homology. This homology was determined in \cite[Theorem 4.6]{Tou12}.
	We then obtain isomorphisms:
	\begin{align*}
	\RHomx{\PcalK}{F^{(r)},\left(S^{d}_{W}\right)^{(r)}}&\simeq \Homx{\PcalK}{F^{(r)},T(W,r)^{*}}\\
	&\simeq H_{*}\Homx{\PcalK}{F^{(r)},T(W,r)^{*}}\\
	&\simeq \Extx{\PcalK}{F^{(r)},\left(S^{d}_{W}\right)^{(r)}}\\
	&\simeq \Homx{\PcalK}{F^{E_r},S^{d}_{W}}\\
	&\simeq \left(F^{E_r}\right)^{\sharp}(W).
	\end{align*}
	natural in $ F $ and $ W $.
	Therefore,
	\begin{align*}
	\RHomxin{\PcalK}{F^{(r)},S^{d(r)}}(V)&=\RHomx{\PcalK}{F^{(r)},\left(S^{d(r)}\right)_{V}}\\
	&\simeq\RHomx{\PcalK}{F^{(r)},\left(S^{d}_{V^{(r)}}\right)^{(r)}}\\
	&\simeq\left(F^{E_r}\right)^{\sharp}\left(V^{(r)}\right)\\
	&\simeq \left(F^{E_r}\right)^{\sharp(r)}(V)
	\end{align*}
	natural in $ V $.
	Thus we obtain the desired isomorphism.
\end{proof}

\begin{cor}\label{cor2.7.2}
	Let $F\in\PcalK$.
	Then $F^{(r)}\in\Formel(r,1)$ and $H_*\Lbf\ell^r\left(F^{(r)}\right)=F^{E_r}$.	
\end{cor}
\begin{proof}
	It follows from the isomorphism \eqref{equ2.4.5} and Theorem \ref{thm2.6.1} that the complex  $\Lbf\ell^r\left(F^{(r)}\right)^{(r)}$ is isomorphic to the graded object  $\left(F^{E_r}\right)^{(r)}$. 
	By Proposition \ref{pro2.4.12}, $F^{(r)}$ belongs to the class $\Formel(r,1)$ and $H_*\Lbf\ell^r\left(F^{(r)}\right)\simeq F^{E_r}$.
\end{proof}

Theorem \ref{thm2.7.3} is a generalization to the multivariable case of Theorem \ref{thm2.6.1} and Corollary \ref{cor2.7.2}.

\begin{proof}[\rm \bfseries Proof of Theorem \ref{thm2.7.3}]
	The assertion (2) of the theorem is a consequence of the first assertion and Corollary \ref{cor2.4.9}.
	
	\medskip
	
	\underline{(1).} Assume that $F$ is an object of $\Pcalx{\dunder}$. Denote by $d$ the sum $d_1+\cdots+d_n$. By the isomorphism \eqref{equ2.4.5} of Proposition \ref{pro2.4.7} and the fact that $S^{\boxtimes\dunder}$ is the homogeneous part of degree $\dunder$ {of the functor} $S^d\circ\boxplus^n$, there are isomorphisms:
	\begin{align*}
	\Lbf\ell^{r}\left(F^{(r)}\right)^{\sharp(r)}
	&\simeq \RHomxin{\PcalKn}{F^{(r)},S^{\boxtimes\dunder(r)}}\\
	&\simeq \RHomxin{\PcalKn}{F^{(r)},S^{d(r)}\circ\boxplus^n}.
	\end{align*}
	By using successively
	Lemma \ref{lem2.5.6} and the isomorphism \eqref{equ2.6.8} of Lemma \ref{lem2.6.10}, we obtain isomorphisms
	\begin{align*}
	\Lbf\ell^r\left(F^{(r)}\right)^{\sharp(r)}(\Vunderline)
	&\simeq \RHomxin{\PcalKn}{F^{(r)},S^{d(r)}\circ\boxplus^n}(\Vunderline)\\
	&\simeq \RHomxin{\PcalKn}{\left(F^{(r)}\right)^{\Vunderline},S^{d(r)}\circ\boxplus^n}(\Delta_n(\Bbbk))\\
	&\simeq \RHomx{\PcalK}{\left(F^{(r)}\right)^{\Vunderline}\circ\Delta_n,S^{d(r)}}\\
	&\simeq \RHomx{\PcalK}{\left(F^{\Vunderline^{(r)}}\circ\Delta_n\right)^{(r)},S^{d(r)}}
	\end{align*}
	where $\Vunderline$ is a $n$-tuple of objects of $\VcalK$.
	Finally, by using successively
	Theorem \ref{thm2.6.1}, the  sum-diagonal adjunction, Yoneda's lemma and the fact that $\Sx{\dunder}{\Delta_n(E_r)}$ is the homogeneous part of degree $\dunder$ of $S^d_{E_r}\circ\boxplus^n$, we obtain isomorphisms
	\begin{align*}
	\Lbf\ell^r\left(F^{(r)}\right)^{\sharp(r)}(\Vunderline)
	&\simeq \RHomx{\PcalK}{F^{\Vunderline^{(r)}}\circ\Delta_n,S^d_{E_r}}\\
	&\simeq \RHomx{\PcalKn}{F^{\Vunderline^{(r)}},S^d_{E_r}\circ\boxplus^n}\\
	&\simeq \RHomx{\PcalKn}{F^{\Vunderline^{(r)}},\Sx{\dunder}{\Delta_n(E_n)}}\\
	&\simeq \RHomxin{\PcalKn}{F^{\Delta_n(E_r)},S^{\boxtimes\dunder}}\left(\Vunderline^{(r)}\right)\\
	&\simeq \left(F^{\Delta_n(E_r)}\right)^{\sharp(r)}(\Vunderline).
	\end{align*}
	Then, the complex  $\Lbf\ell^r\left(F^{(r)}\right)^{(r)}$ is formal and its homology is isomorphic to the graded object  $\left(F^{\Delta_n(E_r)}\right)^{(r)}$. By Proposition \ref{pro2.4.12},  the functor $F^{(r)}$ belongs to $\Formel(r,n)$ and $H_*\Lbf\ell^r\left(F^{(r)}\right)\simeq F^{\Delta_n(E_r)}$.
	
	\medskip

	\underline{(3) and (4).} Proposition \ref{pro2.4.7} and Assertion (1) yield isomorphisms:
	\begin{equation*}
	\Lbf\ell^{r+1}\left(F^{(1)}\right)\simeq \Lbf\ell^{r}\left(\Lbf\ell^1\left(F^{(1)}\right)\right)\simeq \Lbf\ell^r\left(F^{\Delta_n(E_1)}\right)\simeq \left(\Lbf\ell^r(F)\right)^{\Delta_n\left(E_1^{(r)}\right)}.
	\end{equation*}
	If $F$ belongs to $\Formel(r,n)$, then the complex  $\Lbf\ell^r(F)$ is formal. 
	Then, the complex  $\Lbf\ell^{r+1}\left(F^{(1)}\right)$ is formal and its homology is isomorphic to  $H_*\Lbf\ell^r(F)^{\Delta_n\left(E_1^{(r)}\right)}$. 
	It follows that the functor $F^{(1)}$ belongs to $\Formel(r+1,n)$ and $H_*\Lbf\ell^{r+1}\left(F^{(1)}\right)\simeq H_*\Lbf\ell^r(F)^{\Delta_n\left(E_1^{(r)}\right)}$.
	
	If $F^{(1)}$ belongs to $\Formel(r+1,n)$, then the complex  $\Lbf\ell^{r+1}\left(F^{(1)}\right)$ is formal. 
	By Proposition \ref{pro2.6.13}, the complex  $\Lbf\ell^r(F)$ is also formal. Consequently, the functor $F$ belongs to $\Formel(r,n)$.
\end{proof}

\begin{exe}
	Let $\Vunderline$ be a $n$-tuple of objects of $\VcalK$ and $F,G\in\Pcalx{\dunder}$. 
	There are graded isomorphisms natural in $F,G$:
	\begin{align*}
	\Extx{\PcalKn}{\left(\Gx{\dunder}{\Vunderline}\right)^{(r)},G^{(r)}}&\simeq G\left(V_1\otimes E_r,\ldots,V_n\otimes E_r\right),\\
	\Extx{\PcalKn}{F^{(r)},\left(\Sx{\dunder}{\Vunderline}\right)^{(r)}}&\simeq F^\sharp\left(V_1\otimes E_r,\ldots,V_n\otimes E_r\right).
	\end{align*}
\end{exe}

\subsection{ Compatibility of $\Formel(r,n)$ with the precomposition by $\boxtimes^n,\otimes^n$}

The main results of this subsection are Theorems \ref{thm2.6.8} and \ref{thm2.6.20}, that express the compatibility of formal classes with the precomposition by functors $ \boxtimes^{n} $ and $ \otimes^{n} $.
Notice that $ \otimes^{n}=\boxtimes^{n}\circ\Delta_n $. Our
starting point is the isomorphism \cite[Proposition 2.2]{FF08} of Franjou and Friedlander which can be written under the form:
\begin{equation}\label{equ3.1}
\Extx{\PcalKh}{\Gamma^d\circ\boxtimes^2,F_1\boxtimes F_2}\simeq \Extx{\PcalK}{F_1^\sharp,F_2}
\end{equation}
natural in $F_1,F_2\in\Pcalx{d}$.
In the Theorem \ref{pro2.6.6}, we generalize this result, by replacing  $\Gamma^{d}\circ \boxtimes^2$ and $F\boxtimes G$ in the left hand side of \eqref{equ3.1} respectively by $F\circ\boxtimes^n$ and $\boxtimesx{i=1}{n}F_i=F_{1}\boxtimes F_{2}\boxtimes\cdots\boxtimes F_{n}$ with $F,F_1,F_2,\ldots,F_n\in\PcalK$.
Symmetric monoidal structure of the categories $ \Pcalx{d} $, presented in the subsection \ref{Subs1.6}, plays an important role in this subsection.

\begin{thm}\label{pro2.6.6}
	Let $F_1,\ldots,F_n$ and $F$ be objects of $\Pcalx{d}$. 
	There are isomorphisms natural in $F_1,\ldots,F_n$ and $F$:
	\begin{align}
	\Homxin{\PcalKn}{F\circ\boxtimes^n,\boxtimesx{i=1}{n}F_i}&\simeq\left(F\otimesx{\PcalK}F_1^\sharp\otimesx{\PcalK}\cdots\otimesx{\PcalK}F_n^\sharp\right)^\sharp\circ\boxtimes^n,\label{equ2.6.1}\\
	\RHomxin{\PcalKn}{F\circ\boxtimes^n,\boxtimesx{i=1}{n}F_i}&\simeq\left(F\otimesxL{\PcalK}F_1^\sharp\otimesxL{\PcalK}\cdots\otimesxL{\PcalK}F_n^\sharp\right)^\sharp\circ\boxtimes^n.\label{equ2.6.2}
	\end{align}
\end{thm}

\begin{proof}
	Since each side of the isomorphism \eqref{equ2.6.1} is a left exact functor with respect to each variable $F_1,\ldots,F_n$, to prove this isomorphism, we can suppose that the functor $F_i$ is of the form  $F_i=S^d_{V_i}$ for $i=1,\ldots,n$. 
	By Yoneda's lemma and Theorem \ref{pro2.3.2}, there are isomorphisms:
	\begin{align*}
	\Homxin{\PcalKn}{F\circ\boxtimes^n,\boxtimesx{i=1}{n}S^d_{V_i}}
	&\simeq \left(\left(F\circ\boxtimes^n\right)^\sharp\right)_{(V_1,\ldots,V_n)}\\
	&\simeq \left(F^{V_1\otimes\cdots\otimes V_n}\right)^\sharp\circ\boxtimes^n\\
	&\simeq \left(F\otimesx{\PcalK}\Gamma^{d,V_1}\otimesx{\PcalK}\cdots\otimesx{\PcalK}\Gamma^{d,V_n}\right)^\sharp\circ\boxtimes^n
	\end{align*}
	which proves the isomorphism \eqref{equ2.6.1}. 
	To obtain \eqref{equ2.6.2}, let $P^{F_i}$ be a projective resolution of $F_i$ in the category $\Pcalx{d}$ for $i=1,\ldots,n$. 
	Hence the complexes $\left(P^{F_i}\right)^\sharp$ are injective coresolutions of the functors $F_i^\sharp$. Moreover, the isomorphism \eqref{equ2.6.1} yields an isomorphism of the complexes:
	\begin{equation*}
	\Homxin{\PcalKn}{F\circ\boxtimes^n,\boxtimesx{i=1}{n}P^{F_i}}\simeq\left(F\otimesx{\PcalK}\left(P^{F_1}\right)^\sharp\otimesx{\PcalK}\cdots\otimesx{\PcalK}\left(P^{F_n}\right)^\sharp\right)^\sharp\circ\boxtimes^n
	\end{equation*}
	which establishes the isomorphism \eqref{equ2.6.2}.
\end{proof}

Let $ F=\Gamma^{d} $ in the isomorphism \eqref{equ2.6.2} in Theorem \ref{pro2.6.6}, and by using the isomorphism \eqref{iso2.3.4} in Proposition \ref{pro2.3.8}, we obtain
\begin{align*}
\Extx{\PcalKh}{\Gamma^d\circ\boxtimes^2,F_1\boxtimes F_2}&\simeq \Extxin{\PcalKh}{\Gamma^d\circ\boxtimes^2,F_1\boxtimes F_2}(\Bbbk,\Bbbk)\\
&\simeq H^{*}\left(\RHomxin{\PcalKh}{\Gamma^{d}\circ\boxtimes^2,F_{1}\boxtimes F_{2}}(\Bbbk,\Bbbk)\right)\\
&\simeq
H^{*}\left(\left(\left(\Gamma^{d}\otimesxL{\PcalK}F_1^\sharp\otimesxL{\PcalK}F_2^\sharp\right)^\sharp\circ\boxtimes^2\right)(\Bbbk,\Bbbk)\right)\\
&\simeq
H^{*}\left(\left(F_1^\sharp\otimesxL{\PcalK}F_2^\sharp\right)^\sharp(\Bbbk)\right)\simeq H^{*}\left(\Rbf\Homxin{\PcalK}{F_{1}^{\sharp},F_{2}}(\Bbbk)\right)\simeq \Extx{\PcalK}{F_1^\sharp,F_2}
\end{align*}
natural in $ F_{1},F_{2}\in\Pcalx{d} $.
We get back the isomorphism \eqref{equ3.1} of Franjou and Friedlander.

Now, by using Theorem \ref{pro2.6.6}, we establish the relation between formal class $ \Formel(r,n) $ with the precomposition by the functors $ \boxtimes^n $ and $ \otimes^n $.

\begin{thm}\label{thm2.6.8} Let $F\in\Pcalx{p^rd}$. Suppose that the functor $F$ belongs to the formal class $\Formel(r,1)$.
	\begin{enumerate}
		\item[\rm (1)] The functor $F\circ\boxtimes^n$ belongs to $\Formel(r,n)$ and  $H_*\Lbf\ell^r(F\circ\boxtimes^n)\simeq H_*\Lbf\ell^r(F)^{E_r^{\otimes n-1}}\circ\boxtimes^n$.
		\item[\rm (2)] The functor  $F\circ\otimes^n$ belongs to $\Formel(r,1)$ and $H_*\Lbf\ell^r\left(F\circ\otimes^n\right)\simeq H_*\Lbf\ell^r(F)^{E_r^{\otimes n-1}}\circ\otimes^n$.
	\end{enumerate}
\end{thm}
\begin{proof}
	The assertion (2) is a direct consequence of Assertion (1) and Proposition \ref{pro2.5.8}. 
	It remains to prove the first assertion. By isomorphism \eqref{equ2.4.5} of Proposition \ref{pro2.2.11} and the isomorphism \eqref{equ2.6.2} of Theorem \ref{pro2.6.6}, there are isomorphisms:
	\begin{align}
	\Lbf\ell^r(F\circ\boxtimes^n)^{\sharp(r)}&\simeq \RHomxin{\PcalKn}{F\circ\boxtimes^n,\left(S^{d(r)}\right)^{\boxtimes n}}\label{equ2.6.5}\\
	&\simeq \left(F\otimesxL{\PcalK}\underbrace{\Gamma^{d(r)}\otimesxL{\PcalK}\cdots\otimesxL{\PcalK}\Gamma^{d(r)}}_{n\text{-times}}\right)^\sharp\circ\boxtimes^n.\notag
	\end{align}
	Moreover, Theorem \ref{thm2.6.1} and the isomorphism \eqref{iso2.3.4} of Proposition \ref{pro2.3.8} imply  $F^{(r)}\otimesxL{\PcalK}\Gamma^{d(r)}\simeq \left(F^{E_r}\right)^{(r)}$ natural in $F\in\Pcalx{p^rd}$. 
	It follows that
	\begin{equation}
	\underbrace{\Gamma^{d(r)}\otimesxL{\PcalK}\cdots\otimesxL{\PcalK}\Gamma^{d(r)}}_{n\text{-times}}\simeq \left(\Gamma^{d,E_r^{\otimes n-1}}\right)^{(r)}.\label{equ2.6.6}
	\end{equation}
	By isomorphism \eqref{iso2.4.3} of Proposition \ref{pro2.4.7}, the isomorphism \eqref{iso2.3.4} of Proposition \ref{pro2.3.8} and the isomorphisms \eqref{equ2.6.5}, \eqref{equ2.6.6}, we obtain isomorphisms:
	\begin{align*}
	\Lbf\ell^r(F\circ\boxtimes^n)^{\sharp(r)}&\simeq \left(F\otimesxL{\PcalK} \left(\Gamma^{d,E_r^{\otimes n-1}}\right)^{(r)}\right)^\sharp\circ\boxtimes^n\\
	&\simeq \RHomxin{\PcalK}{F,\left(S^d_{E_r^{\otimes n-1}}\right)^{(r)}}\circ\boxtimes^n\\
	&\simeq \RHomxin{\PcalK}{\Lbf\ell^{r}(F),S^d_{E_r^{\otimes n-1}}}^{(r)}\circ\boxtimes^n\\
	&\simeq \left(\Lbf\ell^r(F)^{E_r^{\otimes n-1}}\right)^{\sharp(r)}\circ\boxtimes^n\\
	&\simeq \left(\Lbf\ell^r(F)^{E_r^{\otimes n-1}}\circ\boxtimes^n\right)^{\sharp(r)}.
	\end{align*}
	Since $F$ belongs to $\Formel(r,1)$, the complex  $\Lbf\ell^r(F)$ is formal. 
	Then the complex  $\Lbf\ell^r(F\circ\boxtimes^n)^{(r)}$ is formal and its homology is isomorphic to the graded object $\left(H_*\Lbf\ell^r(F)^{E_r^{\otimes n-1}}\circ\boxtimes^n\right)^{(r)}$. 
	By Proposition \ref{pro2.4.12}, the functor $F\circ\boxtimes^n$ belongs to $\Formel(r,n)$ and $H_*\Lbf\ell^r(F\circ\boxtimes^n)\simeq H_*\Lbf\ell^r(F)^{E_r^{\otimes n-1}}\circ\boxtimes^n$.
\end{proof}

We recall a result of Cha\l upnik \cite{Cha08} (see also \cite{Tou13b}). 
If $X$ denotes one of the symbols $\Gamma,\Lambda,S$, we denote by $\epsilon_X$ respectively the integer $0,1,2$.

\begin{thm}[Cha\l upnik]\label{thm2.6.3}
	Let $X$ denote one of the symbols $\Gamma,\Lambda,S$. 
	There is an isomorphism
	\begin{equation*}
	\RHomxin{\PcalK}{X^{p^rd},S^{d(r)}}\simeq X^{d\sharp(r)}\left[\epsilon_X\left(p^rd-d\right)\right].
	\end{equation*}
\end{thm}

\begin{proof}
	By using \cite[Proposition 6.6]{Tou13b}, \cite[Lemma 3.6]{Tou13b} and \cite[Corollary 5.8]{Tou13b}, we have
	\begin{equation*}
	\RHomxin{\PcalK}{\Gamma^{p^rd},S^{d(r)}}\simeq \Homxin{\PcalK}{\Gamma^{p^rd},S^{d(r)}}\simeq S^{d(r)}\simeq \Gamma^{d\sharp (r)},
	\end{equation*}
	\begin{align*}
	\RHomxin{\PcalK}{\Lambda^{p^rd},S^{d(r)}}&\simeq \left(\RHomxin{\PcalK}{\Lambda^{d},S^{d}}\right)^{(r)}\left[p^{r}d-d\right]\simeq\left(\Homxin{\PcalK}{\Lambda^{d},S^{d}}\right)^{(r)}\left[p^{r}d-d\right]\\
	&\simeq \Lambda^{d\sharp(r)}\left[p^{r}d-d\right]
	\end{align*}
	and
	\begin{align*}
	\RHomxin{\PcalK}{S^{p^rd},S^{d(r)}}&\simeq \left(\RHomxin{\PcalK}{S^{d},S^{d}}\right)^{(r)}\left[2\left(p^{r}d-d\right)\right]\simeq\left(\Homxin{\PcalK}{S^{d},S^{d}}\right)^{(r)}\left[2\left(p^{r}d-d\right)\right]\\
	&\simeq \Gamma^{d\sharp(r)}\left[2\left(p^{r}d-d\right)\right].
	\end{align*}
	We obtain the desired isomorphisms.
\end{proof}

\begin{cor}\label{cor2.6.5}
	Let $X$ denote one of the symbols $\Gamma,\Lambda,S$.
	Then, the functor $X^d$ belongs to $\Formel(r,1)$ for all $r$ and $\Lbf\ell^r(X^d)$ is isomorphic to $ X^e\left[\epsilon_X\left(d-e\right)\right]$ if $d$ is of the form $p^re,e\in\Nbb$ and to $0$ otherwise.
\end{cor}
\begin{proof} 
	If $d$ is not a multiple of $p^r$ then $\Lbf\ell^r(F)=0$ by the definition of the functor $\ell^r$ and hence $X^d\in\Formel(r,1)$. 
	We next suppose that $d$ is of the form $p^re$ with $e\in\Nbb$.
	By the isomorphism \eqref{equ2.4.5} of Proposition \ref{pro2.4.7} and Theorem \ref{thm2.6.3}, the complex  $\Lbf\ell^r\left(X^{d}\right)^{(r)}$ is isomorphic to $X^{e(r)}\left[\epsilon_X\left(d-e\right)\right]$. By Proposition \ref{pro2.4.12}, the functor $X^{d}$ belongs to $\Formel(r,1)$ and $\Lbf\ell^r(X^d)\simeq X^e\left[\epsilon_X\left(d-e\right)\right]$.
\end{proof}

The following theorem is a generalization of Theorem \ref{thm2.6.3} and Corollary \ref{cor2.6.5}.

\begin{thm}\label{thm2.6.20}
	Let $X$ denote one of the symbols  $\Gamma,\Lambda,S$.  Let $F\in\PcalK$ and $G\in\PcalKn$.
	\begin{enumerate}
		\item[\rm (1)] The functor $X^d\circ\boxtimes^n$ belongs to $\Formel(r,n)$ for all $r$ and $\Lbf\ell^r(X^d\circ\boxtimes^n)$ is isomorphic to $ X^{e,E_r^{\otimes n-1}}\left[\epsilon_X\left(d-e\right)\right]\circ\boxtimes^n$ if $d$ is of the form $p^re,e\in\Nbb$ and to $0$ otherwise.
		\item[\rm (2)] The functor $X^d\circ\otimes^n$ belongs to $\Formel(r,1)$ for all $r$ and $\Lbf\ell^r(X^d\circ\otimes^n)$ is isomorphic to $ X^{e,E_r^{\otimes n-1}}\left[\epsilon_X\left(d-e\right)\right]\circ\otimes^n$ if $d$ is of the form $p^re,e\in\Nbb$ and to $0$ otherwise.
		\item[\rm (3)] There exist graded isomorphisms, natural in $F,G$:
		\begin{eqnarray}
		\Extxin{\PcalKn}{X^{p^rd}\circ\boxtimes^n,G^{(r)}}&\simeq& \Extxxin{*-\epsilon_X(p^rd-d)}{\PcalKn}{X^{d,E_r^{\otimes n-1}}\circ\boxtimes^n,G}^{(r)},\label{equ3.7}\\
		\Extxin{\PcalK}{X^{p^rd}\circ\otimes^n,F^{(r)}}&\simeq& \Extxxin{*-\epsilon_X(p^rd-d)}{\PcalK}{X^{d,E_r^{\otimes n-1}}\circ\otimes^n,F}^{(r)}.\label{equ3.8}
		\end{eqnarray}
	\end{enumerate}
\end{thm}

\begin{proof}
	The assertions (1) and (2) are consequences of Theorem \ref{thm2.6.8} and Corollary \ref{cor2.6.5}. In order to obtain the last assertion of the theorem, we use the two assertions (1), (2) and Corollary \ref{cor2.4.9}.
\end{proof}

\begin{exe}
	\begin{enumerate}
		\item[\rm (1)] There are graded isomorphisms:
		\begin{align*}
		\Extx{\PcalKn}{\Gamma^{p^r}\circ\boxtimes^n,\boxtimes^{n(r)}}&\simeq \left(E_r\right)^{\otimes n-1},\\
		\Extx{\PcalKn}{\Lambda^{p^r}\circ\boxtimes^n,\boxtimes^{n(r)}}&\simeq \left(E_r\right)^{\otimes n-1}\left[p^r-1\right],\\
		\Extx{\PcalKn}{S^{p^r}\circ\boxtimes^n,\boxtimes^{n(r)}}&\simeq \left(E_r\right)^{\otimes n-1}\left[2p^r-2\right].
		\end{align*}
		For example, we have $ \Extx{\PcalKn}{\Lambda^{p^r}\circ\boxtimes^n,\boxtimes^{n(r)}}\simeq\Extxx{*-(p^rd-d)}{\PcalKn}{\Lambda^{1,E_r^{\otimes n-1}}\circ\boxtimes^n,\boxtimes^{n}} \simeq \left(E_r\right)^{\otimes n-1}\left[p^r-1\right] $.
		
		\item[\rm (2)] Let $V\in\VcalK$. 
		There are graded isomorphisms:
		\begin{align*}
		\Extx{\PcalK}{\Gamma^{p^rd}\circ\otimes^n,\left(S^{nd}_V\right)^{(r)}}&\simeq S^d\left(V^{\otimes n}\otimes \left(E_r\right)^{\otimes n-1}\right),\\
		\Extx{\PcalK}{\Lambda^{p^rd}\circ\otimes^n,\left(S^{nd}_V\right)^{(r)}}&\simeq \Lambda^d\left(V^{\otimes n}\otimes \left(E_r\right)^{\otimes n-1}\right)\left[p^rd-d\right],\\
		\Extx{\PcalK}{S^{p^rd}\circ\otimes^n,\left(S^{nd}_V\right)^{(r)}}&\simeq \Gamma^d\left(V^{\otimes n}\otimes \left(E_r\right)^{\otimes n-1}\right)\left[2p^rd-2d\right].
		\end{align*}
	\end{enumerate}
\end{exe}

\section{ $\Ext$-groups $\Extx{\PcalK}{\Gamma^{p^rd}\circ X, S^{\mu(r)}}$}\label{sect3.3}

In this section, we compute $\Ext$-groups $\Extx{\PcalK}{\Gamma^{p^rd}\circ X, S^{\mu(r)}}$ where $X$ is a direct factor of the $n$-th tensor power functor $\otimes^n$ and $\mu$ is a tuple of natural numbers of weight $dn$. 
The principal result is the following theorem.

\begin{thm}\label{thm3.3.8}
	Let $\mu$ be a tuple of natural numbers of weight $ nd $. Let $X$ be a direct factor of $\otimes^n$. There is a graded  isomorphism
	\begin{equation*}
	\alpha(S^\mu,X):\Extx{\PcalK}{\Gamma^{p^rd}\circ X,S^{\mu(r)}}\xrightarrow{\simeq} \Homx{\PcalK}{\Gamma^{d,E_r^{\otimes n-1}}\circ X,S^\mu}
	\end{equation*}
	natural in $S^\mu$, and compatible with the products.
\end{thm}

This theorem, more concretely, its corollary \ref{cor3.3.9}, will be used in the following section to compute the cohomology of the orthogonal and symplectic groups schemes with coefficients in $S^\mu(\Bbbk^{2n\vee(r)})$.

The theorem \ref{thm3.3.8} is obtained from the following theorem.

\begin{thm}\label{thm3.4.1}
	Let $\mu$ be a tuple of natural numbers of weight $dn$.
	Let $r$ be a natural number.
	Then, there is a graded isomorphism, natural in $S^\mu$
	\begin{equation}
	\alpha(S^\mu):\Extx{\PcalK}{\Gamma^{p^rd}\circ\otimes^n,S^{\mu(r)}}\xrightarrow{\simeq} \Homx{\PcalK}{\Gamma^{d,E_r^{\otimes n-1}}\circ\otimes^n,S^\mu}.
	\end{equation}
	Moreover, $\alpha(S^\mu)$ satisfies the following properties.
	\begin{enumerate}
		\item[\rm (1)] $\alpha(S^\mu)$ is compatible with the products, i.e., we have
		\begin{equation*}
		\alpha(S^\mu)(c_1)\smile \alpha(S^\lambda)(c_2)=\alpha(S^\mu\otimes S^\lambda)(c_1\smile c_2).
		\end{equation*}
		\item[\rm (2)] $\alpha(S^\mu)$ is $\Sfrak_n$-equivariant (relatively to actions defined in Definition \rm{\ref{def3.2.3}}).
	\end{enumerate}
\end{thm}

An isomorphism between $ \Extx{\PcalK}{\Gamma^{p^rd}\circ\otimes^n,S^{\mu(r)}}$ and $ \Homx{\PcalK}{\Gamma^{d,E_r^{\otimes n-1}}\circ\otimes^n,S^\mu} $ was established  in Theorem \ref{thm2.6.20}, however, we don't know if this isomorphism is compatible with the products and the actions of $\Sfrak_n$.
Our effort is to construct another isomorphism satisfying these properties.
The construction of this new isomorphism use essentially the isomorphism \eqref{equ3.8} in Theorem \ref{thm2.6.20}.

\subsection{ Cup products and actions of the symmetric group}\label{subsect3.2.2}

We recall the definition of cup products. Let $(\Mcal,\otimes)$ be an abelian category with enough injective objects, equipped with a biexact monoidal product $\otimes$ which preserves the injectives.
This monoidal product induces a tensor product on the $\Ext$-groups of $\Mcal$:
\begin{equation}\label{equ3.2.1}
\otimes: \Extx{\Mcal}{A_1,B_1}\otimes\Extx{\Mcal}{A_2,B_2}\to\Extx{\Mcal}{A_1\otimes A_2,B_1\otimes B_2}.
\end{equation}

The category of the strict polynomial functors in one or several variables  $(\PcalKn,\otimes)$ and the category $(G\Mod,\otimes)$ of the rational modules over algebraic group scheme $G$ have a tensor product on the $\Ext$-groups.

Let $C=(C^0,C^1,\ldots)$ be a graded coalgebra in $(\Mcal,\otimes)$. We denote by $\Delta_{d,e}:C^{d+e}\to C^d\otimes C^e$ the diagonal morphism.
We define the cup product
\begin{equation*}
\begin{array}{cccc}
\smile: &\Extx{\Mcal}{C^d,B_1}\otimes \Extx{\Mcal}{C^e,B_2}&\to&\Extx{\Mcal}{C^{d+e},B_1\otimes B_2}\\
\noalign{\medspace}
&e_1\otimes e_2&\mapsto& \Delta_{d,e}^*(e_1\otimes e_2).
\end{array}
\end{equation*}
The following lemma states some elementary properties of the products that we need in the sequel.

\begin{lem}\label{lem3.2.1}
	Let $\Mcal,\Mcal_1,\Mcal_2$ be abelian categories with enough injective objects, equipped with a biexact monoidal product $\otimes$ which preserves the injectives.
	\begin{enumerate}
		\item[\rm (1)] Let $\phi:\Mcal_1\to\Mcal_2$ be an exact functor.
		Suppose that $\phi$ is a strict monoidal functor. 
		The morphism $\phi:\Extx{\Mcal_1}{A,B}\to\Extx{\Mcal_2}{\phi A,\phi B}$ induced by $\phi$ is compatible with the products, i.e., we have $\phi(c_1\smile c_2)=(\phi c_1)\smile(\phi c_2)$.
		\item[\rm (2)] Let $\phi:\Mcal\to\Mcal$ be a strict monoidal functor.
		Let $\phi\to\Id$ be a natural transformation which is compatible with the monoidal product.
		The morphism induced by this natural transformation $\Extx{\Mcal}{A,B}\to \Extx{\Mcal}{\phi A,B}$ is compatible with the products.
	\end{enumerate}
\end{lem}

\begin{deff}\label{def3.2.2}
	Let $\sigma$ be an element of the symmetric group $\Sfrak_n$. 
	\begin{enumerate}
		\item[\rm (1)] Denote by $\sigma_\otimes:\otimes^n\to\otimes^n$ the morphism which sends the element $v_1\otimes\cdots\otimes v_n\in V^{\otimes n},V\in\VcalK$ to $v_{\sigma^{-1}(1)}\otimes\cdots\otimes v_{\sigma^{-1}(n)}$. We obtain an action of $\Sfrak_n$ on the functor $\otimes^n\in\Pcalx{n}$.
		\item[\rm (2)] Denote by $\sigma_{\mathcal{V}}:\VcalK^{\times n}\to\VcalK^{\times n}$ the morphism which sends the object $(V_1,\ldots,V_n)$ of $\VcalK^{\times n}$ to $\left(V_{\sigma^{-1}(1)},\ldots,V_{\sigma^{-1}(n)}\right)$. We obtain an action of $\Sfrak_n$ on the category $\VcalK^{\times n}$.
		\item[\rm (3)] Denote by $\sigma_\boxtimes:\boxtimes^n\to\boxtimes^n\circ\sigma_{\mathcal{V}}$ the morphism which sends the element $v_1\otimes\cdots\otimes v_n\in V_1\otimes\cdots\otimes V_n$ to $v_{\sigma^{-1}(1)}\otimes\cdots\otimes v_{\sigma^{-1}(n)}$.
		\item[\rm (4)] Denote by $\sigma_\oplus:\oplus^n\to\oplus^n$ the morphism which sends the element $(v_1,\ldots, v_n)$ $\in V^{\oplus n},V\in\VcalK$ on $\left(v_{\sigma^{-1}(1)},\ldots, v_{\sigma^{-1}(n)}\right)$. We obtain an action of $\Sfrak_n$ on the functor $\oplus^n\in\Pcalx{1}$.
		\item[\rm (5)] Denote by $\sigma_\boxplus:\boxplus^n\to\boxplus^n\circ\sigma_{\mathcal{V}}$ the morphism which sends the element $(v_1,\ldots,v_n)\in V_1\oplus\cdots\oplus V_n$ on $(v_{\sigma^{-1}(1)},\ldots,v_{\sigma^{-1}(n)})$.
	\end{enumerate}
\end{deff}

For each $\sigma\in\Sfrak_n$, we have $\boxtimes^n\circ\sigma_{\mathcal{V}}\circ\Delta_n=\otimes^n$ and $\boxplus^n\circ\sigma_{\mathcal{V}}\circ\Delta_n=\oplus^n$.
By definition \ref{def3.2.2}, $\sigma_\otimes$ and $\sigma_\oplus$ are respectively the composite morphisms:
\begin{gather*}
\otimes^n=\boxtimes^n\circ\Delta_n\xrightarrow{\sigma_\boxtimes(\Delta_n)}\boxtimes^n\circ\sigma_{\mathcal{V}}\circ\Delta_n=\otimes^n,\\
\oplus^n=\boxplus^n\circ\Delta_n\xrightarrow{\sigma_\boxplus(\Delta_n)}
\boxplus^n\circ\sigma_{\mathcal{V}}\circ\Delta_n=\oplus^n.
\end{gather*}
Let $F,G\in\PcalK$.
The symmetric group $\Sfrak_n$ acts on $F\circ\otimes^n$ and $\Extxx{i}{\PcalK}{F\circ\otimes^n,G}$.
More precisely, an element $\sigma\in\Sfrak_n$ acts on $F\circ\otimes^n$ by $F(\sigma_\otimes)$ and acts on $\Extxx{i}{\PcalK}{F\circ\otimes^n,G}$ by $\Extxx{i}{\PcalK}{F(\sigma^{-1}_\otimes),G}$.
Moreover, each morphism $f:F_1\to F_2$ in $\PcalK$ induces  $\Sfrak_n$-equivariant morphisms:
\[ F_1\circ\otimes^n\to  F_2\circ\otimes^n,\qquad \Extxx{i}{\PcalK}{F_2\circ\otimes^n,G}\to\Extxx{i}{\PcalK}{F_1\circ\otimes^n,G}. \]

\begin{deff}\label{def3.2.3} Let $F,G\in\PcalKn$. 
	For each $\sigma\in\Sfrak_n$, we define the action of $\sigma$ on $\Extx{\PcalKn}{F\circ\boxtimes^n,G\circ\boxplus^n}$ by the composition:
	\begin{align*}
	\Extx{\PcalKn}{F\circ\boxtimes^n,G\circ\boxplus^n}&
	\to\Extx{\PcalKn}{F\circ\boxtimes^n\circ\sigma^{-1}_{\mathcal{V}},G\circ\boxplus^n\circ\sigma^{-1}_{\mathcal{V}}}\\
	&\to \Extx{\PcalKn}{F\circ\boxtimes^n,G\circ\boxplus^n}
	\end{align*}
	where the first map is induced by precomposition with $\sigma^{-1}_{\mathcal{V}}$, and the second one is the map $\Extx{\PcalKn}{F(\sigma^{-1}_\boxtimes),G(\sigma_\boxplus)}$. 
	We obtain an action of $\Sfrak_n$ on $\Ext$-group $\Extx{\PcalKn}{F\circ\boxtimes^n,G\circ\boxplus^n}$.
	We can define in analogous manner the action of the symmetric group $\Sfrak_n$ on the $\Ext$-groups:
	\begin{equation*}
	\Extx{\PcalKn}{F\circ\boxtimes^n,G\circ\boxtimes^n},\, \Extx{\PcalKn}{F\circ\boxplus^n,G\circ\boxplus^n},\, \Extx{\PcalKn}{F\circ\boxplus^n,G\circ\boxtimes^n}.
	\end{equation*}
\end{deff}
By this definition, for $\sigma\in\Sfrak_n$ and $f:F\circ\boxtimes^n\to G\circ\boxplus^n$, the morphism $\sigma\cdot f$ is the composition:
\begin{equation*}
F\left(\bigotimes_{i=1}^nV_i\right)\xrightarrow{F(\sigma^{-1}_\boxtimes)}F\left(\bigotimes_{i=1}^nV_{\sigma(i)}\right)\xrightarrow{f}G\left(\bigoplus_{i=1}^nV_{\sigma(i)}\right)\xrightarrow{G(\sigma_\boxplus)}G\left(\bigoplus_{i=1}^nV_{i}\right).
\end{equation*}
In particular, we have an action of $\Sfrak_n$ on the vector space $\Homx{\PcalKn}{\boxtimes^n,\boxplus^n}$. The Yoneda's lemma induces an isomorphism $\Homx{\PcalKn}{\boxtimes^n,\boxplus^n}\simeq \boxplus^n(\Bbbk,\ldots,\Bbbk)=\Bbbk^n$. This isomorphism is $\Sfrak_n$-equivariant where the group $\Sfrak_n$ acts canonically on $\Bbbk^n$, i.e., $\sigma\cdot(\lambda_1,\ldots,\lambda_n)=\left(\lambda_{\sigma^{-1}(1)},\ldots,\lambda_{\sigma^{-1}(n)}\right)$ for $\lambda_i\in\Bbbk$.

\begin{vid}\label{rem3.2.4}
	Let $F,G$ be two strict polynomial functors. The precomposition with $\boxplus^n$ induces a graded morphism $\Extx{\PcalK}{F,G}$ $\to\Extx{\PcalKn}{F\circ\boxplus^n,G\circ\boxplus^n}$. 
	The symmetric group $\Sfrak_n$ acts trivially on $\Homx{\PcalK}{F,G}$ {and acts} on the vector space $\Homx{\PcalKn}{F\circ\boxplus^n,G\circ\boxplus^n}$ as in Definition \ref{def3.2.3}. Then, this morphism is $\Sfrak_n$-equivariant. 
	
	The graded morphism
	$\Extx{\PcalK}{F,G}\to\Extx{\PcalKn}{F\circ\boxtimes^n,G\circ\boxtimes^n}$
	induced by the precomposition with $\boxtimes^n$ is also $\Sfrak_n$-equivariant if $\Sfrak_n$ acts trivially on the source of the morphism and its action on the target of the morphism is given by the precomposition with $\boxtimes^n$.
\end{vid}

Let $F_1,\ldots,F_n$ be strict polynomial functors and $\sigma$ be an element of $\Sfrak_n$.
Denote by $\sigma_{(F_1,\ldots,F_n)}$ the map $\bigotimes_{i=1}^nF_i(V_i)$ $\to \bigotimes_{i=1}^nF_{\sigma^{-1}(i)}(V_{\sigma^{-1}(i)})$ induced by the permutation in the tensor product of factors. Thus the morphism $\sigma_{(F_1,\ldots,F_n)}$ is a natural transformation of $\boxtimesx{i=1}{n}F_i\to \left(\boxtimesx{i=1}{n}F_{\sigma^{-1}(i)}\right)\circ\sigma_{\mathcal{V}}$. In particular, if $F_1=\cdots=F_n=I$ then $\sigma_{(I,\ldots,I)}$ is equal to the morphism $\sigma_\boxtimes$ in Definition \ref{def3.2.2}.

\begin{deff}\label{def3.2.5} Let $F_1,\ldots,F_n$ be strict polynomial functors. Let $\sigma$ be an element of $\Sfrak_n$. Define a morphism $\sigma\cdot:\Extx{\PcalKn}{\Gamma^d\circ\boxtimes^n,\boxtimesx{i=1}{n}F_i}\to\Extx{\PcalKn}{\Gamma^d\circ\boxtimes^n,\boxtimesx{i=1}{n}F_{\sigma^{-1}(i)}}$
	to be the composition
	\begin{align*}
	\Extx{\PcalKn}{\Gamma^d\circ\boxtimes^n,\boxtimesx{i=1}{n}F_i}&\to \Extx{\PcalKn}{\Gamma^d\circ\boxtimes^n\circ\sigma^{-1}_{\mathcal{V}},\left(\boxtimesx{i=1}{n}F_i\right)\circ\sigma^{-1}_{\mathcal{V}}}\\
	&\to \Extx{\PcalKn}{\Gamma^d\circ\boxtimes^n,\boxtimesx{i=1}{n}F_{\sigma^{-1}(i)}}
	\end{align*}
	where the first map is induced by the precomposition with $\sigma^{-1}_{\mathcal{V}}$, and the second is the map $\Extx{\PcalKn}{\Gamma^d(\sigma^{-1}_{\boxtimes}),\sigma_{(F_1,\ldots,F_n)}}$.
\end{deff}

By this definition, for $\sigma\in\Sfrak_n$, $(V_1,\ldots,V_n)\in\VcalK^{\times n}$ and $f:\Gamma^d\circ\boxtimes^n\to \boxtimesx{i=1}{n}F_i$, the morphism $\sigma\cdot f$ is the composition:
\begin{equation*}
\Gamma^d\left(\bigotimes_{i=1}^nV_i\right)\xrightarrow{\Gamma^d(\sigma^{-1}_{\mathcal{V}})} \Gamma^d\left(\bigotimes_{i=1}^nV_{\sigma(i)}\right)\xrightarrow{f} \bigotimes_{i=1}^nF_i(V_{\sigma(i)})\xrightarrow{\sigma_{(F_1,\ldots,F_n)}}\bigotimes_{i=1}^nF_{\sigma^{-1}(i)}(V_{i}).
\end{equation*}

\subsection{ Compatibility with the products and the actions of $\Sfrak_n$}

The following result is an important consequence of Theorem \ref{thm2.6.20}.

\begin{cor}\label{cor3.3.1} 
	Let $\mu^2,\ldots,\mu^n$ be tuples of natural numbers. 
	For $k\in\Nbb$, the following functor from $\PcalK$ to $\VcalK$ is exact:
	\begin{equation}\label{equ3.3.1}
	F\mapsto 
	\Extxx{k}{\PcalKn}{\Gamma^{p^rd}\circ\boxtimes^n,F^{(r)}\boxtimes\boxtimesx{i=2}{n}S^{\mu^i(r)}}.
	\end{equation}
\end{cor}

In order to prove this corollary, we first establish the following lemma that will also be used in the proof of Theorem \ref{thm3.3.1}.
\begin{lem}\label{lem3.3.2a}
	Let $J^2,\ldots,J^n\in\PcalK $ be injectives  and $V\in\VcalK$. 
	The functor
	\begin{equation*}
	F\mapsto \Homx{\PcalKn}{\Gamma^{d,V}\circ\boxtimes^n,F\boxtimes\boxtimesx{i=2}{n}J^i}
	\end{equation*}
	from $\PcalK$ to $\VcalK$ is an exact functor.
\end{lem}
\begin{proof}
	By Theorem \ref{pro2.6.6}, there is an isomorphism natural in $F$:
	\begin{equation*}
	\Homx{\PcalKn}{\Gamma^{d,V}\circ\boxtimes^n,F\boxtimes\boxtimesx{i=2}{n}J^i}\simeq\left(F^{\sharp}\otimesx{\PcalK}J^{2\sharp}\otimesx{\PcalK}\cdots\otimesx{\PcalK}J^{n\sharp}\right)^\sharp(V).
	\end{equation*}
	Moreover, since $J^{i\sharp}$ is projective, the functor $-\otimesx{\PcalK}J^{i\sharp}$ is exact. It follows that the right hand side of the above isomorphism is an exact functor of  $F$, which proves the lemma.
\end{proof}

\begin{proof}[\rm \bfseries Proof of Corollary {\rm \ref{cor3.3.1}}]
	By Theorem \ref{thm2.6.20}, the functor \eqref{equ3.3.1} is isomorphic to the functor 
	\begin{equation*}
	F\mapsto \Extxx{k}{\PcalKn}{\Gamma^{d,E_r^{\otimes n-1}}\circ\boxtimes^n,F\boxtimes\boxtimesx{i=2}{n}S^{\mu^i}}
	\end{equation*}
	which is exact by Lemma \ref{lem3.3.2a}.
	This is the desired conclusion.
\end{proof}

The second ingredient that we use  to prove Theorem \ref{thm3.3.1} is the explicit injective coresolution $T$ of $I^{(r)}$ constructed by Friedlander-Suslin \cite{FS97} if $p=2$ and by Troesch \cite{Tro05b} in general case.
This idea of using this coresolution of Troesch to compute $\Ext$-groups in the category of the strict polynomial functors is due to Touz\'e, see \cite{Tou12}.
We describe the injective coresolution $T$ as a graded object in the following lemma (the explicit formula of the differential of $T$ is not necessary in the sequel).

\begin{lem}[\text{\cite{Tou12,Tro05b}}]\label{lem3.2.5}
	As graded functors, we have a decomposition
	\begin{equation*}
	T=\left(E_r\otimes S^{p^r}\right)\oplus T'
	\end{equation*}
	where the functor $T'$ is a direct sum of $S^\mu$ for some tuple of natural numbers $\mu$ which is not a multiple of $p^r$.
\end{lem}

\begin{proof} 
	For a natural number $i$, denote by $\Ical(i)$ the set of tuples of natural numbers $\mu=\left(\mu_0,\mu_1,\ldots,\mu_{p^r-1}\right)$ such that $\sum_{j=0}^{p^r-1}\mu_j=p^r$ and $\sum_{j=0}^{p^r-1}j\mu_j=p^r\lflooroper{\frac{i}{2}}+p^{r-1}\left(i-2\lflooroper{\frac{i}{2}} \right)$. 
	By \cite{Tro05b}, the part of cohomological degree $i$ of the complex  $T$ is $T^i=\bigoplus_{\mu\in\Ical(i)} S^{\mu}$.
	
	Let $\mu$ be an element of $\Ical(i)$.
	We have that $\sum_{j=0}^{p^r-1}\mu_j=p^r$ and $\mu_j\in\Nbb$. 
	Then $p^r$ divides $\mu$ if and only if $\mu$ is of the form $\left(0,\ldots,0,p^r,0,\ldots,0\right)$. Then we have $p^r\in \Ical(i)$ if and  only if $i\in\{0,2,\ldots,2p^r-2\}$. 
	Moreover, for each $\mu\in\Ical(i)\setminus\{p^r\}$, $p^r$ does not divide $\mu$.
	Thus we obtain the desired decomposition of $T$ as graded functor.
\end{proof}

The third ingredient of the proof of Theorem \ref{thm3.3.1} is the following generalization of \cite[Lemma 2.2, 2.3]{Tou12}.

\begin{lem}\label{lem3.3.2} Let $\mu^2,\ldots,\mu^n$ be tuples of natural numbers and $F,F_1,\ldots,F_n$ be objects of $\PcalK$.
	\begin{enumerate}
		\item[\rm (1)] If there is an $i$ such that $p^r$ does not divide $\mu^i$ then 
		\begin{equation*}
		\Homx{\PcalKn}{\Gamma^{p^rd}\circ\boxtimes^n,F^{(r)}\boxtimes S^{\mu^2}\cdots\boxtimes S^{\mu^n}}=0.
		\end{equation*}
		\item[\rm (2)] The following morphism, which is induced by Frobenius twist, is an isomorphism: 
		\begin{equation*}
		\Homx{\PcalKn}{\Gamma^d\circ\boxtimes^n,\boxtimesx{i=1}{n}F_i}\xrightarrow{\simeq} \Homx{\PcalKn}{\Gamma^{p^rd}\circ\boxtimes^n,\boxtimesx{i=1}{n}F_i^{(r)}}.
		\end{equation*}
		\item[\rm (3)] The following morphism, which is induced by the canonical inclusions $S^{\mu^i(r)}\hookrightarrow S^{p^r\mu^i}, i=2,\ldots,n$, is an isomorphism:
		\begin{equation*}
		\Homx{\PcalKn}{\Gamma^{p^rd}\circ\boxtimes^n, F^{(r)}\boxtimesx{i=1}{n}S^{\mu^i(r)} }\xrightarrow{\simeq}\Homx{\PcalKn}{\Gamma^{p^rd}\circ\boxtimes^n, F^{(r)}\boxtimesx{i=1}{n}S^{p^r\mu^i} }.
		\end{equation*}
	\end{enumerate}
\end{lem}
\begin{proof} 
	The isomorphism \eqref{iso2.3.2a} of Proposition \ref{pro2.3.8} yields an isomorphism $\left(F^\sharp\otimesx{\PcalK}G^\sharp\right)^\sharp(V)$ $\simeq \Homx{\PcalK}{F^\sharp,G_V}$ natural in $F,G\in\PcalK$ and $V\in\VcalK$.
	By  using this isomorphism and \cite[Lemmas 2.2, 2.3]{Tou12}, we obtain:
	\begin{enumerate}
		\item[\rm (1)] if $\mu$ is not a multiple of  $p^r$ then $F^{(r)}\otimesx{\PcalK}\Gamma^{\mu}=0$;
		\item[\rm (2)] the morphism induced by the Frobenius twist $\left(F\otimesx{\PcalK}G\right)^{(r)}\to F^{(r)}\otimesx{\PcalK}G^{(r)}$ is an isomorphism;
		\item[\rm (3)] the morphism $F^{(r)}\otimesx{\PcalK}\Gamma^{p^r\mu}\to F^{(r)}\otimesx{\PcalK}\Gamma^{\mu(r)}$ induced by the canonical projection  $\Gamma^{p^r\mu}\to\Gamma^{\mu(r)}$ is an isomorphism.
	\end{enumerate}
	Moreover, by Theorem \ref{pro2.6.6}, we have an isomorphism natural in $F_1,\ldots,F_n\in\Pcalx{d}$:
	\begin{equation}\label{equ3.3.2}
	\Homx{\PcalK}{\Gamma^d\circ\boxtimes^n,\boxtimesx{i=1}{n}F_i}\simeq \left(F_1^\sharp\otimesx{\PcalK}\cdots\otimesx{\PcalK}F_n^{\sharp} \right)^\sharp(\Bbbk).
	\end{equation}
	Combining the isomorphisms, we deduce the desired conclusion.
\end{proof}

\begin{thm}\label{thm3.3.1}
	Let $\mu^1,\ldots,\mu^n$ be tuples of natural numbers of weight $d$. Let $r$ be a natural number. There exists a graded isomorphism $\theta=\theta(S^{\mu^1},\ldots,S^{\mu^n})$, natural in $S^{\mu^i}, i=1,\ldots,n$
	\begin{equation*}
	\theta:\Extx{\PcalKn}{\Gamma^{p^rd}\circ\boxtimes^n,\boxtimesx{i=1}{n}S^{\mu^i(r)}}\xrightarrow{\simeq} \Homx{\PcalKn}{\Gamma^{d,E_r^{\otimes n-1}}\circ\boxtimes^n,\boxtimesx{i=1}{n}S^{\mu^i}}.
	\end{equation*}
	Moreover, $\theta(S^{\mu^1},\ldots,S^{\mu^n})$ satisfies the following properties.
	\begin{enumerate}
		\item[\rm (1)] $\theta(S^{\mu^1},\ldots,S^{\mu^n})$ is  compatible with the products, i.e., we have
		\begin{equation*}
		\theta(S^{\mu^1}\otimes S^{\lambda^1},\ldots,S^{\mu^n}\otimes S^{\lambda^n})(c_1\smile c_2)=\theta(S^{\lambda^1},\ldots,S^{\lambda^n})(c_1)\smile \theta(S^{\mu^1},\ldots,S^{\mu^n})(c_2).
		\end{equation*}
		\item[\rm (2)] $\theta(S^{\mu^1},\ldots,S^{\mu^n})$ is $\Sfrak_n$-equivariant, i.e., the diagram
		\begin{equation*}
		\xymatrix{
			\Extx{\PcalKn}{\Gamma^{p^rd}\circ\boxtimes^n,\boxtimesx{i=1}{n}S^{\mu^i(r)}}\ar[r]^-{\theta}_-\simeq\ar[d]^{\sigma\cdot}&\Homx{\PcalKn}{\Gamma^{d,E_r^{\otimes n-1}}\circ\boxtimes^n,\boxtimesx{i=1}{n}S^{\mu^i}}\ar[d]^{\sigma\cdot}\\
			\Extx{\PcalKn}{\Gamma^{p^rd}\circ\boxtimes^n,\boxtimesx{i=1}{n}S^{\mu^{\sigma^{-1}(i)}(r)}}\ar[r]^-{\theta}_-\simeq&\Homx{\PcalKn}{\Gamma^{d,E_r^{\otimes n-1}}\circ\boxtimes^n,\boxtimesx{i=1}{n}S^{\mu^{\sigma^{-1}(i)}}}
		}
		\end{equation*}
		is commutative with $\sigma\in\Sfrak_n$ and the vertical maps defined in definition {\rm \ref{def3.2.5}}.
	\end{enumerate}
\end{thm}

\begin{proof}
	We first consider the case  $S^{\mu^1}=\cdots=S^{\mu^n}=\otimes^d$.
	Denote by $T(\otimes^d)$ the complex  $T^{\otimes d}$. 
	Since $T$ is an injective coresolution of the functor $I^{(r)}$, and the tensor product is exact with respect to each variable and preserves injectives, then $T(\otimes^d)$ is an injective coresolution of $\otimes^{d(r)}$, i.e., there is a quasi-isomorphism $\phi_d:\otimes^{d(r)}\to T(\otimes^d)$. 
	By definition, the morphism $\phi_d$ is compatible with tensor product, this means that $\phi_d\otimes\phi_e=\phi_{d+e}$. 
	Then, the complex $\otimes^{d(r)}\boxtimes T(\otimes^d)^{\boxtimes n-1}$ is a coresolution of $\otimes^{d(r)\boxtimes n}$. 
	Moreover, by Lemma \ref{lem3.3.2a}, this complex  is $\Gamma^{p^rd}\circ\boxtimes^n$-acyclic.
	Hence, the $\Ext$-group  $\Extx{\PcalKn}{\Gamma^{p^rd}\circ\boxtimes^n,(\otimes^{d(r)})^{\boxtimes n}}$ is the homology of the complex
	\begin{equation*}
	\Homx{\PcalKn}{\Gamma^{p^rd}\circ\boxtimes^n,\otimes^{d(r)}\boxtimes T(\otimes^d)^{\boxtimes n-1}}.
	\end{equation*}
	By Lemma \ref{lem3.3.2}(1), this complex  is null at odd degree. 
	We next define an isomorphism $\bar{\theta}$ such that the diagram
	\begin{equation*}
	\xymatrix{W_1
		\ar[r]^-\simeq\ar[dd]^{\bar{\theta}}& \Homx{\PcalKn}{\Gamma^{p^rd}\circ\boxtimes^n,\otimes^{d(r)}\boxtimes T(\otimes^d)^{\boxtimes n-1}}\\
		&\ar[u]^\simeq_{(\dag_1)}\Homx{\PcalKn}{\Gamma^{p^rd}\circ\boxtimes^n,\otimes^{d(r)}\boxtimes \left(E_r^{\otimes d}\otimes \left(S^{p^r}\right)^{\otimes d}\right)^{\boxtimes n-1}}\\
		W_2\ar[r]^-\simeq&\Homx{\PcalKn}{\Gamma^{p^rd}\circ\boxtimes^n,\otimes^{d(r)}\boxtimes \left(E_r^{\otimes d}\otimes^{d(r)}\right)^{\boxtimes n-1}}\ar[u]^\simeq_{(\dag_2)},
	}
	\end{equation*}
	is commutative
	where the map $(\dag_1)$ is induced by the canonical inclusion $E_r\otimes S^{p^r}\hookrightarrow T$ (see Lemma \ref{lem3.2.5}), and the map $(\dag_2)$ is induced by the inclusion $\otimes^{d(r)}\hookrightarrow \left(S^{p^r}\right)^{\otimes d}$, $W_1$ is the vector space $\Extx{\PcalKn}{\Gamma^{p^rd}\circ\boxtimes^n,(\otimes^{d(r)})^{\boxtimes n}}$, and $W_2$ is the vector space $\Homx{\PcalKn}{\Gamma^{p^rd}\circ\boxtimes^n \otimes^{d(r)\boxtimes n}}\otimes\left(E_r^{\otimes d}\right)^{\otimes n-1}$.
	The morphism $\bar{\theta}$ is compatible with products and is $\Sfrak_n$-equivariant.
	Lemma \ref{lem3.3.2} yields an isomorphism 
	\begin{equation*}
	\Homx{\PcalKn}{\Gamma^{p^rd}\circ\boxtimes^n,\otimes^{d(r)\boxtimes n}}\simeq \Homx{\PcalKn}{\Gamma^{d}\circ\boxtimes^n,\otimes^{d\boxtimes n}}.
	\end{equation*}
	We define $\theta(\otimes^d,\ldots,\otimes^d)$ such that the following diagram is commutative
	\begin{equation*}
	\xymatrix{
		\Extx{\PcalKn}{\Gamma^{p^rd}\circ\boxtimes^n,(\otimes^{d(r)})^{\boxtimes n}}\ar[r]^-{\theta(\otimes^d,\ldots,\otimes^d)}\ar[d]^{\bar{\theta}}&\Homx{\PcalKn}{\Gamma^{d,E_r^{\otimes n-1}}\circ\boxtimes^n,(\otimes^d)^{\boxtimes n}}\ar[d]^{(\dag_3)}_\simeq\\
		\Homx{\PcalKn}{\Gamma^{p^rd}\circ\boxtimes^n,\otimes^{d(r)\boxtimes n}}\otimes\left(E_r^{\otimes d}\right)^{\otimes n-1}\ar[r]^-{(\dag_4)}_\simeq&\Homx{\PcalKn}{\Gamma^{d}\circ\boxtimes^n,\otimes^{d\boxtimes n}}\otimes\left(E_r^{\otimes d}\right)^{\otimes n-1}.
	}
	\end{equation*}
	Since the morphisms  $(\dag_3),(\dag_4)$ and $\bar{\theta}$ are compatible with products and  $\Sfrak_n$-equivariant, the morphism
	$\theta(\otimes^d,\ldots,\otimes^d)$ is too.
	
	We now can prove the general case. 
	By Corollary \ref{cor3.3.1}, the functor of $\PcalK\to\VcalK$ 
	\begin{equation}\label{equ3.3.10a}
	F\mapsto \Extx{\PcalKn}{\Gamma^{p^rd}\circ\boxtimes^n,F^{(r)}\boxtimes\boxtimesx{i=2}{n}S^{\mu^i(r)}}
	\end{equation}
	is left exact.
	To define the isomorphism $\theta(S^{\mu^1},\ldots,S^{\mu^n})$ we use the isomorphism $\theta(\otimes^d,\ldots,\otimes^d)$ constructed as above, the exactness of the functor \eqref{equ3.3.10a} and the fact that the functor $S^\mu$ admits a presentation of the form $T^\mu_1\to T^\mu_0\twoheadrightarrow S^\mu$ where $T^\mu_0=\otimes^d, T^\mu_1=\oplus_{\sigma\in\Sfrak_\mu}\otimes^d$.	
	The isomorphism $\theta(S^{\mu^1},\ldots,S^{\mu^n})$ is compatible with products and the action of the group $\Sfrak_n$ because the isomorphism $\theta(\otimes^d,\ldots,\otimes^d)$ is too.
\end{proof}

\subsection{ Calculations of $\Extx{\PcalK}{\Gamma^{p^rd}\circ X,S^{\mu(r)}}$ }\label{subs3.3.4}

The goal of this section is to establish Theorems \ref{thm3.3.8} and \ref{thm3.4.1}.
Theorem \ref{thm3.4.1} is an analogue to Theorem \ref{thm3.3.1} for the functors in one variable.

Let $G\in\Pcalx{d}$ and $U$ be a $\Bbbk$-vector space of finite dimension. 
The sum-diagonal adjunction $\text{-}\circ\Delta_n:\PcalKn\rightleftarrows \PcalK:\text{-}\circ\boxplus^n$ induces a graded isomorphism, natural in $G$:
\[\beta(G):\Extx{\PcalK}{\Gamma^{d,U}\circ\otimes^n,G}\xrightarrow{\simeq} \Extx{\PcalK(n)}{\Gamma^{d,U}\circ\boxtimes^n,G\circ\boxplus^n}.\] 
The following lemma gives two properties of this isomorphism.
\begin{lem}\label{cor3.2.5}
	The isomorphism $\beta(G)$ satisfies the following properties.
	\begin{enumerate}
		\item[\rm (1)] It is compatible with products
		\begin{equation*}
		\beta(G_1)(c_1)\smile\beta(G_2)(c_2)=\beta(G_1\otimes G_2)(c_1\smile c_2).
		\end{equation*}
		\item[\rm (2)] It is $\Sfrak_n$-equivariant (with respect to actions defined in Definition {\rm \ref{def3.2.3}}).
	\end{enumerate}
\end{lem}
\begin{proof}
	By definition, the isomorphism $\beta(G)$ is the composition
	\begin{align*}
	\Extx{\PcalK}{\Gamma^{d,U}\circ\otimes^n,G}&\xrightarrow{(\dag_1)}\Extx{\PcalK(n)}{\Gamma^{d,U}\circ\otimes^n\circ\boxplus^n,G\circ\boxplus^n} \\ &\xrightarrow{(\dag_2)}  \Extx{\PcalKn}{\Gamma^{d,U}\circ\boxtimes^n,G\circ\boxplus^n}
	\end{align*}
	where the first map is induced by the precomposition with $\boxplus^n$, and the second is induced by the map
	\begin{equation*}
	\boxtimes^n(\eta): \boxtimes^n\to \boxtimes^{n}\circ \Delta_n\circ \boxplus^n=\otimes^n\circ \boxplus^n
	\end{equation*}
	where $\eta:\mathrm{Id}\to \Delta_n\circ \boxplus^n$ denotes the unit of adjunction.
	By Lemma \ref{lem3.2.1} the maps $(\dag_1)$ and $(\dag_2)$ are compatible with products. 
	To finish the proof of the lemma, it remains to verify that the composition $(\dag_2)\circ (\dag_1)$ is  $\Sfrak_n$-equivariant. 
	To get that, we define an action of $\Sfrak_n$ on $E:=\Extx{\PcalK(n)}{\Gamma^{d,U}\circ\otimes^n\circ\boxplus^n,G\circ\boxplus^n}$ and we next prove that the maps $(\dag_1)$ and $(\dag_2)$ are $\Sfrak_n$-equivariant.
	
	There are two actions of $\Sfrak_n$ on the $\Ext$-group $E$: the first is induced by the action of $\Sfrak_n$ on $\otimes^n$ and the second is induced by the precomposition with $\boxplus^n$ (see Definition \ref{def3.2.3}). 
	Since the two compositions $\otimes^n\circ\boxplus^n \xrightarrow{\sigma_\otimes(\boxtimes^n)}\otimes^n\circ\boxplus^n \xrightarrow{\otimes^n(\sigma_\boxplus)}\otimes^n\circ\boxplus^n\circ\sigma_{\mathcal{V}}$ and $\otimes^n\circ\boxplus^n\xrightarrow{\otimes^n(\sigma_\boxplus)}\otimes^n\circ\boxplus^n\circ\sigma_{\mathcal{V}}\xrightarrow{\sigma_\otimes(\boxplus^n\circ\sigma_{\mathcal{V}})}\otimes^n\circ\boxplus^n\circ\sigma_{\mathcal{V}}$  are equal, these two actions are commutative. 
	We obtain an action of $\Sfrak_n\times\Sfrak_n$ on $E$. 
	By taking the diagonal action, we deduce an action of $\Sfrak_n$ on $E$.
	
	To prove that $(\dag_1)$ is  $\Sfrak_n$-equivariant, we prove that $(\dag_1)$ is $\Sfrak_n\times\Sfrak_n$-equivariant where the second factor of $\Sfrak_n\times\Sfrak_n$ acts trivially on the source of $(\dag_1)$. 
	The compatibility of $(\dag_1)$ with the first action of $\Sfrak_n$ derives from the precomposition with $\boxplus^n$ and the compatibility with the second action of $\Sfrak_n$ is indicated in the paragraph \ref{rem3.2.4}.
	
	By the definitions \ref{def3.2.2} and \ref{def3.2.3}, to prove  the map $(\dag_2)$ is $\Sfrak_n$-equivariant, it suffices to verify that the map $\boxtimes^n(\eta)$ is compatible with the action of $\Sfrak_n$, i.e., the below diagram is commutative
	\begin{equation*}
	\xymatrix{
		\boxtimes^n\ar[rrrr]^-{\sigma_\boxtimes}\ar[d]_{\boxtimes^n(\eta)}&&&&\boxtimes^n\circ\sigma_{\mathcal{V}}\ar[d]^{\boxtimes^n(\eta)}\\
		\otimes^n\circ\boxplus^n\ar[rr]^-{\sigma_\otimes(\boxplus^n)}&&\otimes^n\circ\boxplus^n\ar[rr]^-{\otimes^n(\sigma_\boxplus)}&& \otimes^n\circ\boxplus^n\circ\sigma_{\mathcal{V}}.
	}
	\end{equation*}
	Since the transpositions form a system of generators of the symmetric group $\Sfrak_n$, we can assume that $\sigma$ is a transposition. In this case, we can assume more that $n=2$. We make concrete calculations: $v_1\otimes v_2\overset{\sigma_\boxtimes}{\longmapsto} v_2\otimes v_1\overset{\boxtimes^2(\eta)}{\longmapsto}(v_2,0)\otimes (0,v_1) $ and $v_1\otimes v_2\overset{\boxtimes^2(\eta)}{\longmapsto}(v_1,0)\otimes(0,v_2)\overset{\sigma_\otimes(\boxplus^2)}{\longmapsto}(0,v_2)\otimes (v_1,0)\overset{\otimes^2(\sigma_\boxplus)}{\longmapsto}(v_2,0)\otimes (0,v_1)$. Then, the above diagram is commutative.
\end{proof}

\begin{rem}
	In the definition of $\beta(G)$, we can replace $\Gamma^d$ by $\Lambda^d$ or $S^d$, or more generally by a family of strict polynomial functors $C^d$ equipped with a  structure of coalgebra. 
	In this case, the isomorphism
	\begin{equation*}
	\Extx{\PcalK}{C^{d,U}\circ\otimes^n,G}=\Extx{\PcalK}{C^{d,U}\circ\boxtimes^n\circ\Delta_n,G}\simeq \Extx{\PcalKn}{C^{d,U}\circ\boxtimes^n,G\circ\boxplus^n}
	\end{equation*}
	verifies all properties of Lemma \ref{cor3.2.5} (the proof is identical). 
	However, we only use the case $C^d=\Gamma^d$ in the sequel. 
\end{rem}

\begin{proof}[\rm \bfseries Proof of Theorem \ref{thm3.4.1}]
	Lemma \ref{cor3.2.5} yields an isomorphism 
	\begin{equation*}
	\beta(S^{\mu(r)}):\Extx{\PcalK}{\Gamma^{p^rd}\circ\otimes^n,S^{\mu(r)}}\xrightarrow{\simeq} \Extx{\PcalK}{\Gamma^{p^rd}\circ\boxtimes^n,S^{\mu(r)}\circ\boxplus^n}
	\end{equation*}
	which is compatible with the products and  $\Sfrak_n$-equivariant. The exponential isomorphism  $S^d(V\oplus W)\simeq\bigoplus_{i=0}^dS^i(V)\otimes S^{d-i}(W)$ induces an isomorphism $S^{\mu}\circ \boxplus^n\simeq \bigoplus_{\sum\mu^i=\mu}\boxtimesx{i=1}{n}S^{\mu^i}$ which is compatible with the products. Moreover, this isomorphism is compatible with the action of $\Sfrak_n$, i.e., we have a commutative diagram
	\begin{equation*}
	\xymatrix{
		S^{\mu}\circ\boxplus^n\ar[rr]^-\simeq\ar[d]^{S^\mu(\sigma_\boxplus)}&&\bigoplus\limits_{\sum\mu^i=\mu}\boxtimesx{i=1}{n}S^{\mu^i}\ar[d]\\
		S^{\mu}\circ\boxplus^n\circ\sigma_{\mathcal{V}}\ar[rr]^-\simeq&&\bigoplus\limits_{\sum\mu^i=\mu}\left(\boxtimesx{i=1}{n}S^{\mu^i}\right)\circ\sigma_{\mathcal{V}}
	}
	\end{equation*}
	where $\sigma$ is an element of $\Sfrak_n$ and the vertical morphism on the right is induced by the isomorphisms $\sigma_{(S^{\mu^1},\ldots,S^{\mu^n})}:\boxtimesx{i=1}{n}S^{\mu^i}\to \left(\boxtimesx{i=1}{n}S^{\mu^{\sigma^{-1}(i)}}\right)\circ\sigma_{\mathcal{V}}$. By the isomorphism induced by the exponential isomorphism as above and Theorem \ref{thm3.3.1}, we have isomorphisms 
	\begin{align*}
	\Extx{\PcalKn}{\Gamma^{p^rd}\circ\boxtimes^n,S^{\mu(r)}\circ\boxplus^n}&\simeq \bigoplus_{\sum\mu^i=\mu}\Extx{\PcalKn}{\Gamma^{p^rd}\circ\boxtimes^n,\boxtimesx{i=1}{n}S^{\mu^i(r)}}\\
	&\simeq \bigoplus_{\sum\mu^i=\mu}\Homx{\PcalKn}{\Gamma^{d,E_r^{\otimes n-1}}\circ\boxtimes^n,\boxtimesx{i=1}{n}S^{\mu^i}}\\
	&\simeq \Homx{\PcalKn}{\Gamma^{d,E_r^{\otimes n-1}}\circ\boxtimes^n,S^{\mu}\circ\boxplus^n}.
	\end{align*}
	These isomorphisms are compatible with the products and is $\Sfrak_n$-equivariant. Denote by $\chi(S^\mu)$ the composition. Moreover, by Lemma \ref{cor3.2.5}, we have an isomorphism
	\begin{equation*}
	\beta(S^\mu):\Homx{\PcalK}{\Gamma^{d,E_r^{\otimes n-1}}\circ\otimes^n,S^{\mu}}\simeq \Homx{\PcalKn}{\Gamma^{d,E_r^{\otimes n-1}}\circ\boxtimes^n,S^{\mu}\circ\boxplus^n}.
	\end{equation*}
	We define $\alpha(S^\mu)$ to be the composition $\beta(S^\mu)^{-1}\circ\chi(S^\mu)\circ\beta(S^{\mu(r)})$. It is compatible with the products and is $\Sfrak_n$-equivariant. 
	We obtain the desired isomorphism.
\end{proof}

\begin{proof}[\rm \bfseries Proof of Theorem \ref{thm3.3.8}]
	By Theorem \ref{thm3.4.1}, we have an isomorphism
	\begin{equation*}
	\alpha(S^\mu):\Extx{\PcalK}{\Gamma^{p^rd}\circ\otimes^n,S^{\mu(r)}}\xrightarrow{\simeq} \Homx{\PcalK}{\Gamma^{d,E_r^{\otimes n-1}}\circ\otimes^n,S^\mu}.
	\end{equation*}
	This isomorphism is compatible with the products and is $\Sfrak_n$-equivariant. Moreover, we have the isomorphism $\Homx{\PcalK}{\otimes^n,\otimes^n}\simeq\Bbbk\Sfrak_n$, then $\beta(S^\mu)$ is natural in $\otimes^n$. 
	Moreover, since $X$ is a direct factor of $\otimes^n$, there exists an idempotent $f$ of $\otimes^n$ whose image is $X$. 
	The following diagram is commutative
	\begin{equation}\label{equ3.3.6}
	\xymatrix{
		\Extx{\PcalK}{\Gamma^{p^rd}\circ \otimes^n,S^{\mu(r)}}\ar[rr]^-{\alpha(S^\mu)}\ar[d]_{\Extx{\PcalK}{\Gamma^{p^rd}(f),S^{\mu(r)}}} &&\Homx{\PcalK}{\Gamma^{d,E_r^{\otimes n-1}}\circ \otimes^n,S^\mu} \ar[d]^{\Homx{\PcalK}{\Gamma^{d,E_r^{\otimes n-1}}(f),S^\mu}}\\
		\Extx{\PcalK}{\Gamma^{p^rd}\circ \otimes^n,S^{\mu(r)}}\ar[rr]^-{\alpha(S^\mu)}&&\Homx{\PcalK}{\Gamma^{d,E_r^{\otimes n-1}}\circ \otimes^n,S^\mu}.
	}
	\end{equation}	
	Then the images of the vertical maps are isomorphic. 
	Since $f$ is an idempotent of $\otimes^n$ whose image is $X$, then $\Extx{\PcalK}{\Gamma^{p^rd}(f),S^{\mu(r)}}$
	is an idempotent of $\Extx{\PcalK}{\Gamma^{p^rd}\circ \otimes^n,S^{\mu(r)}}$ whose image is $\Extx{\PcalK}{\Gamma^{p^rd}\circ X,S^{\mu(r)}}$ and the map
	$\Homx{\PcalK}{\Gamma^{d,E_r^{\otimes n-1}}(f),S^{\mu}}$	
	is an idempotent of $\Homx{\PcalK}{\Gamma^{d,E_r^{\otimes n-1}}\circ \otimes^n,S^{\mu}}$ whose image is
	$\Homx{\PcalK}{\Gamma^{d,E_r^{\otimes n-1}}\circ X,S^{\mu}}$. 
	Hence, we obtain an isomorphism $\alpha(S^\mu,X)$ from $\Extx{\PcalK}{\Gamma^{p^rd}\circ X,S^{\mu(r)}}$ to $ \Homx{\PcalK}{\Gamma^{d,E_r^{\otimes n-1}}\circ X,S^\mu}$. Moreover, since the maps in the diagram \eqref{equ3.3.6} are compatible with the products, the morphism is also compatible with the products.
\end{proof}

\begin{cor}\label{cor3.3.9}
	Let $\mu$ be a tuple of natural numbers of weight $ nd $. Let $L$ be a simple functor in the category $\Pcalx{n}$ of the homogeneous strict polynomial functors of degree $n$. 
	We assume that the characteristic $p$ of the field $\Bbbk$ is strictly greater than $n$. 
	Then there is a graded isomorphism, natural in $S^\mu$
	\begin{equation*}
	\alpha(S^\mu,L):\Extx{\PcalK}{\Gamma^{p^rd}\circ L,S^{\mu(r)}}\xrightarrow{\simeq} \Homx{\PcalK}{\Gamma^{d,E_r^{\otimes n-1}}\circ L,S^\mu}
	\end{equation*}
	which is compatible with the products.
\end{cor}
\begin{proof}
	Since $p>n$, the simple functor $L$ is a direct factor of the tensorial product $\otimes^n$. 
	In fact, according to Friedlander-Suslin \cite[Theorem 3.2]{FS97}, the evaluation on $\Bbbk^n$ yields an equivalence of categories $\Pcalx{n}\to S(n,n)\Modp$ where $S(n,n)$ is the Schur algebra and $S(n,n)\Modp$ is the category of the representations of finite dimension of $S(n,n)$. 
	Moreover, since $p>n$, by Schur \cite{Schur73I,Schur73III} (see also \cite[Corollary 2.6e]{Gre07}), the algebra $S(n,n)$ is semi-simple. 
	Then the simple module $L(\Bbbk^n)$ is a direct factor of $\left(\Bbbk^{n}\right)^{\otimes n}$. 
	Therefore, we deduce that $L$ is a direct factor of $\otimes^n$. 
	By applying Theorem \ref{thm3.3.8}, we obtain the desired conclusion.
\end{proof}

\section{ Cohomology of the orthogonal and symplectic groups}\label{sec-cohom}

In this section, assume that the characteristic $p$ of the field $\Bbbk$ is \emph{odd}. 
Our purpose is to calculate explicitly examples of cohomology algebra for the symplectic group scheme $\Spoper_n$ and orthogonal group scheme $\Ooper_{n,n}$. 
More precisely, if $G_{n}$ is either $\Spoper_{n}$ or $\Ooper_{n,n}$, then $G_n$ is by definition a  subgroup scheme of $\GL_{2n}$, and acts naturally on the vector space $\Bbbk^{2n}$ (by matrix multiplication). 
By applying the $r$-th Frobenius twist (for $r\ge 0$), and by taking $\ell$ copies of the representation obtained, we will obtain an action of $G_n$ on $\Bbbk^{2n(r)\oplus \ell}$, subsequently an action by algebraic automorphisms on the $\Bbbk$-algebra $\Bbbk\left[\Bbbk^{2n(r)\oplus \ell}\right]$ of the polynomials on $\Bbbk^{2n(r)\oplus \ell}$. 
We remark that this algebra of polynomials can also be written  (as $G_n$-module) under the form $S^{*}\left(\Bbbk^{2n\vee(r)\oplus \ell}\right)$ where $\Bbbk^{2n\vee}$ denotes the dual representation of $\Bbbk^{2n}$. In this section, we compute the following cohomology algebra with $n\ge p^{r}\ell$:
\begin{equation}\label{equ3.5.1}
H^{*}_{\mathrm{rat}}\left(G_{n},S^{*}\left(\Bbbk^{2n\vee(r)\oplus \ell}\right)\right)\;.
\end{equation}

The cohomology subalgebras of degree zero, do not depend on $r$ (because applying the Frobenius twist does not change the invariants). 
For $r=0$, these algebras of invariants were described by the fundamental theorems of classical invariant theory. 
Our contribution is to describe the cohomology of higher degree, with an algebraic structure.

\begin{thm}\label{thm3.3.11}
	Let $n\ge p^{r}\ell$. 
	Suppose that the characteristic $p$ of the field $\Bbbk$ is odd. 
	The cohomology algebra
	$ H^{*}_{\mathrm{rat}}\left(G_{n},S^{*}\left(\Bbbk^{2n\vee(r)\oplus \ell}\right)\right)$ is a symmetric algebra on a set of generators $(h|i|j)_{G_{n}}\in H^{2h}_{\mathrm{rat}}\left(G_{n},S^{*}\left(\Bbbk^{2n\vee(r)\oplus \ell}\right)\right)$  where $0\le h<p^r,0\le i\le j\le \ell$ and $i\ne j$ if $G_n$ is the symplectic group $\Spoper_n$.
	
	Moreover, there are no relations among the $(h|i|j)_{G_{n}}$.
\end{thm}

A similar theorem to \ref{thm3.3.1} for general linear group scheme $\GL_n$ was proved by Touz\'e \cite[Theorem 6.15]{Tou12}.

We now recall the first fundamental theorems for the symplectic and orthogonal group schemes. 
To deal with the two cases simultaneously, we introduce some notations.
Denote by $e_1^\vee,e_2^\vee,\ldots,e_{2n}^\vee$ the dual basis of the canonical basis of $\Bbbk^{2n}$. For each group $G_n$ we define a functor $X_G$, a particular element  $\omega_{G_n}\in X_G(\Bbbk^{2n})$ and a set of pairs of integers $\mathtt I_{G_{n}}$ as follows.

\begin{table}[!ht]
	\centering
	\renewcommand{\arraystretch}{1.25}
	\begin{tabular}{|c|c|c|}
		\hline
		$G_{n}$& $\Spoper_{n}$ &  $\Ooper_{n,n}$\\ 
		\hline
		$X_{G}$& $\Lambda^{2}$ & $S^{2}$ \\ 
		\hline
		$\omega_{G_{n}}$ &  $\sum_{i=1}^{n}e_{i}^{\vee}\wedge e_{n+i}^{\vee}$ &$\sum_{i=1}^{n}e_{i}^{\vee}e_{n+i}^{\vee}$\\ 
		\hline
		$\mathtt I_{G_{n}}$& $\{(i,j):1\le i<j\le \ell\}$ &  $\{(i,j):1\le i \le  j\le \ell\}$\\ 
		\hline
	\end{tabular}
\end{table}

\noindent We define invariants $(i|j)_{G_n}\in S^{2}\left(\Bbbk^{2n\vee\oplus \ell}\right)$ with $(i,j)\in\mathtt{I}_{G_n}$ under the action of $G_n$ by
\begin{equation*}
(i|j)_{G_{n}}(x_{1},...,x_{\ell})=\omega_{G_{n}}(x_{i},x_{j}).
\end{equation*}

\begin{thm}[Fundamental theorem for the groups $\Spoper_n$ and $\Ooper_{n,n}$ \cite{deCP76}] The set $\left\{(i|j)_{G_{n}}:(i,j)\in\mathtt I_{G_{n}}\right\}$ is a system of generators of the algebra $H^{0}_{\mathrm{rat}}\left(G_{n},S^{*}\left(\Bbbk^{2n\vee\oplus \ell}\right)\right)$.
	Moreover, if $n\ge \ell$, there are no relations among the $(i|j)_{G_{n}}.$
\end{thm}

\subsection{ Relation with $\Ext$-group in $\PcalK$}
We first recall the relation between the cohomology of the symplectic and orthogonal groups with $\Ext$-groups in $\PcalK$.
The standard form $\omega_G\in X_G(\Bbbk^{2n})$ is  invariant under the action of $G_n$. 
Then we have a $G_n$-equivariant map $\iota_d:\Bbbk\to \Gamma^{d}\left(X_{G}\left(\Bbbk^{2n\vee}\right)\right)$ which sends $\lambda$ to $\lambda\omega_{G_{n}}^{\otimes d}$.
\begin{deff}
	Let $F\in\PcalK$.  We define a morphism $\phi_{G_n,F}$ by the following commutative diagram
	\begin{equation*}
	\xymatrix{
		\Extx{\PcalK}{\Gamma^{p^rd}\circ X_G,F}\ar[rr]^-{\mathrm{ev}_{\Bbbk^{2n\vee}}}\ar[dd]^{\phi_{G_n,F}}&&\Extx{\GL_n}{\Gamma^{p^rd}\circ X_G(\Bbbk^{2n\vee}),F(\Bbbk^{2n\vee})}\ar[d]^-{\mathrm{res}^{\GL_n}_{G_n}}\\
		&&\Extx{G_n}{\Gamma^{p^rd}\circ X_G(\Bbbk^{2n\vee}),F(\Bbbk^{2n\vee})}\ar[d]^-{\iota_{dp^r}^*}\\
		H^*_\mathrm{rat}\left(G_n,F(\Bbbk^{2n\vee})\right)\ar[rr]^-\simeq&&\Extx{G_n}{\Bbbk,F(\Bbbk^{2n\vee})}.
	}
	\end{equation*}
\end{deff}

The following key result is due to Touz\'e \cite[Subsections 3.2, 3.3]{Tou10a}.
\begin{thm}\label{cle0}  Let $F\in\Pcalx{d}$. The morphism
	\begin{equation*}
	\phi_{G_n,F}:\Extx{\PcalK}{\Gamma^{p^rd}\circ X_G,F}\to H^*_\mathrm{rat}\left(G_n,F(\Bbbk^{2n\vee})\right)
	\end{equation*}
	satisfies the following properties.
	\begin{enumerate}
		\item[\rm (1)] $\phi_{G_{n},F}$ is a graded map natural in $F$ and is compatible with the products.
		\item[\rm (2)]  $\phi_{G_{n},\text{-}}$ is a  morphism of universal $\delta$-functors.
		\item[\rm (3)] $\phi^{0}_{G_{n},F}$ is an epimorphism if $F$ is an injective.
		\item[\rm (4)] $\phi_{G_{n},F}$ is an isomorphism if $2n\ge \deg F$.
	\end{enumerate}
\end{thm}

The cohomology algebra \eqref{equ3.5.1} can be decomposed as
\begin{align*}
H^{*}_{\mathrm{rat}}\left(G_{n},S^{*}\left(\Bbbk^{2n\vee(r)\oplus \ell}\right)\right)
&=\bigoplus_{d\ge 0} H^{*}_{\mathrm{rat}}\left(G_{n},S^{d}\left(\Bbbk^{2n\vee(r)\oplus \ell}\right)\right)\\
&=\bigoplus_{d\ge 0} H^{*}_{\mathrm{rat}}\left(G_{n},S^{d(r)}_{\Bbbk^\ell}\left(\Bbbk^{2n\vee}\right)\right).
\end{align*}
By Theorem \ref{cle0}(1), there is a morphism of  $\Bbbk$-graded algebras:
\begin{equation}\label{equ3.5.3}
\Phi_{G_n}=\bigoplus_{d\ge 0}\phi_{G_n,S^{2d(r)}_{\Bbbk^\ell}}:\bigoplus_{d\ge 0}\Ext^{*}_{\PcalK} \left(\Gamma^{p^rd}\circ X_{G},S^{2d(r)}_{\Bbbk^\ell}\right) \to H^{*}_{\mathrm{rat}}\left(G_{n},S^{*}\left(\Bbbk^{2n\vee(r)\oplus \ell}\right)\right). 
\end{equation}

\begin{lem}\label{lem3.4.4}
	There is an algebra isomorphism
	\begin{equation*}
	\bigoplus_{d\ge 0}\Ext^{*}_{\PcalK} \left(\Gamma^{p^rd}\circ X_{G},{S^{2d(r)}_{\Bbbk^\ell}}\right)\simeq S^{*}\left(E_{r}\otimes X_G(\Bbbk^{\ell})\right).
	\end{equation*}
\end{lem}
\begin{proof}
	Since $p>2$ and $X_G\in\{S^2,\Lambda^2\}$, $X_G$ is a simple functor in the category $\Pcalx{2}$. By Corollary \ref{cor3.3.9}, there is an isomorphism
	\begin{equation*}
	\Extx{\PcalK}{\Gamma^{p^rd}\circ X_G,S^{2d(r)}_{\Bbbk^\ell}}\simeq \Homx{\PcalK}{\Gamma^{d,E_r}\circ X_G,S^{2d}_{\Bbbk^\ell}}.
	\end{equation*}
	Moreover, Yoneda's lemma yields isomorphisms
	\begin{equation*}
	\Homx{\PcalK}{\Gamma^{d,E_r}\circ X_G,S^{2d}_{\Bbbk^\ell}}\simeq S^d\left(E_r\otimes X_G^\sharp(\Bbbk^\ell)\right)\simeq S^d\left(E_r\otimes X_G(\Bbbk^\ell)\right).
	\end{equation*}
	Furthermore, these isomorphisms are compatible with the products. 
	Then we obtain an isomorphism of graded algebras $\bigoplus_{d\ge 0}\Ext^{*}_{\PcalK} \left(\Gamma^{p^rd}\circ X_{G},S^{2d(r)}_{\Bbbk^\ell}\right)\simeq S^{*}\left(E_{r}\otimes X_G(\Bbbk^{\ell})\right)$.
\end{proof}

\subsection{ $n$-coresolved functors \`a la Touz\'e}\label{subsect3.5.2}

In order to compute the cohomology algebras \eqref{equ3.5.1} of the symplectic and orthogonal groups, we need calculate the groups of cohomology of the form $H^*_{\mathrm{rat}}\left(G_n,F(\Bbbk^{2n\vee})\right)$ for $2n<\deg F$, i.e., these values of $n$ are different from the ones that Theorem \ref{cle0} gives an isomorphism:
\begin{equation*}
\phi_{G_n,F}:\Extx{\PcalK}{\Gamma^{p^rd}\circ X_G,F}\xrightarrow{\simeq} H^*_{\mathrm{rat}}\left(G_n,F(\Bbbk^{2n\vee})\right).
\end{equation*}
However, we will prove that for the bad values of $ n $, i.e., $ 2n < \deg (F) $, the morphism $\phi_{G_n,F}$ remains an isomorphism for the functors $F$ considered.
For this, we use the notion of functor $n$-coresolved introduced in \cite{Tou12}.
Let $F\in\Pcalx{d}$ and $n$ be a positive integer.
We recall that the morphism $\theta_{F}=\theta_{F,\Bbbk^n}:\Gamma^{d,\Bbbk^n}\otimes F\left(\Bbbk^{n}\right)\to F$ defined for all $V\in\VcalK$ by
\begin{equation*}
\Gamma^{d}\left(\Hom\left(\Bbbk^{n},V\right)\right)\otimes F\left(\Bbbk^{n}\right)\to F(V),\quad f\otimes x\mapsto F(f)(x).
\end{equation*}
Dually, we have a morphism $ \left(\theta_{F^{\sharp}}\right)^{\sharp}:F\to S^{d}_{\Bbbk^n}\otimes F\left(\Bbbk^{n\vee}\right) $.
\begin{deff}\label{def3.4.5}
	A functor $F\in \Pcalx{d}$ is  \emph{$n$-cogenerated} if $\theta_{F^\sharp}$ is an epimorphism, or equivalently, if $\left(\theta_{F^\sharp}\right)^{\sharp}$ is a monomorphism.
\end{deff}

\begin{lem}[\text{\cite[Lemma 6.8]{Tou12}}]\label{cle1} Let $J\in\PcalKoper{d}$ be a $n$-cogenerated injective.
	\begin{enumerate}
		\item[\rm (1)] $J(\Bbbk^{n})$ is injective in the category $S(n,d)${\rm -mod}.
		\item[\rm (2)] For $F\in \PcalKoper{d}$, the evaluation map induces an isomorphism
		\begin{equation*}
		\Hom_{\PcalK}(F,J)\simeq \Hom_{S(n,d)}\left(F(\Bbbk^{n}),J(\Bbbk^{n})\right).
		\end{equation*}
	\end{enumerate}
\end{lem}

\begin{lem}\label{cle1.5}
	Let $J\in \PcalKoper{2d}$ be a  $2n$-cogenerated injective. 
	Then $\phi^{0}_{G_{n},J}$ is an isomorphism.
\end{lem}
\begin{proof}
	{By Theorem} \ref{cle0}(3), it suffices to show that $\phi^{0}_{G_{n},J}$ is a monomorphism.
	Denote by $\tilde \theta$ the composition
	\begin{align*}
	\Gamma^{2d,\Bbbk^{n\vee}\oplus\Bbbk^{n\vee}}\twoheadrightarrow &\Gamma^{d}\left(\Hom(\Bbbk^{n\vee},-)\right)\otimes \Gamma^{d}\left(\Hom(\Bbbk^{n\vee},-)\right)\\
	\xrightarrow{\simeq}&\Gamma^{d}\left(\Hom(-^{\vee},\Bbbk^{n\vee})\right)\otimes \Gamma^{d}\left(\Hom(\Bbbk^{n\vee},-)\right)\\
	\to&\Gamma^{d}\left(\Hom(-^{\vee},-)\right)\xrightarrow{\simeq}\Gamma^{d}\circ\otimes^{2}.
	\end{align*}
	Since $p$ is odd, $\otimes^{2}=\Lambda^{2}\oplus S^{2}$.
	We denote by $\tilde\theta_{G_{n}}$ the composition $\Gamma^{2d,\Bbbk^{n\vee}\oplus\Bbbk^{n\vee}}\to \Gamma^{d}\circ \otimes^{2}\to \Gamma^{d}\circ X_{G}.$
	By definition, $\tilde \theta_{\Bbbk^{2n\vee}}$ is an epimorphism. Then the $\left(\tilde \theta_{G_{n}}\right)_{\Bbbk^{2n\vee}}$ are epimorphisms.
	By Lemma \ref{cle1}, we have a commutative diagram
	\begin{equation*}
	\xymatrix{   \Hom\left(\Gamma^{d}\circ X_{G},J\right) \ar[rrr]^-{\Hom\left(\tilde\theta_{G_{n}},J\right)}\ar[d]^{\simeq}&&& \Hom\left(\Gamma^{2d,\Bbbk^{n}\oplus\Bbbk^{n}},J\right)   \ar[d]^{\simeq}\\
		\Hom\left(\Gamma^{d}\circ X_{G} (\Bbbk^{2n\vee}),J(\Bbbk^{2n\vee})\right) \ar[rrr]^-{\Hom\left(\left(\tilde\theta_{G_{n}}\right)_{\Bbbk^{2n\vee}},J(\Bbbk^{2n\vee})\right)}&&&\Hom\left(\Gamma^{2d,\Bbbk^{n}\oplus\Bbbk^{n}}(\Bbbk^{2n\vee}),J(\Bbbk^{2n\vee})\right).
	}
	\end{equation*}
	Then $\Hom\left(\tilde\theta_{G_{n}},J\right)$ is a monomorphism. Moreover, the following diagram commutes
	\begin{equation*}
	\xymatrix{  \Hom\left(\Gamma^{2d},\Bbbk^{n\vee}\oplus \Bbbk^{n\vee},J\right)  \ar[rrr]^{\simeq}_{\rm Yoneda}&&&  J\left(\Bbbk^{2n\vee}\right) \\
		\Hom\left(\Gamma^{d}\circ X_{G},J\right)  \ar[rrr]^{\phi^{0}_{G_{n},J}}\ar[u]^{\Hom\left(\left(\tilde\theta\right)_{G_{n}},J\right)}&&&H^{0}_{\mathrm{rat}}\left(G_{n},J\left(\Bbbk^{2n\vee}\right)\right).\ar[u]
	}
	\end{equation*}
	Then we obtain the conclusion.
\end{proof}

\begin{deff} A functor $F\in\Pcalx{d}$ is called \emph{$n$-coresolved} if there is an injective coresolution  $J_F^\bullet$ of $F$ such that the injective functors $J_F^{i}$ are $n$-cogenerated.
\end{deff}
The Troesch coresolutions \cite{Tro05b} of the functors $S^{d(r)}$ are $p^r$-coresolutions. 
Thus we deduce the following result.
\begin{lem}[\text{\cite[Proposition 6.11]{Tou12}}]\label{lem3.5.8}
	The functors $S^{d(r)}_{\Bbbk^\ell}$ are $p^{r}\ell$-coresolved.
\end{lem}

\begin{thm}\label{cle2}
	Let $F\in\PcalKoper{d}$ be $2n$-coresolved. Then $\phi_{G_{n},F}$ is an isomorphism.  
\end{thm}
\begin{proof}
	Since the functor $F\in\PcalKoper{d}$ is $2n$-coresolved, by definition, there is an injective coresolution $J^\bullet$ of $F$ such that the injective functors $J^{i}$ are $2n$-cogenerated. 
	By Lemma \ref{cle1}, the evaluation on $\Bbbk^{2n}$ induces an injective coresolution  $J^\bullet(\Bbbk^{2n})$ of $F(\Bbbk^{2n})$. Moreover, by Lemma \ref{cle1.5}, we have an isomorphism of the complexes $\phi^0_{G_n,J^\bullet}:\Homx{\PcalK}{\Gamma^d\circ X_G,J^\bullet}\to \Homx{G_n}{\Bbbk,J^\bullet(\Bbbk^{2n})}$. 
	Taking the homology, we obtain the result.
\end{proof}

\subsection{ Proof of Theorem \ref{thm3.3.11}}\label{subs3.3.5}

By \eqref{equ3.5.3} and Lemma \ref{lem3.4.4}, there is a morphism of algebras
\begin{equation}\label{equ3.4.5}
\Phi_{G_n}:S^{*}\left(E_{r}\otimes X_G(\Bbbk^{\ell})\right)\to  H^{*}_{\mathrm{rat}}\left(G_{n},S^{*}\left(\Bbbk^{2n\vee(r)\oplus \ell}\right)\right).
\end{equation}
Moreover, this morphism is an isomorphism if the following morphisms are isomorphisms for all $d\in\Nbb$:
\begin{equation*}
\phi_{G_n,S^{2d(r)}_{\Bbbk^\ell}}:\Extx{\PcalK}{\Gamma^{p^rd}\circ X_G,S^{2d(r)}_{\Bbbk^\ell}}\to H^*_\mathrm{rat}\left(G_n,S^{2d(r)}_{\Bbbk^\ell}(\Bbbk^{2n\vee})\right).
\end{equation*}
By Lemma \ref{lem3.5.8}, $S^{2d(r)}_{\Bbbk^\ell}$ is $2p^r\ell$-coresolved. 
By the hypothesis $n\ge p^r\ell$ and Theorem \ref{cle2},  the morphisms $\phi_{G_n,S^{2d(r)}_{\Bbbk^\ell}}$ are isomorphisms. 
Then, the morphism of algebras $\Phi_{G_n}$ is an isomorphism. 

Since $E_r$ is a graded vector space given by $\left(E_r\right)^i=\Bbbk$ if $i=0,2,\ldots 2p^r-2$ and $0$ otherwise, $E_r$ has a basis $\epsilon_h,h=0,1,\ldots,p^r-1$ with $\epsilon_h\in \left(E_r\right)^{2h}$. 
Denote by $e_i, i=1,2,\ldots,\ell$ a basis of $\Bbbk^\ell$. 
Then $\epsilon_h\otimes (e_i\wedge e_j)$ with $0\le h<p^r$ and $1\le i<j\le n$ is a basis of the algebra $S^*\left(E_r\otimes\Lambda^2(\Bbbk^\ell)\right)$, and $\epsilon_h\otimes (e_ie_j)$ with $0\le h<p^r$ and $1\le i\le j\le n$ is a basis of algebra $S^*\left(E_r\otimes S^2(\Bbbk^\ell)\right)$.	
For $0\le h<p^{r}, (i,j)\in\mathtt I_{G_{n}}$ we define a cohomological class $(h|i|j)_{G_n}\in H_{\mathrm{rat}}^{2h}\left(G_n,S^{*}\left(\Bbbk^{2n\vee(r)\oplus \ell}\right)\right) $ by
\begin{equation*}
(h|i|j)_{G_{n}}= \begin{cases}
\Phi_{\Spoper_{n}}(\epsilon_{h}\otimes (e_{i}\wedge e_{j}))& G_{n}=\Spoper_{n},\\
\Phi_{\Ooper_{n,n}}(\epsilon_{h}\otimes (e_{i} e_{j}))& G_{n}=\Ooper_{n,n}.
\end{cases}
\end{equation*}
Since the morphism \eqref{equ3.4.5} is an isomorphism, the elements $(h|i|j)_{G_{n}}$ with $0\le h<p^r$ and $(i,j)\in\mathtt{I}_{G_n}$ is a basis of algebra $H^{*}_{\mathrm{rat}}\left(G_{n},S^{*}\left(\Bbbk^{2n\vee(n)\oplus \ell}\right)\right)$. 
This completes the proof of Theorem \ref{thm3.3.11}.

\subsection*{Acknowledgments} 
This is part of the author's Ph.D. thesis \cite{PhamVT15}, written under the supervision of Lionel Schwartz and Antoine Touz\'e at the Universit\'e
Paris 13. 
The author is greatly indebted to Antoine Touz\'e for suggesting the problem and for many stimulating conversations.
I would like to thank
the two referees for carefully reading our manuscript and for giving such  valuable comments which substantially helped improving the quality of the paper.

\end{document}